\documentclass[10pt,fleqn]{article}
\usepackage[paperwidth=192mm,paperheight=262mm,left=13.5mm,right=13.5mm,
 top=19mm,bottom=18mm,headheight=12pt,headsep=7pt,footskip=18pt]{geometry}
\usepackage[T1]{fontenc}
\usepackage{amsmath,amssymb,amsfonts,mathtools,amsthm}
\usepackage{txfonts}
\usepackage{booktabs,array,enumitem,microtype,needspace,float,mathrsfs}
\usepackage[numbers,sort&compress]{natbib}
\usepackage{etoolbox,xcolor,url,fancyhdr}
\usepackage[hyperfootnotes=false,hypertexnames=false,hidelinks]{hyperref}
\hypersetup{bookmarksnumbered=true,bookmarksdepth=2,
 pdfauthor={Yangcheng Li},
 pdftitle={Descent, Seminormalization, and Conductors of Sparse Profile Images}}
\numberwithin{equation}{section}
\AtBeginEnvironment{theorem}{\Needspace{6\baselineskip}}
\AtBeginEnvironment{proposition}{\Needspace{5\baselineskip}}
\AtBeginEnvironment{corollary}{\Needspace{5\baselineskip}}
\AtBeginEnvironment{lemma}{\Needspace{5\baselineskip}}
\AtBeginEnvironment{example}{\Needspace{5\baselineskip}}
\newtheorem{theorem}{Theorem}[section]
\newtheorem{proposition}[theorem]{Proposition}
\newtheorem{lemma}[theorem]{Lemma}
\newtheorem{corollary}[theorem]{Corollary}

\newtheorem{remark}[theorem]{Remark}
\newtheorem{example}[theorem]{Example}

\newcommand{\PP}{\mathbb P}
\newcommand{\Gm}{\mathbb G_{\mathrm m}}

\newcommand{\Spec}{\operatorname{Spec}}
\newcommand{\rank}{\operatorname{rank}}
\newcommand{\Hilb}{\operatorname{Hilb}}

\newcommand{\Span}{\operatorname{span}}
\newcommand{\sym}{\operatorname{Sym}}

\newcommand{\Res}{\operatorname{Res}}

\begin{document}
\title{Descent, Seminormalization, and Conductors of Sparse Profile Images}
\author{Yangcheng Li\thanks{Corresponding author. Email: \texttt{liyc@m.scnu.edu.cn}.}\\[3pt]
\small School of Mathematical Sciences, South China Normal University\\
\small Guangzhou 510631, Guangdong, China}
\date{}
\maketitle
\begin{abstract}
We develop a descent theory for sparse factorization-profile images, asking
which functions on a normalization descend to the actual image. Whole-fiber
completions yield a Fourier-zero
classification for homogeneous masks, with a sharp stable converse, and a
Fourier-line classification for mixed aperiodic profiles in the stated marked
stable ranges. For homogeneous profiles, bounded integer relations determine
geometric fibers and seminormalization, while explicit semigroup models compute
actual image algebras, conductors, and depth in substantial families. For
singleton--complement profiles, exact one-jet conditions determine the
finite-output algebra, conductor, and projective descent through the first
nonnormal boundary; the local type is computed in the smooth-boundary range and
at the total-zero point of the first nonnormal layer. We also prove a saturated
root-of-unity specialization theorem and give separating examples showing that
normalization, natural polarizations, point fibers, seminormalization, and even
the source conductor need not determine the actual image.
\end{abstract}
\noindent\textit{Keywords:} sparse factorization; Fourier spectra; descent;
seminormalization; conductor; Krylov strata.

\noindent\textit{2020 MSC:} 14M05, 13F45, 13B22, 14B05.
\section{Introduction and main results}\label{sec:introduction}

This paper studies the descent problem for sparse factorization-profile
images.  The companion reconstruction paper \cite{LiProfileGeometry}
constructs the reduced profile branches and, for complete progressions,
identifies the normalizations of their images as products of projective
spaces with explicit polarizations.  Here that normalization is taken as
known.  The question is no longer which normal variety maps to the image, but
which functions on that normal variety actually descend to the image.

This distinction is essential.  A finite normalization can identify points,
fail to be immersive on individual branches, or carry additional
infinitesimal gluing invisible from the set-theoretic fibers.  Accordingly we
work with completed images over the entire normalization fiber rather than a
single selected branch.  The resulting embedded completed subalgebra records
both the reduced identifications and the infinitesimal conditions that the
actual coefficient functions satisfy.  The analysis passes through four layers of structure
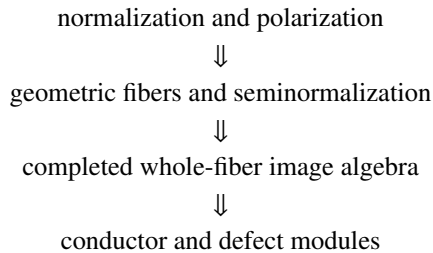
\[
 \begin{gathered}
 \text{normalization and polarization}\\[-1pt]
 \Downarrow\\[-1pt]
 \text{geometric fibers and seminormalization}\\[-1pt]
 \Downarrow\\[-1pt]
 \text{completed whole-fiber image algebra}\\[-1pt]
 \Downarrow\\[-1pt]
 \text{conductor and defect modules}
 \end{gathered}
\]
Fourier invariants govern the normalization-side layers; semigroup and one-jet
calculations recover the actual descended algebras, from which conductors and
defect modules are extracted.  The separating examples show that the coarser
data do not, in general, recover the actual descended algebra.

For a homogeneous aperiodic mask \(S\subset C_d\), the monic image algebra is
generated by those power sums whose Fourier coefficients do not vanish.
Whole-fiber completion upgrades this elementary observation to a global
classification: equality of Fourier zero spectra gives the same image with
the normalization identified at every multiplicity, while a sharp stable
sampling bound gives the converse.  For mixed profiles the invariant is no
longer a zero set but a projective Fourier line at each frequency.  Equality
of these lines gives image isomorphisms compatible with the fixed
normalization; additional markings, such as a calibrated singleton or the
normalized relation map, recover progressively finer mask data.

Point fibers alone lead naturally to seminormalization.  In the homogeneous
setting, we show that bounded integer relations in the circulant mask matrix
determine exactly the geometric normalization fibers and therefore the
seminormal image.  The
actual ring contains more information.  Under a spectral condition it is a
semigroup boundary algebra plus the interior ideal, from which the conductor,
Hilbert polynomial, and local depth follow.  A separate finite-output
calculation treats the entire first singular interval for the mask
\(\{0,1\}\subset C_{2L}\).

The mixed singleton--complement family displays the infinitesimal descent most
sharply.  After dividing once by \(x-u\), membership in the actual output
algebra is governed by one transverse jet.  We prove finite free
presentations below the critical multiplicity and at the critical
multiplicity, globalize them through the full normalization fiber, compute the
projective conductor and the local Cohen--Macaulay type, and then pass through
the first multiplicity at which the reduced boundary itself becomes
nonnormal.  At that boundary the ambient image remains locally
Cohen--Macaulay, while the boundary acquires an explicit normalization defect
and the infinitesimal module ceases to be invertible along its nonnormal
locus.

Section~\ref{v7:sec:separations} tests how much normalization-side data
suffice to recover an image.  The examples show successive failures: mixed degrees need not
determine the normalization; the normalization with both hyperplane pullbacks
need not determine the image; all geometric point fibers and the
seminormalization need not determine conductor thickness; and even a fixed
source conductor need not determine the boundary subalgebra.  Thus the
completed whole-fiber image algebra is not merely a technical device but the
natural descent datum in the families treated here.

\subsection{Main theorem chain}

Section~\ref{v6:spec:homogeneous-section} develops the homogeneous completed
fiber model.  Theorem \ref{v6:spec:completed-model} identifies the completed
image as the closed algebra generated by exact moment tuples over the entire
normalization fiber.  Theorem \ref{v6:spec:homogeneous-classification} then
gives the Fourier-zero classification and its stable converse.

Section~\ref{v6:spec:mixed-section} replaces zero spectra by Fourier lines.
Theorem \ref{v6:spec:mixed-classification} gives the mixed marked
classification, and Theorem \ref{v6:spec:relation-rigidity} gives rigidity
when a singleton is present and the normalized relation map is fixed.

Section~\ref{v7:sec:descent} separates reduced gluing from infinitesimal
descent.  Theorem \ref{v7:descent:seminormal-classification} classifies the homogeneous
seminormal image by bounded integer relations.  Theorems \ref{v7:cond:semigroup}
and \ref{v7:cond:first-interval} compute two substantial homogeneous
families explicitly.

Section~\ref{v7:sec:mixed-trace} is the mixed one-jet core.  Theorems
\ref{v7:mixed:affine} and \ref{v8:mixed:critical} give the affine descent
algebras, Theorems \ref{v7:mixed:projective} and
\ref{v8r11:mixed:projective-first} globalize the conductor picture, and
Theorem \ref{v7:mixed:type} computes the local type in the smooth-boundary
range.  Theorem \ref{v8r9:mixed:first-excluded} supplies the finite-output
certificate at the first nonnormal boundary.

\subsection{Relation with preceding work and novelty boundaries}

The two earlier papers \cite{LiYuanOrbits,LiYuanSchur} use the same
remainder and Schur--Pl\"ucker coordinates to study the
rational-normal-curve/GRS locus and the MDS open set.  The companion
reconstruction paper \cite{LiProfileGeometry} moves to the closed rank loci:
it reconstructs the defining linear system, resolves the arithmetic profile
structure, and determines the relevant normalizations and polarizations.
The present paper begins after that normalization step.  Its subject is the
actual descended image algebra and the information lost when one passes from
that algebra to successively coarser normalization-side data.

Generalized power sums and root-string restriction algebras have direct
precedents in
\cite{BrooknerCorwinEtingofSam,EtingofRains,SergeevVeselovMR,KasataniMultiwheel},
and root-of-unity regularization has substantial existing theory
\cite{FJMMTUnity,CherednikNonsemisimple}.  Under the parameter dictionary of
Appendix~\ref{r17:mixed:comparisons}, the nonroot, nonresonant
difference-kernel description is due to Sergeev--Veselov, as recorded in
\cite[Proposition~3.5(i)]{EtingofRains}; the generic quasi-invariance
mechanism is therefore not claimed as new.  Reduced-relation descent and
conductor pinching are classical \cite{Manaresi,Ferrand}.

The contributions emphasized here occur at the prescribed root-of-unity and
actual-image level.  They include whole-fiber completed image algebras;
exact marked Fourier classifications; the bounded-relation description of
homogeneous seminormal images; finite-output generation of the prescribed
root-of-unity algebras with explicit membership conditions and free bases in
the proved ranges; one-jet descent; projective conductors and defect modules;
and separating examples showing that point fibers, seminormalization, and
even the source conductor can be too coarse.  Appendix~\ref{r17:mixed:comparisons}
records the detailed parameter identifications and the boundary of the
generic and ambient precedents; the finite-output and whole-fiber arguments
themselves do not depend on those comparisons.

\subsection{Organization and imported inputs}

Section~\ref{sec:normalization-input} records the small imported
normalization package from \cite{LiProfileGeometry} and derives the
homogeneous Fourier threshold that starts the descent theory.  Sections
\ref{v6:spec:homogeneous-section} and \ref{v6:spec:mixed-section} treat
homogeneous and mixed Fourier descent.  Section~\ref{v7:sec:descent}
studies seminormalization and homogeneous conductor models.  Section
\ref{v7:sec:mixed-trace} develops one-jet descent.  Section
\ref{v7:sec:separations} gives the separating examples, and Section
\ref{sec:outlook-b} records the precise proved ranges and open extensions.
The two appendices contain the squarefree linear models and the parameter
comparisons used by the one-jet theory.
\section{Normalization input and Fourier thresholds}\label{sec:normalization-input}
We work over an algebraically closed field of characteristic zero whenever
Fourier masks are used.

\paragraph{Imported normalization package.}
Only the following results are imported from the companion reconstruction
paper.  First, \cite[Proposition~5.1]{LiProfileGeometry} classifies the normal
relation covers.  Second, \cite[Theorem~5.4]{LiProfileGeometry} gives, for
complete progressions, the product normalizations of the relation branches
and their images together with the two natural polarizations.  Third,
\cite[Theorems~6.1--6.2]{LiProfileGeometry} gives the monic and projective
normality criteria used below.  No completed-image, seminormalization,
image-ring, conductor, or other descent result is imported.

Concretely, for an occupied necklace type \(\nu\) with multiplicity
\(n_\nu\), weight \(w_\nu\), and rotation-stabilizer order \(h_\nu\), the
image normalization has a factor \(\PP^{n_\nu}\); if \(H_\nu\) is its
hyperplane class, the two natural pullbacks are
\[
 \eta=\sum_\nu \frac{w_\nu}{h_\nu}H_\nu,
 \qquad
 \xi=\sum_\nu \frac{d}{h_\nu}H_\nu.
\]
The normality criteria are expressed in terms of the weighted generator
degrees, the relevant Fourier coefficients, and the endpoint allocation.
We use this package as input and do not repeat its proofs.

\subsection{The first vanishing Fourier coefficient}
A profile is \emph{homogeneous} if all its occupied blocks have the same
necklace.  For an aperiodic nonempty necklace \(S\subseteq\mathbb Z/d\mathbb
Z\), repeated \(m\) times, write \(Y_{S,m}\) for its projective image after
removing the fixed factor, and put
\[
 w=|S|,\qquad \lambda_j=\sum_{b\in S}\zeta^{bj},\qquad
 M_S(T)=\sum_{b\in S}T^b.
\]
\begin{corollary}[Fourier and cyclotomic thresholds]
\label{v6:norm:homogeneous-threshold}
The normalization of $Y_{S,m}$ is $\PP^m$, with
$\eta=wH$ and $\xi=dH$, and its complete vector of mixed degrees is
\begin{equation}\label{v6:norm:homogeneous-degrees}
 \delta_j=d^{m-j}w^j\qquad(0\le j\le m).
\end{equation}
Moreover,
\begin{equation}\label{v6:norm:homogeneous-normality}
 Y_{S,m}\text{ is normal}\quad\Longleftrightarrow\quad
 \lambda_1\cdots\lambda_m\ne0
 \quad\Longleftrightarrow\quad Y_{S,m}\simeq\PP^m.
\end{equation}
Let
\[
 \mathcal Q_S=
 \{q:q\mid d,\ q>1,\ \Phi_q(T)\mid M_S(T)\}.
\]
If $\mathcal Q_S$ is empty, every admissible $m$ gives a smooth
image.  Otherwise, with $q_{\max}=\max\mathcal Q_S$,
\begin{equation}\label{v6:norm:cyclotomic-threshold}
 Y_{S,m}\text{ is normal}\quad\Longleftrightarrow\quad
 m<d/q_{\max}.
\end{equation}
\end{corollary}

\begin{proof}
The normalization and polarizations follow from \cite[Theorem~5.4]{LiProfileGeometry}, and integration on $\PP^m$ gives
\eqref{v6:norm:homogeneous-degrees}.  Empty blocks would preserve the image but add relation factors and
change its incidence vector.  With one occupied type, (A) is automatic
and (C) is injectivity of $k\mapsto wk$.
Condition (B) is precisely the middle assertion in
\eqref{v6:norm:homogeneous-normality}.  The projective criterion
proves that equivalence.

The order of $\zeta^j$ is $d/\gcd(d,j)$.  Since $M_S$ has rational
coefficients, its vanishing at $\zeta^j$ is equivalent to
divisibility by the corresponding cyclotomic polynomial.  It never
vanishes at $1$, because $M_S(1)=w>0$.  For a fixed divisor $q>1$
the least positive index with $\zeta^j$ of order $q$ is $d/q$.
The first positive vanishing index is therefore
$\min_{q\in\mathcal Q_S}d/q=d/q_{\max}$, proving
\eqref{v6:norm:cyclotomic-threshold}.
\end{proof}

\begin{remark}\label{v6:norm:periodic-reduction}
If $S$ has a stabilizer of order $h>1$, it is a union of the orbits
of that subgroup and descends to an aperiodic necklace of length
$d/h$.  On the quotient root parameter the relevant Fourier
coefficients are $\lambda_h,\lambda_{2h},\ldots,\lambda_{mh}$.
Writing $D=d/h$, $S=T+D\{0,\ldots,h-1\}$ and
$\theta=\zeta^h$, we have
$\lambda_{hj}=h\sum_{t\in T}\theta^{tj}$.
The cyclotomic set in the threshold is computed from $T$
at length $D$.
Thus the same threshold statement applies after this reduction,
with polarizations $(w/h)H$ and $(d/h)H$.
\end{remark}
\section{Descent I: completed fibers and Fourier zero spectra}
\label{v6:spec:homogeneous-section}

The first descent problem is to determine what the normalization map itself remembers.  Completed images over entire normalization fibers give the Fourier criterion needed for that comparison.  Throughout this section and the next,
$K$ is algebraically closed of characteristic zero, $d\ge2$, and
$\zeta$ is a fixed primitive $d$th root.  All images are reduced.

\subsection{The homogeneous profile and its monic algebra}

Let $S\subset C_d=\mathbb Z/d\mathbb Z$ be nonempty and aperiodic, meaning
that no nonidentity translation fixes $S$, and set $w=|S|$.  The profile
consisting of $m$ copies of this necklace has normalization map
\begin{equation}\label{v6:spec:homogeneous-map}
 \nu_S:\sym^m\PP^1=\PP^m\longrightarrow Y_{S,m},\qquad
 \{[u_i:v_i]\}_{i=1}^m\longmapsto
 \left[\prod_{i=1}^m\prod_{a\in S}(v_iX-\zeta^a u_iZ)\right].
\end{equation}
The ample pullback $\mathcal O_{\PP^m}(w)$ gives finiteness;
aperiodicity recovers the general input from its nonzero $\mu_d$-orbits
and gives birationality, as in \cite[Theorem~5.4]{LiProfileGeometry}, for
every $m\ge1$.  The sparse Krylov interpretation additionally
requires $I=d\{0,\ldots,m\}$, $r=mw$, and $m+1\le mw<dm$;
those inequalities are not imposed below.  Restoring a fixed factor
$X^e$ only gives a linear embedding of the image.

Write
\begin{equation}\label{v6:spec:fourier-coefficients}
 \lambda_j(S)=\sum_{a\in S}\zeta^{aj},\qquad
 Z(S)=\{j\in C_d:\lambda_j(S)=0\}.
\end{equation}
On the monic chart the normalization ring is
$B_m=K[e_1,\ldots,e_m]=K[u_1,\ldots,u_m]^{\mathfrak S_m}$.
Put $p_j=\sum_i u_i^j$.

\begin{proposition}[The monic subalgebra]\label{v6:spec:monic-algebra}
Inside the fixed normalization ring $B_m$, the coordinate ring of the monic
image is
\begin{equation}\label{v6:spec:monic-ring}
 A_{S,m}=K[p_j:j\ge1,\ \lambda_j(S)\ne0].
\end{equation}
Consequently, equal zero spectra give equal monic image subalgebras, conductor
ideals, and normalization quotient modules under this identification.
\end{proposition}

\begin{proof}
The output power sums are $Q_j=\lambda_j(S)p_j$.  Newton's identities
give its coefficient algebra $K[Q_1,\ldots,Q_{mw}]$; the output
polynomial recurrence puts all higher $Q_j$ in the same algebra.
Removing the nonzero scalars $\lambda_j(S)$ proves
\eqref{v6:spec:monic-ring}.  The conductor and quotient depend on
this inclusion in $B_m$.
\end{proof}

The projective calculation must also include fibers meeting both zero
and infinity.

\subsection{All branches of the completed image}

Let
\begin{equation}\label{v6:spec:circulant}
 C_S=(1_S(a-c))_{a,c\in C_d}.
\end{equation}
\emph{Point fibers.}
For an output form $G$, record the multiplicities $wn_0$ and $wn_\infty$
of its roots at zero and infinity.  For each nonzero $\mu_d$-orbit appearing
among its roots, choose a representative $v$ and write
$b_a=\operatorname{mult}_{\zeta^a v}(G)$.  The possible input multiplicities
on this orbit form the finite set
\begin{equation}\label{v6:spec:feasible-fiber}
 \mathcal F_S(b)=\{(n_c)_{c\in C_d}\in\mathbb Z_{\ge0}^d:C_Sn=b\}.
\end{equation}
As a set, the normalization fiber is the product of these
sets over the nonzero orbits; every input divisor also contains
the fixed endpoint part $n_0[0]+n_\infty[\infty]$.
Indeed, root multiplicities add as in
\eqref{v6:spec:homogeneous-map}; $\sum_a b_a=w\sum_c n_c$ recovers
the input size, and $w>0$ makes every occupied orbit visible.

\emph{Exact moments on each configuration.}
For a chosen feasible configuration, write the input roots formally as
$u_{c,i}=\zeta^cv(1+t_{c,i})$, $1\le i\le n_c$, and put
$p_{c,j}=\sum_i t_{c,i}^j$.  If $q_{a,j}$ is the $j$th power sum of the
relative deviations of the output roots in the cluster at $\zeta^av$, then
\begin{equation}\label{v6:spec:exact-moments}
 q_{a,j}=\sum_c(C_S)_{a,c}p_{c,j}\qquad(j\ge1).
\end{equation}
At zero and infinity the exact formulas are
\begin{equation}\label{v6:spec:endpoint-moments}
 Q_j^{(0)}=\lambda_j(S)p_j^{(0)},\qquad
 Q_j^{(\infty)}=\lambda_{-j}(S)p_j^{(\infty)},
\end{equation}
where reciprocal parameters are used at infinity.

\begin{theorem}[Completed image model]\label{v6:spec:completed-model}
Let $G$ be a closed point of $Y_{S,m}$, let $x_1,\ldots,x_h$ be all
points of $\nu_S^{-1}(G)$, and put
$B_i=\widehat{\mathcal O}_{\PP^m,x_i}$.  In the product
$\prod_{i=1}^h B_i$, interpret
\eqref{v6:spec:exact-moments}--\eqref{v6:spec:endpoint-moments} as tuples,
one entry for every feasible configuration in the fiber.  Then
$\widehat{\mathcal O}_{Y_{S,m},G}$ is the closed subalgebra generated by
these moment tuples.

Equivalently, let $R=\widehat{\mathcal O}_{\PP^{mw},G}$ and let $I_i$ be
the kernel of the coefficient map $R\longrightarrow B_i$, expressed by
these formulas after formal factorization into root clusters.  Then
\begin{equation}\label{v6:spec:kernel-intersection}
 \widehat{\mathcal O}_{Y_{S,m},G}=R\big/\bigcap_{i=1}^h I_i.
\end{equation}
\end{theorem}

\begin{proof}
\emph{Root-cluster coordinates.}
Use a local affine chart at each output root and the reciprocal chart at
infinity.  Distinct root factors are coprime, so Hensel factorization
identifies the completed ambient coefficient ring with the completed
tensor product of their cluster coefficient rings.  For a multiplicity-$b$ cluster, Newton's identities and polynomial
recurrence replace its coefficients by the first $b$ moments and
generate all higher moments.  Multiplication gives the exact identities
\eqref{v6:spec:exact-moments}--\eqref{v6:spec:endpoint-moments}.
Each $B_i$ is the power series ring in the elementary symmetric
functions of the input deviations at $x_i$.

\emph{The entire fiber.}
Put $A_G=\mathcal O_{Y_{S,m},G}$ and
$\widetilde A_G=((\nu_S)_*\mathcal O_{\PP^m})_G$.  Finite normalization
and flat completion give
\[
 \widehat A_G\lhook\joinrel\longrightarrow
 \widetilde A_G\otimes_{A_G}\widehat A_G
 \simeq\prod_i B_i
\]
\cite[Tags~00MA and~07N9]{Stacks}.  Its image is closed: the quotient
is a finite module over the complete local ring $\widehat A_G$ and
is separated.  The topology on the finite algebra is also the product
of the maximal-ideal topologies, since its closed fiber is Artinian.
The ambient coefficient map thus has kernel $\bigcap_i I_i$; the exact
cluster moments topologically generate its image.  Repeated roots
remain inside their clusters, proving both descriptions without a
squarefree restriction.
\end{proof}

\subsection{Local tests for the normalization map}
\label{r20:spec:local-tests}
The same model determines differential rank, nonnormal support,
and, for immersive branches, the conductor.

\begin{proposition}[Differential rank]\label{v6:spec:differential-rank}
For a feasible nonzero-orbit configuration $n$, set
$J_j(n)=\{c:n_c\ge j\}$.  Its contribution to the rank of the differential
of $\nu_S$ is
\begin{equation}\label{v6:spec:rank-formula}
 \sum_{j\ge1}\rank C_S[:,J_j(n)].
\end{equation}
The zero and infinity contributions are respectively
\begin{equation}\label{v6:spec:endpoint-ranks}
 \#\{1\le j\le n_0:\lambda_j(S)\ne0\},\qquad
 \#\{1\le j\le n_\infty:\lambda_{-j}(S)\ne0\}.
\end{equation}
The nonzero-orbit map is immersive exactly when the active columns
$C_S[:,\operatorname{supp}n]$ are linearly independent.
\end{proposition}

\begin{proof}
If $e_{c,j}$ denotes the $j$th elementary symmetric function in the
$n_c$ deviations, then modulo the square of the source maximal ideal
\[
 p_{c,j}\equiv(-1)^{j-1}j e_{c,j}\quad(j\le n_c),
 \qquad p_{c,j}\equiv0\quad(j>n_c).
\]
Different $j$ use disjoint cotangent variables.  Formula
\eqref{v6:spec:exact-moments} hence gives the direct sum of the displayed
matrix blocks.  If an output cluster has fewer than $j$ roots, its row in
the indicated submatrix is zero: a contributing column would force that
cluster to contain at least $n_c\ge j$ roots.  Thus no extraneous output
coordinate is being counted.  The endpoint argument is identical.
Dependence of the active columns already causes a rank loss for $j=1$;
independence implies independence for every smaller column set $J_j(n)$.
Moreover, a constant left inverse of the active-column matrix recovers
every $p_{c,j}$ from all $q_{a,j}$, and Newton's identities recover all
source coordinates.  This also proves formal immersivity directly.
\end{proof}

Let $\mathfrak c$ denote the conductor of
$\mathcal O_{Y_{S,m}}\subset(\nu_S)_*\mathcal O_{\PP^m}$, regarded
also as an ideal on the normalization.  Put
$\tau(S)=\min\{j\ge1:\lambda_j(S)=0\}$, with value $\infty$ if the
set is empty.  Since the mask polynomial has integer coefficients,
$\lambda_j(S)$ and $\lambda_{-j}(S)$ are Galois conjugate, so they have
the same vanishing behavior.

\begin{corollary}[The nonnormal locus]\label{v6:spec:conductor-locus}
The point $G$ belongs to the nonnormal locus $V(\mathfrak c)$ if and only
if at least one of the following occurs:
\begin{enumerate}[label=\textup{(\roman*)}]
\item $n_0\ge\tau(S)$ or $n_\infty\ge\tau(S)$;
\item for some nonzero orbit, $\mathcal F_S(b)$ contains at least two
elements;
\item some feasible configuration has linearly dependent active columns
of $C_S$.
\end{enumerate}
\end{corollary}

\begin{proof}
If none occurs, the normalization fiber is a singleton.  Left inverses
in the nonzero clusters, and the nonzero endpoint coefficients for
$1\le j\le n_0,n_\infty$, recover every source elementary symmetric
coordinate.  Theorem~\ref{v6:spec:completed-model} then identifies the
completed image ring with the regular normalization ring.  Faithful
flatness of completion for the finite normalization quotient gives an
isomorphism near $G$.  Conversely, a multiple fiber or a differential
rank loss is incompatible with normalization being an isomorphism.
Proposition~\ref{v6:spec:differential-rank} detects exactly these losses.
\end{proof}

There is also a scheme-theoretic conductor formula when all formal
branches are immersive.

\begin{proposition}[Conductor of immersive formal branches]
\label{v6:spec:immersive-conductor}
Suppose every branch over $G$ is immersive.  With the notation of
Theorem~\ref{v6:spec:completed-model}, one has $B_i=R/I_i$, and the
conductor in $\prod_i B_i$ is
\begin{equation}\label{v6:spec:branch-conductor}
 \widehat{\mathfrak c}\,\prod_iB_i
 =\prod_{i=1}^h
 \frac{I_i+\bigcap_{j\ne i}I_j}{I_i}.
\end{equation}
For $h=1$, the empty intersection is understood to be $R$.
\end{proposition}

\begin{proof}
Put $D=\widetilde A_G/A_G$ and
$\mathfrak c_G=\operatorname{Ann}_{A_G}(D)$.
Since $D$ is a finite $A_G$-module, its annihilator commutes
with flat completion.  Thus the conductor of the completed
inclusion is $\mathfrak c_G\widehat A_G$.
Formal immersivity gives the surjection $R\twoheadrightarrow B_i$.
An element supported only on factor $i$ of $\prod_iB_i$ lies in the
image of $R/\bigcap_jI_j$ precisely when it has a lift in
$\bigcap_{j\ne i}I_j$.  Its possible values are therefore the $i$th
ideal in \eqref{v6:spec:branch-conductor}.  The product of these ideals
is contained in the image and is an ideal of its normalization, hence
belongs to the conductor.  Conversely, multiplying a conductor element
by each idempotent of the product isolates such an element and proves
the reverse inclusion.
\end{proof}

For a nonimmersive branch, $B_i$ need not equal $R/I_i$; one must
retain the full completed subalgebra instead of using
\eqref{v6:spec:branch-conductor}.

\subsection{Classification with the normalization identified}
\label{r20:spec:classification-route}

\paragraph{The finite graph argument.}
For the finite normalization maps in this section and the next, let
$f_1:N\to Y_1$ and $f_2:N\to Y_2$ have the same geometric fibers
on the identified normalization.  Suppose their completed image rings
coincide inside $\prod_{x\in F}\widehat{\mathcal O}_{N,x}$ for
every common fiber $F$.  Take the reduced image $Z$ of $(f_1,f_2)$, with projections
$\pi_i:Z\to Y_i$.  These and $N\to Z$ are proper with finite fibers,
hence finite \cite[Tag~02LS]{Stacks}.  Their function fields are all $K(N)$, so $N\to Z$
is the normalization.  A common fiber gives one graph point $z$ over each $y_i$.  Its
coefficient ring is generated by both outputs.  The whole-fiber
argument of Theorem~\ref{v6:spec:completed-model} identifies
$\widehat{\mathcal O}_{Z,z}$ with the closed algebra generated by
both completed image rings inside the same product.  Their equality makes both projections
isomorphisms on completions.  The finite quotients
$(\pi_i)_*\mathcal O_Z/\mathcal O_{Y_i}$ then vanish by
faithful flatness of completion \cite[Tag~00MC]{Stacks}; closed points detect
their support since the images are of finite type over $K$.
Thus both projections are isomorphisms.  The induced isomorphism $Y_1\to Y_2$ lifts the identity on $N$;
uniqueness follows on the dense normalization-isomorphism open.
The hypothesis concerns embedded completed rings.

\paragraph{The converse sampling bound.}
Put $b(d)=d/p_{\min}(d)$, with $p_{\min}(d)$ the smallest prime
factor.  Cyclotomic conjugacy reduces nonzero frequencies to
$j=d/q\le b(d)$, $q\mid d$, $q>1$.  The bound below is needed
only to recover the spectrum from the image, not for sufficiency.

\begin{theorem}[Homogeneous Fourier classification]
\label{v6:spec:homogeneous-classification}
Let $S,T\subset C_d$ be nonempty aperiodic masks.  If $Z(S)=Z(T)$,
then, for every $m\ge1$, there is a unique isomorphism
\begin{equation}\label{v6:spec:marked-isomorphism}
 \varphi:Y_{S,m}\xrightarrow{\ \sim\ }Y_{T,m},\qquad
 \varphi\nu_S=\nu_T,
\end{equation}
under the fixed identification of their normalizations with $\PP^m$.
If $m\ge b(d)$, the converse holds.
\end{theorem}

\begin{proof}
\emph{Sufficiency: fibers and completed subalgebras.}
Fourier diagonalization identifies equality of zero spectra with equality
of both kernels and row spaces of $C_S,C_T$.  Kernel equality gives
the same equivalence relation on each nonzero orbit in
\eqref{v6:spec:feasible-fiber}.  Endpoint multiplicities divided by $|S|$ or $|T|$ recover the input
counts.  Thus both maps induce the same equivalence relation on
$\PP^m$, even when their output degrees differ.

Row-space equality gives constant matrices $L,M$ with $C_T=LC_S$
and $C_S=MC_T$.  Applied to \eqref{v6:spec:exact-moments}, they
express either set of output moment tuples in terms of the other, simultaneously
on all branches of a common fiber.  Absent output positions are padded
by zero moments.  Different cluster multiplicities change the finite
Newton truncation, not the algebra topologically generated by all moments.
The endpoint algebras agree by \eqref{v6:spec:monic-ring} and its reciprocal version.
Theorem~\ref{v6:spec:completed-model} thus gives equal completed
subalgebras in the common product.  The graph argument above proves
\eqref{v6:spec:marked-isomorphism} and uniqueness for every $m\ge1$.

\emph{Necessity: the total-zero collision.}
At the unique normalization point with all parameters zero, put
$\widehat B=K[[e_1,\ldots,e_m]]$.  The image of the completed
image maximal ideal in $\mathfrak m_{\widehat B}/\mathfrak m_{\widehat B}^2$ is
\begin{equation}\label{v6:spec:homogeneous-cotangent}
 \Span_K\{e_j:1\le j\le m,\ \lambda_j(S)\ne0\}.
\end{equation}
Newton's identities give these linear terms; higher $p_j$ and
products add none.  The identity lift fixes the $e_j$ and hence
equates the embedded subspaces.
For $m\ge b(d)$, they detect every $j=d/q$, $q\mid d$, $q>1$.
The integral mask polynomial vanishes there exactly when $\Phi_q$
divides it, which determines vanishing at all frequencies of order $q$.
These recover the nonzero spectrum; $\lambda_0=|S|>0$.
\end{proof}

\begin{remark}\label{v6:spec:classification-scope}
The theorem classifies images with a fixed normalization identification.
It does not classify arbitrary abstract isomorphisms allowing an arbitrary
automorphism of $\PP^m$ on the normalization.  Nor does equality of the
pulled-back hyperplane bundles prove that their descents to a nonnormal
image are identified.  In particular, complementary masks have equal zero
spectra and canonically isomorphic images, although their hyperplane
pullbacks are $\mathcal O(|S|)$ and $\mathcal O(d-|S|)$.
\end{remark}

The next examples separate forgotten necklace shape from sharpness of
the converse bound.

\begin{example}[A mask forgotten by its homogeneous image]
\label{v6:spec:forgotten-mask}
In $C_{12}$, put
\[
 S=\{0,1,2,3\},\qquad T=\{0,1,2,7\}.
\]
The cyclic gap multisets are respectively $\{1,1,1,9\}$ and
$\{1,1,5,5\}$, so the masks are inequivalent even after reflection.
Both are aperiodic and have weight four.  Their mask polynomials have
exactly the cyclotomic divisors $\Phi_2$ and $\Phi_4$ among the
$\Phi_q$ with $q\mid12$, $q>1$, and hence
$Z(S)=Z(T)=\{3,6,9\}$.  Theorem~\ref{v6:spec:homogeneous-classification}
gives the same marked normalization image for every $m$.  For $m\ge3$
these images are nonnormal at the total-zero collision by
Corollary~\ref{v6:spec:conductor-locus}, so the identification also
retains nontrivial boundary structure.
\end{example}

\begin{example}[Sharpness of the cyclotomic sampling bound]
\label{v7:spec:sharp-bound}
Let $d$ be composite, $p=p_{\min}(d)$, $S=\{0,\ldots,p-1\}$, and
$T=\{0\}$.  Both masks are aperiodic.  The geometric-series formula
shows that the first positive zero of $\lambda_j(S)$ is $d/p=b(d)$,
whereas $T$ has no zero.  For $m<b(d)$, both marked images are
isomorphic to their normalization $\PP^m$ by
Corollary~\ref{v6:norm:homogeneous-threshold}, although their spectra
differ.  Thus no smaller uniform bound works in the full parameterized
class of Theorem~\ref{v6:spec:homogeneous-classification}.

Sharpness occurs with equal weights within the Krylov range for
$d=2L$, $L\ge3$ odd, using $S=\{0,1\}$ and $T=\{0,2\}$.
The first zero of $S$ is $L=b(d)$, while the odd order of $\zeta^2$
prevents any zero for $T$.  For $1\le m<L$, both have $r=2m$ and
$s=m+1\le r<dm$, but identical marked smooth images despite different
spectra.  For prime $d$, all proper nonempty masks have empty zero
spectrum; no nontrivial sharpness claim is made in that case.
\end{example}

\section{Descent II: mixed Fourier lines and marked rigidity}
\label{v6:spec:mixed-section}

The homogeneous zero spectrum records only one Fourier coefficient at a time.  For mixed types the corresponding invariant is a projective Fourier line.  The successive markings considered here are an identified normalization, a calibrated singleton, and the normalized relation map; their converse bounds are different.

\subsection{Projective Fourier lines}

Choose two ordered lists of nonempty aperiodic masks
$\boldsymbol S=(S_1,\ldots,S_t)$ and
$\boldsymbol T=(T_1,\ldots,T_t)$ in $C_d$.  Within each list assume that
no two masks are rotation-equivalent.  Fix positive multiplicities
$n_1,\ldots,n_t$, and put $m=\sum_\nu n_\nu$ and
\[
 \mathcal N=\prod_{\nu=1}^t\sym^{n_\nu}\PP^1
           =\prod_{\nu=1}^t\PP^{n_\nu}.
\]
These are the occupied-type multiplicities in the product normalization of
\cite[Theorem~5.4]{LiProfileGeometry}.  Identify the normalization
coordinates on both sides and define
\begin{equation}\label{v6:spec:mixed-map}
 f_{\boldsymbol S}:\mathcal N\longrightarrow Y_{\boldsymbol S},
 \qquad
 ([u_{\nu,i}:v_{\nu,i}])\longmapsto
 \left[\prod_{\nu,i}\prod_{a\in S_\nu}
 (v_{\nu,i}X-\zeta^a u_{\nu,i}Z)\right].
\end{equation}
The ample pullback $\mathcal O_{\mathcal N}(|S_1|,\ldots,|S_t|)$
gives finiteness.  The distinct aperiodic types recover a general
input from distinct nonzero orbits, so $\mathcal N$ is the normalization
as in \cite[Theorem~5.4]{LiProfileGeometry}.
Write $\lambda^{\boldsymbol S}_{\nu,j}=\sum_{a\in S_\nu}\zeta^{aj}$
and define, with the zero subspace allowed,
\begin{equation}\label{v6:spec:fourier-lines}
 \ell_{\boldsymbol S}(j)
 =K(\lambda^{\boldsymbol S}_{1,j},\ldots,
       \lambda^{\boldsymbol S}_{t,j})\subset K^t,
 \qquad j\in C_d.
\end{equation}
On the monic chart put $P_{\nu,j}=\sum_i u_{\nu,i}^j$ and
$Q_j^{\boldsymbol S}=\sum_\nu\lambda^{\boldsymbol S}_{\nu,j}P_{\nu,j}$.
Newton's identities and the output recurrence give the actual algebra
$K[Q_1^{\boldsymbol S},\ldots,Q_{R_{\boldsymbol S}}^{\boldsymbol S}]
=K[Q_j^{\boldsymbol S}:j\ge1]$, where
$R_{\boldsymbol S}=\sum_\nu n_\nu|S_\nu|$.
This is the generalized power-sum form of
\cite[Section~1]{EtingofRains}, but allows zero coefficients,
unlike the convention there; no specialization theorem is used here.
The endpoint weight line is
$\ell_{\boldsymbol S}(0)=K(|S_1|,\ldots,|S_t|)$.
Detecting it accounts for the two converse bounds below.

\begin{theorem}[Mixed Fourier classification]
\label{v6:spec:mixed-classification}
If $\ell_{\boldsymbol S}(j)=\ell_{\boldsymbol T}(j)$ for every
$j\in C_d$, then there is a unique isomorphism
\[
 \varphi:Y_{\boldsymbol S}\xrightarrow{\ \sim\ }Y_{\boldsymbol T},
 \qquad \varphi f_{\boldsymbol S}=f_{\boldsymbol T}.
\]
If every $n_\nu\ge d$, the converse holds.  If the weight vectors
$(|S_\nu|)_\nu$ and $(|T_\nu|)_\nu$ are already known to be
proportional, the converse requires only $n_\nu\ge b(d)$ for all $\nu$.
\end{theorem}

\begin{proof}
\emph{Fourier comparison.}
Form $K_{\boldsymbol S}=[C_{S_1}\ \cdots\ C_{S_t}]$ and
$K_{\boldsymbol T}=[C_{T_1}\ \cdots\ C_{T_t}]$.
For $\mathcal F_{j,a}=\zeta^{ja}$ one has
$\mathcal FC_{S_\nu}=\operatorname{diag}(\lambda^{\boldsymbol S}_{\nu,j})\mathcal F$.
Thus the block row space at frequency $j$ is
$\ell_{\boldsymbol S}(j)$.  Choose nonzero $c_j$ with
$\lambda^{\boldsymbol T}_{\nu,j}=c_j\lambda^{\boldsymbol S}_{\nu,j}$
for all $\nu$, arbitrarily if the common line is zero, and extend them
periodically.  The invertible circulant matrix
$L=\mathcal F^{-1}\operatorname{diag}(c_j)\mathcal F$ satisfies
$K_{\boldsymbol T}=LK_{\boldsymbol S}$; the block kernels agree.

\emph{Entire fibers with coupled type counts.}
On a visible nonzero orbit $\Omega$, let $n_{\nu,\Omega,c}$ be the
input counts.  Feasible allocations satisfy the output vector equations
$K_{\boldsymbol S}n_\Omega=b_\Omega$, the endpoint equations with
weights $|S_\nu|$, and, for every type,
\[
 \sum_\Omega\sum_{c\in C_d}n_{\nu,\Omega,c}
       +n_{\nu,0}+n_{\nu,\infty}=n_\nu.
\]
The nonnegative integer allocations are coupled by these global
constraints.  Kernel equality preserves nonzero-orbit output equalities;
proportionality of the weight vectors preserves endpoint equalities.
Every occupied orbit is visible, so both maps have the same entire
fibers.  This compares fixed source allocations and requires no
positivity of $L$.

\emph{Completed subalgebras and descent.}
Fix a common fiber.  On every branch write the nonzero-orbit inputs as
$\zeta^cv(1+t_{\nu,c,i})$ and put
$p_{\nu,c,j}=\sum_i t_{\nu,c,i}^j$.  Their exact output moments satisfy
\[
 q^{\boldsymbol S}_{a,j}
 =\sum_{\nu,c}1_{S_\nu}(a-c)p_{\nu,c,j},\qquad
 q^{\boldsymbol T}_{\bullet,j}=Lq^{\boldsymbol S}_{\bullet,j}.
\]
The same $L$ acts on all moment orders and feasible allocations
in the common normalization product, with inverse giving the reverse
inclusion.  Unoccupied output positions have zero moments on every branch.  At zero,
\begin{equation}\label{v6:spec:mixed-zero-moments}
 Q_j^{\boldsymbol S}=\sum_\nu
       \lambda^{\boldsymbol S}_{\nu,j}p_{\nu,j},
 \qquad Q_j^{\boldsymbol T}=c_j Q_j^{\boldsymbol S};
\end{equation}
both moments vanish when their common line is zero.  At infinity use
$c_{-j}$ and reciprocal parameters.  These identities hold even when
endpoint allocations vary among branches.
The Hensel--Newton argument of Theorem~\ref{v6:spec:completed-model}
with these coupled configurations identifies the completed images
with the closed algebras of their whole-fiber moment tuples.
The identities give equal subalgebras there.
The finite graph argument of Section~\ref{r20:spec:classification-route}
now proves sufficiency and uniqueness.

\emph{Necessity and its two thresholds.}
At the unique all-zero normalization point put
\[
 \widehat B=K[[e_{\nu,j}:1\le\nu\le t,\ 1\le j\le n_\nu]],
 \qquad \deg e_{\nu,j}=j.
\]
The image of the completed image maximal ideal in
$\mathfrak m_{\widehat B}/\mathfrak m_{\widehat B}^2$ has weight-$j$
part spanned by
\begin{equation}\label{v6:spec:mixed-cotangent}
 (-1)^{j-1}j\sum_{\nu:n_\nu\ge j}
       \lambda^{\boldsymbol S}_{\nu,j}e_{\nu,j}.
\end{equation}
These Newton linear terms all occur among the finite outputs since
$\max_\nu n_\nu\le R_{\boldsymbol S}$; products and completion
add no classes.  An identity lift equates the weighted subspaces.  If every $n_\nu\ge d$, weights
$1,\ldots,d$ detect every line, including frequency zero at weight $d$.
If the weight line is known, it is enough to see every
$j=d/q\le b(d)$, $q\mid d$, $q>1$.
The coordinate entries and $2\times2$ minors belong to
$\mathbb Q(\zeta)$.  Nonzero lines agree exactly when these minors
vanish; a zero line requires all coordinates to vanish on both sides.
Cyclotomic conjugacy transports these conditions, including coordinate
nonvanishing, to every frequency of the same order.  This proves the
sharpened converse over the stated field $K$.
\end{proof}

On squarefree nonzero outputs, Proposition~\ref{v7:linear:exact}
gives the completed image from all feasible allocations as an exact
linear configuration, with the global type counts retained.

\subsection{Calibrating a marked singleton}
\label{r20:spec:calibration}
A fixed singleton turns each observed line into a vector.  Its calibration
recovers the nonconstant Fourier data without first reading the weight line.

\begin{corollary}[A calibrated singleton]\label{v6:spec:singleton-marked}
Suppose both ordered lists contain a distinguished singleton
$S_0=T_0=\{0\}$, whose normalization coordinate is part of the marking,
and suppose every multiplicity is at least $b(d)$.  Then the completed
image germ at the total-zero collision, embedded in the fixed completed
normalization ring, determines every mask
representative $S_\nu$.  In particular, an image isomorphism lifting the
identity of the marked normalization exists if and only if
$S_\nu=T_\nu$ for every $\nu$.
\end{corollary}
\begin{proof}
At every sampled frequency $j=d/q\le b(d)$, the hypotheses include
$n_0\ge j$, so the singleton coordinate $e_{0,j}$ is present in
\eqref{v6:spec:mixed-cotangent}.
After removing the common Newton factor $(-1)^{j-1}j$,
its coefficient is one, so the recovered line has a unique
representative whose singleton coordinate is $1$.  Cyclotomic conjugacy then
recovers every $\lambda_{\nu,j}$ with $j\ne0$ in $C_d$.  Inverse Fourier transformation without the
constant frequency gives
\begin{equation}\label{v6:spec:fourier-inversion}
 \frac1d\sum_{j=1}^{d-1}\lambda_{\nu,j}\zeta^{-aj}
   =1_{S_\nu}(a)-\frac{|S_\nu|}{d}.
\end{equation}
Each mask is a nonempty proper subset: the full subset is not
aperiodic for $d\ge2$.  The two values in the right-hand function
therefore determine its $0/1$ indicator uniquely.  Equivalently, two
such indicators differing by a constant must coincide.  Necessity follows from \eqref{v6:spec:mixed-cotangent} without the
weight frequency; sufficiency is the parameterization.
\end{proof}

Fixed coordinates recover representatives; the next marking instead
allows the rotations preserving the normalized relation map.

\subsection{The normalized relation correspondence}
\label{r20:spec:relation-correspondence}

For an aperiodic occupied profile $\pi$ with types $S_\nu$ and
multiplicities $n_\nu$, put
$\mathcal N_\pi=\prod_\nu\sym^{n_\nu}\PP^1$ and write
$f_\pi:\mathcal N_\pi\longrightarrow Y_\pi$ for its normalization
map.  The normalized relation map is
\begin{equation}\label{v6:spec:relation-map}
 \widehat q_\pi:\mathcal N_\pi\longrightarrow\sym^m\PP^1=\PP^m,
 \qquad
 ([u_{\nu,i}:v_{\nu,i}])\longmapsto
 \sum_{\nu,i}[u_{\nu,i}^{d}:v_{\nu,i}^{d}].
\end{equation}
Under the relation-space identification for $I=d\{0,\ldots,m\}$,
this records the primitive relation divisor.  A
\emph{relation-compatible image isomorphism}
$\psi:Y_\pi\xrightarrow{\sim}Y_{\pi'}$ means that its unique
normalization lift $\widetilde\psi$ satisfies
$\widehat q_{\pi'}\widetilde\psi=\widehat q_\pi$, with the relation
base fixed.  Thus the object being classified is the diagram
$Y_\pi\longleftarrow\mathcal N_\pi\longrightarrow\PP^m$.

\begin{proposition}[Automorphisms of the normalized relation cover]
\label{v6:spec:normalizer}
Let $\Gamma=C_d\wr\mathfrak S_m$, and let
$H_\pi=\prod_\nu\mathfrak S_{n_\nu}\subset\mathfrak S_m
\subset\Gamma$ be the Young subgroup determined by the types.  Then
\begin{equation}\label{v6:spec:normalizer-formula}
 \operatorname{Aut}_{\PP^m}(\mathcal N_\pi)
 \simeq N_\Gamma(H_\pi)/H_\pi
 \simeq C_d^t\rtimes
       \prod_{n\ge1}\mathfrak S_{\#\{\nu:n_\nu=n\}}.
\end{equation}
These automorphisms independently rotate all parameters of each type
by a common $d$th root of unity and permute type factors of equal
dimension.  Two such normalizations are isomorphic over the fixed
relation base exactly when their multiplicity multisets agree.
\end{proposition}

\begin{proof}
The quotients of $(\PP^1)^m$ by $\Gamma$ and $H_\pi$ are the
fixed relation base and $\mathcal N_\pi$, respectively.  On
$F=K(u_1,\ldots,u_m)$ the action is faithful, with
$F^\Gamma=K(e_1(u_1^d,\ldots,u_m^d),\ldots,e_m(u_1^d,\ldots,u_m^d))$;
hence $F/F^\Gamma$ is Galois with group $\Gamma$
\cite[Tag~09I3]{Stacks}.  Restriction from $N_\Gamma(H_\pi)$ to
$\operatorname{Aut}_{F^\Gamma}(F^{H_\pi})$ has kernel $H_\pi$ and is
surjective: any intermediate-field automorphism extends to the finite
Galois extension $F$ \cite[Tag~0BME]{Stacks}.  The finite normal cover is, on each affine
base open, the integral closure in $F^{H_\pi}$.  Its field
automorphisms preserve that closure, and therefore extend uniquely
across all boundary points.

In additive notation on $C_d^m$, for $h\in H_\pi$,
\[
 (b,\sigma)(0,h)(b,\sigma)^{-1}
  =\bigl(b-(\sigma h\sigma^{-1})b,\ \sigma h\sigma^{-1}\bigr).
\]
Thus $\sigma$ permutes Young blocks of equal size, and $b$ is
constant on each block since its transpositions generate $H_\pi$.
These conditions are also sufficient.  Modulo internal permutations,
they give one $C_d$ per type; permutations of equal-sized blocks
supply the splitting in \eqref{v6:spec:normalizer-formula}.
For singleton blocks the constancy condition is vacuous, so they
cause no exception.  The final assertion follows from \cite[Proposition~5.1]{LiProfileGeometry}:
conjugacy of these Young subgroups is equivalent to equality of their
block-size multisets.  The base morphism is fixed throughout, and
no sparse-range inequality is needed.
\end{proof}

\subsection{Rigidity with the relation base fixed}
\label{r20:spec:relation-rigidity-route}
Here the singleton is present but not premarked.  Unlike the calibrated
case, identifying it uses the weight line, seen at moment order $d$.

\begin{theorem}[Rigidity with a singleton and the relation map]
\label{v6:spec:relation-rigidity}
Fix $d\ge2$ and the relation base $\PP^m$.  Let $\pi,\pi'$ be
profiles consisting of distinct nonempty aperiodic necklace types,
each containing the singleton necklace.  Suppose every type on each
side has multiplicity at least $d$.  Then the following are equivalent:
\begin{enumerate}[label=\textup{(\roman*)}]
\item $Y_\pi$ and $Y_{\pi'}$ admit a relation-compatible image
isomorphism;
\item there is a bijection of their necklace types preserving both
the necklace classes and their multiplicities.
\end{enumerate}
No separate marking of the singleton factor or of the normalization
coordinates is required.
\end{theorem}

\begin{proof}
\emph{Allowed lifts.}
An isomorphism in (i) first identifies the normalized relation covers.
Proposition~\ref{v6:spec:normalizer} recovers the multiplicity multiset,
so after matching factors its lift has the form of a permutation
$\sigma$ preserving multiplicities and rotations
$u_{\nu,i}\mapsto\zeta^{c_\nu}u_{\nu,i}$.  The point where all
parameters vanish is the unique point over the all-zero relation
divisor, and is preserved.  It is also the unique point over its image
under $f_\pi$.  On cotangent coordinates the lift pulls back
$e'_{\sigma(\nu),j}$ to $\zeta^{c_\nu j}e_{\nu,j}$; all these
coordinates occur for $1\le j\le d$ by the multiplicity bound.

\emph{Identifying the singleton.}
Equality of the two completed image subrings at this point, pulled back
by the lift, gives from \eqref{v6:spec:mixed-cotangent} nonzero scalars
$a_j$ such that
\begin{equation}\label{v6:spec:rigidity-fourier}
 \lambda'_{\sigma(\nu),j}\zeta^{c_\nu j}
       =a_j\lambda_{\nu,j}\qquad(1\le j\le d).
\end{equation}
The singleton ensures the relevant lines are nonzero.  At $j=d$ the
rotations disappear and $|S'_{\sigma(\nu)}|=a_d|S_\nu|$.  The scalar
$a_d$ and its inverse are positive integers, since each weight list
contains one.  Thus $a_d=1$; no ordering of $K$ is used.  Distinctness of the necklace
types makes the weight-one type unique, so $\sigma$ sends the singleton
to the singleton.  This uses the frequency-zero line at weight $d$.

\emph{Recovering the necklace classes.}
Choose its representative on both sides to be $\{0\}$ and call its
index $0$.  The singleton coordinate in
\eqref{v6:spec:rigidity-fourier} gives $a_j=\zeta^{c_0j}$, and hence
\[
 \lambda'_{\sigma(\nu),j}
 =\zeta^{(c_0-c_\nu)j}\lambda_{\nu,j}\qquad(1\le j\le d).
\]
Fourier inversion now yields
$S'_{\sigma(\nu)}=S_\nu+(c_0-c_\nu)$ in $C_d$.  The necklace
classes and the multiplicities agree, proving (ii).

\emph{The converse construction.}
Conversely,
permuting the types and absorbing rotations of their representatives
by inverse parameter rotations gives the same output forms on the
entire products, including zero and infinity, and preserves their
$d$th powers.  The resulting image isomorphism therefore identifies
the required diagrams, not merely their point fibers.
\end{proof}

In the same range, the relation-compatible automorphisms of a fixed
$Y_\pi$ are exactly the diagonal $C_d$.  Indeed the proof forces
$\sigma$ to preserve each distinct necklace class, hence $\sigma$ is
the identity; aperiodicity then gives $c_\nu=c_0$ for every type.
Conversely, simultaneous rotation induces
$[G(X,Z)]\mapsto[G(X,\zeta^{c_0}Z)]$ and fixes $\widehat q_\pi$.
This action is faithful because $f_\pi$ is birational.

The sufficient stable bound is not asserted optimal.  The relation
map remains part of the classified diagram.

\begin{example}[The relation map need not descend to the image]
\label{v6:spec:no-relation-descent}
Take $d=3$, the singleton $\{0\}$ and pair $\{0,1\}$, each with
multiplicity three.  Then
$\mathcal N=\PP^3\times\PP^3$, $m=6$, $r=9$, and
$I=3\{0,\ldots,6\}$.  At one normalization point put two singleton
parameters at infinity and all other parameters at zero.  At a second
point put one pair parameter at infinity and all others at zero.
Both output forms are $[X^7Z^2]$.  Their primitive relation divisors
have infinity multiplicities two and one, respectively.  Therefore
$\widehat q$ is not constant on the normalization fiber and cannot
descend to a morphism from $Y$.  The diagram in
Theorem~\ref{v6:spec:relation-rigidity} deliberately places the relation
map on the normalization.
\end{example}

\begin{example}[A single probe already separates equal numerical data]
\label{v6:spec:small-probe}
Take $d=12$, and pair the singleton type with either mask from
Example~\ref{v6:spec:forgotten-mask}, each type occurring once.  Both
profiles have $I=\{0,12,24\}$ and $r=5$, and their common data are
\[
 D=288,\qquad (\delta_0,\delta_1,\delta_2)=(288,60,8),\qquad
 \mathcal N=\PP^1\times\PP^1,
 \quad \eta=(1,4),\quad\xi=(12,12).
\]
Nevertheless their normalized relation correspondences are not
isomorphic.  The cover automorphism group is
$C_{12}^2\rtimes\mathfrak S_2$.  On the first cotangent Fourier line
at the total-zero point, such an automorphism can only multiply the
ratio of the two coordinates by a twelfth root of unity or invert the
ratio.  For a complex realization of $\zeta$ one computes
\[
 |\lambda_1(S)|^2=6+3\sqrt3,\qquad
 |\lambda_1(T)|^2=3.
\]
These numbers are neither equal nor reciprocal.  The claimed
isomorphism is therefore impossible.  The forbidden equalities lie in the cyclotomic field, so their failure
over $\mathbb C$ rules them out over every algebraically closed
characteristic-zero field.
Here a singleton of multiplicity one separates the normalized
relation correspondences despite equality of the homogeneous
marked images; the stable bound of
Theorem~\ref{v6:spec:relation-rigidity} is unchanged.
\end{example}

We next separate point identifications from the functions that descend.

\section{Seminormal descent and homogeneous conductor models}
\label{v6:cond:section}\label{v7:sec:descent}

Work over an algebraically closed characteristic-zero field $k$, with
reduced profile images.  We recover homogeneous seminormal images from
point relations and compute descent in two explicit families.

\subsection{Point relations and the descent data}
\label{r21:descent:point-data}

Let $\nu:N\to Y$ be a finite normalization and factor it through the
seminormalization $\pi:Y^{\mathrm{sn}}\to Y$.  On $Y$ put
\begin{equation}\label{v7:descent:layers}
 \mathcal Q=\nu_*\mathcal O_N/\mathcal O_Y,\qquad
 \mathcal Q_{\mathrm{inf}}=\pi_*\mathcal O_{Y^{\mathrm{sn}}}/\mathcal O_Y,
 \qquad
 \mathcal Q_{\mathrm{glue}}=\nu_*\mathcal O_N/\pi_*\mathcal O_{Y^{\mathrm{sn}}}.
\end{equation}
There is a natural exact sequence
\begin{equation}\label{v7:descent:extension}
 0\longrightarrow\mathcal Q_{\mathrm{inf}}
 \longrightarrow\mathcal Q\longrightarrow\mathcal Q_{\mathrm{glue}}
 \longrightarrow0.
\end{equation}
Its annihilator $\mathfrak c=\operatorname{Ann}_{\mathcal O_Y}\mathcal Q$
is the conductor.  In affine notation $A\subset B$, the classical
reduced-relation equalizer and conductor square are
\begin{equation}\label{v7:descent:classical}
 A^{\mathrm{sn}}=
 \operatorname{Eq}\bigl(B\rightrightarrows(B\otimes_A B)_{\mathrm{red}}\bigr),
 \qquad
 A=B\mathop{\times}_{B/\mathfrak c}A/\mathfrak c.
\end{equation}
These are Manaresi's characteristic-zero reduced-relation construction
\cite[Theorem I.6 and Remark I.7]{Manaresi} and Ferrand's
conductor square \cite[Lemma 1.3]{Ferrand}.  Indeed, the equalizer $C\subset B$ is a finite
$A$-module.  Every geometric $C$-point lifts to $B$ by integrality,
and the equalizer makes lifts over the same $A$-point agree on $C$.
Thus $\Spec C\to\Spec A$ is a finite universal homeomorphism; its
residue extensions are trivial in characteristic zero, so it is
subintegral.  If $b^2,b^3\in C$ in the common total quotient ring,
normality gives $b\in B$.  Its two images in
$(B\otimes_A B)_{\mathrm{red}}$ have equal squares and cubes and hence
agree modulo every prime.  The square-and-cube criterion proves that
$C$ is seminormal.  The second equality follows from the definition of
$\mathfrak c$.

The remaining descent data are the embedded boundary subring
$A/\mathfrak c\subset B/\mathfrak c$ and the possibly nonsplit extension
\eqref{v7:descent:extension}.

For an aperiodic nonempty mask $S\subset C_d$, let $C_S$ be the integer
matrix in \eqref{v6:spec:circulant}.  Define the finite set
\begin{equation}\label{v7:descent:bounded-kernel}
 \mathcal K_m(S)=\{v\in\ker_{\mathbb Z}C_S:
       \|v^+\|_1=\|v^-\|_1\le m\}.
\end{equation}
The equality of the two masses follows from the nonzero column sum $|S|$.

\begin{theorem}[Bounded relations and seminormalization]
\label{v7:descent:seminormal-classification}
Fix $d$, $\zeta$, and $N=\sym^m\PP^1$.  For nonempty aperiodic masks
$S,T\subset C_d$, the following are equivalent:
\begin{enumerate}[label=\textup{(\roman*)},leftmargin=2em]
\item $\mathcal K_m(S)=\mathcal K_m(T)$;
\item $\nu_S$ and $\nu_T$ have the same geometric point fibers on $N$;
\item $Y_{S,m}^{\mathrm{sn}}$ and $Y_{T,m}^{\mathrm{sn}}$ are isomorphic
compatibly with this fixed normalization.
\end{enumerate}
There is no stable-range assumption on $m$.
\end{theorem}
\begin{proof}
\emph{Point relations.}
On a nonzero $\mu_d$-orbit, two multiplicity vectors give the same
output exactly when $C_S(n-n')=0$.  Their difference belongs to
\eqref{v7:descent:bounded-kernel}; zero and infinity counts are separately
recovered by dividing output multiplicities by $|S|$.  This proves
(i)$\Rightarrow$(ii).  Conversely, realize $v^+,v^-$ in one nonzero orbit
and pad with $m-\|v^+\|_1$ zero parameters.  Equality of the fibers forces
$C_Tv=0$, and interchanging the masks proves (i).  The argument works
after every algebraically closed extension of $k$.

\emph{Seminormal descent.}
For (ii)$\Rightarrow$(iii), the finite surjections
$N\to Y_S'=Y_{S,m}^{\mathrm{sn}}$ and $N\to Y_T'=Y_{T,m}^{\mathrm{sn}}$
have the original point fibers, since seminormalization is subintegral.
Take the reduced image $Z$ of $N\to Y_S'\times Y_T'$.
Each projection is surjective with singleton geometric fibers: the other
map is constant on each normalization fiber.  Properness and
quasi-finiteness make the projections finite; characteristic zero then
makes their residue extensions trivial, so they are subintegral.
A reduced scheme finite and subintegral
over a seminormal scheme equals that scheme \cite[Tag 0H3G]{Stacks}.
Both projections are therefore isomorphisms, compatibly with $N$.
Conversely, (iii) equates the seminormal and hence original fibers.
Uniqueness holds on the dense normalization-isomorphism open.
\end{proof}

This classification retains the specified normalization and forgets
polarizations.  For mixed profiles, global multiplicities of each type
also constrain which choices on different output orbits are feasible.

\subsection{Finite-output recovery and spectral semigroup descent}
\label{r21:cond:semigroup-route}

Fix a homogeneous aperiodic mask $S$ of weight $w$.  Write
$\lambda_j=\sum_{a\in S}\zeta^{aj}$ and retain the exact Newton algebra
\begin{equation}\label{v6:cond:newton-algebra}
 A_{S,m}=k[\lambda_jP_j:j\ge1]\subset B=k[e_1,\ldots,e_m],
 \qquad P_j=\sum_{i=1}^m u_i^j.
\end{equation}
Only the first $mw$ output Newton sums are needed to generate this ring.
Assume $m\ge2$ and
\begin{equation}\label{v7:cond:spectral-hypothesis}
 \lambda_j\ne0\quad\text{for every positive integer }j\text{ with }m\nmid j.
\end{equation}
This is a finite periodic condition, checked modulo $\operatorname{lcm}(d,m)$.
If a zero coefficient exists, periodicity in $d$ forces $m\mid d$.
Put
\begin{equation}\label{v7:cond:general-semigroup}
 q=e_m,\quad J=(e_1,\ldots,e_{m-1})B,\quad
 \Gamma=\langle a\ge1:\lambda_{ma}\ne0\rangle,\quad
 g=\gcd\Gamma,\quad\Delta=g^{-1}\Gamma,\quad
 \delta=\#(\mathbb N\setminus\Delta).
\end{equation}
Here $0\in\mathbb N$ and $c(\Gamma)=\min\{c:c+\mathbb N\subset\Gamma\}$
when $g=1$, with $c(\mathbb N)=0$.

\begin{theorem}[Spectral semigroup descent]\label{v7:cond:semigroup}
Under \eqref{v7:cond:spectral-hypothesis},
\begin{align}
 A_{S,m}&=k[q^\Gamma]+J,
 &A_{S,m}^{\mathrm{sn}}&=k[q^g]+J,\label{v7:cond:semigroup-algebra}\\
 (A_{S,m}:B)&=
 \begin{cases}J,&g>1,\\(J,q^{c(\Gamma)})B,&g=1.\end{cases}
 &&\label{v7:cond:semigroup-conductor}
\end{align}
For the entire projective image $Y=Y_{S,m}$,
\begin{equation}\label{v7:cond:semigroup-hilbert}
 P_Y(n)=\binom{wn+m}{m}-w(1-g^{-1})n-2\delta.
\end{equation}
The variety is seminormal exactly when $\Gamma=g\mathbb N$.
Let $\Lambda\subset\PP^m$ be the line in the normalization
on which all middle coefficients vanish.
At a nonnormal closed point $y\in Y$, the local depth is one
if $\Delta\ne\mathbb N$ and $y$ is the image of an endpoint
of $\Lambda$; in every other case it is $\min(m,2)$.  Consequently, for $m=2$ local
Cohen--Macaulayness is equivalent to seminormality, whereas for $m\ge3$
the image is everywhere locally Cohen--Macaulay exactly when
$\Gamma=\mathbb N$, that is, when $Y\simeq\PP^m$.
\end{theorem}
\begin{proof}
\emph{Boundary restriction.}
Put $A=A_{S,m}$ and $R=mw$.  Proposition~\ref{v6:spec:monic-algebra}
gives $A=k[Q_1,\ldots,Q_R]$, with $Q_j=\lambda_jP_j$ and all higher
outputs in this algebra.  Modulo $J$,
$P_{ma}=m(-1)^{a(m-1)}q^a$ and $P_j=0$ for $m\nmid j$.  Thus, before
assuming $J\subset A$, the image of $A$ in $B/J$ is $k[q^\Gamma]$, where
\[
 \Gamma=\langle a:1\le a\le w,\ \lambda_{ma}\ne0\rangle.
\]
The last output coefficient, not necessarily $Q_R$, is a nonzero
multiple of $q^w$.  Hence $w\in\Gamma$, $g\mid w$, and
$b=\min(\Gamma\setminus\{0\})$ satisfies $1\le b\le w$.
Minimality of $b$ gives $\lambda_{mb}\ne0$.

\emph{Finite-output recovery.}
Give $e_i$ weight $i$.  Expanding
$-z\,\frac{d}{dz}\log(1-e_1z+\cdots+(-1)^mqz^m)$ yields
\[
 P_{am+j}=(-1)^{a(m-1)+j-1}(am+j)e_jq^a
       +q^aU_{a,j}+\sum_{v<a}q^vV_{a,j,v}
 \qquad(a\ge0,\ 1\le j<m),
\]
where $U_{a,j}\in k[e_1,\ldots,e_{j-1}]$ and
$V_{a,j,v}\in k[e_1,\ldots,e_{m-1}]$ have zero constant term.
The coefficient of $q^a$ has weight $j$ and the displayed $e_j$ term;
lower coefficients have positive weight $m(a-v)+j$.
Newton induction at $a=0$ recovers the $e_j$.  Now induct on $a<b$,
then on $j$.  Each monomial of $q^vV_{a,j,v}$ uses an already recovered
$e_iq^v$, and each of $q^aU_{a,j}$ uses $e_iq^a$ with $i<j$.
Subtracting them recovers $e_jq^a$ from the supported output at
$am+j\le mb-1<R$.  The top term of $P_{mb}$ is
$m(-1)^{b(m-1)}q^b$; its other terms use the recovered $e_iq^v$, $v<b$.
Since $\lambda_{mb}\ne0$ and $mb\le R$, this gives $q^b\in A$.
For $a=kb+r$, $0\le r<b$, we now have
$e_jq^a=(q^b)^ke_jq^r\in A$.  Thus $J\subset A$, and the computed
boundary image proves
$A=k[q^\Gamma]+J$: lift a boundary monomial to $A$ and subtract its
error in $J\subset A$.

\emph{Conductor and seminormalization.}
The conductor reduces modulo $J$ to that of $k[q^\Gamma]\subset k[q]$.
For $g>1$ it is zero, since multiplication by $q$ moves every exponent
out of $g\mathbb N$.  For $g=1$ it is $(q^{c(\Gamma)})$: the condition
for $q^a$ is $a+\mathbb N\subset\Gamma$, and distinct exponents cannot
cancel under translation.  This proves \eqref{v7:cond:semigroup-conductor}.
The ring $A'=k[q^g]+J$ is seminormal.  For $F\in\operatorname{Frac}(B)$
with $F^2,F^3\in A'$, normality puts $F\in B$, and reduction modulo $J$ lies in
$k[q]\cap k(q^g)=k[q^g]$ (with zero treated separately).
The finite extension $A\subset A'$ is subintegral: off $J$ both equal
$B$, and on $J$ it is the subintegral normalization
$k[(q^g)^\Delta]\subset k[q^g]$ of a numerical-semigroup curve.
In the coordinate $q^g$, the fraction fields agree, the map is an
isomorphism away from zero, and the origin has one preimage with residue
field $k$.  This proves the seminormalization formula.

\emph{Projective boundary and polarization.}
Let $\Lambda\simeq\PP^1\subset\PP^m$ be the line with all middle
coefficients zero, with monic ideal $J$.  The reciprocal chart has
the same supported indices by cyclotomic conjugacy.  Points omitted by
both charts have parameters at both zero and infinity.  Hensel
factorization separates finite and infinite output factors; their
input sizes are their output degrees divided by $w$, hence both less
than $m$.  The required Fourier indices lie in $\{1,\ldots,m-1\}$,
so both truncated maps are isomorphisms by
Corollary~\ref{v6:norm:homogeneous-threshold}.  The whole fiber is a
singleton, and Theorem~\ref{v6:spec:completed-model} identifies the
completed image with the regular source completion.  Faithfully flat
completion proves smoothness on the omitted locus, including repeated
finite roots.  On the two standard charts $J\subset A$ gives an
isomorphism off $\Lambda$.

Let $C$ be the reduced image of $\Lambda$ and $\rho:\Lambda\to C$.
On the overlap invert the last output coefficient, a nonzero multiple
of $q^w$.  The reciprocal coordinates $q'=q^{-1}$,
$e_j'=e_{m-j}/q$ identify the middle ideals, and the Laurent equality
below identifies the boundary subrings.  The affine fiber products
therefore glue, together with the omitted smooth locus, to
\begin{equation}\label{v6:cond:surface-fiberproduct}
 \mathcal O_Y=\nu_*\mathcal O_{\PP^m}
       \mathop{\times}_{\rho_*\mathcal O_\Lambda}\mathcal O_C,
\end{equation}
and
\begin{equation}\label{v6:cond:surface-quotient}
 0\longrightarrow\mathcal O_Y\longrightarrow\nu_*\mathcal O_{\PP^m}
 \longrightarrow\rho_*\mathcal O_\Lambda/\mathcal O_C\longrightarrow0.
\end{equation}
Sheaves on $C$ are pushed forward to $Y$.  This fiber product along
$J$ is not necessarily the actual conductor square when $g=1$.
Since $\Delta$ is numerical and $w/g\in\Delta$,
\[
 k[q^\Gamma,q^{-w}]=k[q^g,q^{-g}].
\]
For every integer $n$, choose $k\ge0$ with $n+k(w/g)\in\Delta$;
then $q^{gn}=q^{g(n+k(w/g))}(q^{-w})^k$, proving the Laurent equality.
Thus $C$ has normalization $\PP^1$, no torus defect, and semigroup
$\Delta$ at each endpoint; $\rho$ has degree $g$ and $p_a(C)=2\delta$.
The actual bundle $\mathcal M=\mathcal O_Y(1)|_C$ pulls back with
degree $w$ on $\Lambda$ and $w/g$ on the normalization of $C$.
Hence $\chi(C,\mathcal M^n)=(w/g)n+1-2\delta$.
Twisting \eqref{v6:cond:surface-quotient} subtracts
$(wn+1)-((w/g)n+1-2\delta)$ from
$\chi(\PP^m,\mathcal O(wn))$, proving
\eqref{v7:cond:semigroup-hilbert} with its actual polarization.

\emph{Quotient module and depth.}
The ideal $J$ annihilates $B/A=k[q]/k[q^\Gamma]$, and as a
$k[q^\Gamma]$-module
\[
 B/A=\bigl(k[q^g]/k[q^\Gamma]\bigr)
       \oplus\bigoplus_{r=1}^{g-1}q^r k[q^g].
\]
The first summand has length $\delta$ at the endpoint and vanishes
off it, giving depth zero exactly when $\Delta$ has gaps.
If $\Gamma=g\mathbb N$ and $g>1$, the quotient is free of rank $g-1$
over $k[q^g]$ and has depth one.  After inverting $q^w$ the same free
quotient occurs for $g>1$ and the quotient vanishes for $g=1$.
At a closed point the normalization module has depth $m$: by finiteness,
image parameters remain a system of parameters in each local
normalization ring and form a regular sequence there.
Completion retains their entire product.  The depth lemma \cite[Tag 00LX]{Stacks}
applied to $0\to A\to B\to B/A\to0$ gives depth one when the
finite-length summand is nonzero, and $\min(m,2)$ at the other
nonnormal closed points.  If $\Gamma=\mathbb N$, then $A=B$.
The reciprocal chart and the omitted smooth locus finish the global
seminormality and Cohen--Macaulay assertions.
\end{proof}

\subsection{Surface consequences and separating examples}
\label{r21:cond:surface-consequences}

The surface specialization supplies the algebras used in
Section~\ref{v7:sec:separations}.

\begin{corollary}[The surface specialization]\label{v6:cond:surface}
For $m=2$, write $B=k[p,q]$.  Condition
\begin{equation}\label{v6:cond:odd-condition}
 \lambda_{2j+1}\ne0\quad(j\ge0)
\end{equation}
gives, with
\begin{equation}\label{v6:cond:semigroup}
 \Gamma=\langle j\ge1:\lambda_{2j}\ne0\rangle,\quad
 g=\gcd\Gamma,\quad\delta_{\Gamma'}=\#(\mathbb N\setminus g^{-1}\Gamma),
\end{equation}
\begin{align}
 A_{S,2}&=k[q^\Gamma]+pB,\label{v6:cond:surface-algebra}\\
 (A_{S,2}:B)&=\begin{cases}pB,&g>1,\\(p,q^{c(\Gamma)})B,&g=1,\end{cases}
 \label{v6:cond:surface-conductor}\\
 P_Y(n)&=\binom{wn+2}{2}-w(1-g^{-1})n-2\delta_{\Gamma'},
 \label{v6:cond:surface-hilbert}\\
 Y\text{ locally Cohen--Macaulay}&\ \Longleftrightarrow\
 Y\text{ seminormal}\ \Longleftrightarrow\ \Gamma=g\mathbb N.
 \label{v6:cond:surface-equivalences}
\end{align}
\end{corollary}
\begin{proof}
Here $J=pB$, so each statement is the corresponding specialization of
Theorem~\ref{v7:cond:semigroup}.
\end{proof}

\begin{example}[A bijective normalization with two isolated defects]
\label{v6:cond:isolated}
Take $d=12$ and $S=\{0,2,3,4,6\}$.  Among the cyclotomic factors of
$T^{12}-1$, only $\Phi_6$ divides its mask polynomial.  Thus
$\Gamma=\langle2,3\rangle$ and
\[
 A=k[p,pq,q^2,q^3],\quad(A:B)=(p,q^2)B,\quad
 P_Y(n)=\binom{5n+2}{2}-2.
\]
The normalization is bijective, but the two endpoint local rings have
depth one and are not seminormal.  At each endpoint $B/A$ has basis
$[q]$ and length one, whereas $B/(A:B)$ has length two.
\end{example}

\begin{corollary}[Unbounded boundary-cover degree for surfaces]
\label{v6:cond:halfmask}
For $t\ge3$, the mask $S=\{0,\ldots,t-1\}\subset C_{2t}$ repeated twice
has
\[
 A=k[q^t]+pB,\quad(A:B)=pB,\quad
 \mathcal Q=\mathcal O_C(-1)^{\oplus(t-1)},\quad
 P_Y(n)=\binom{tn+2}{2}-(t-1)n.
\]
Here $C$ is a line and the normalization line covers it with degree $t$.
The mask $S'=\{0,\ldots,t-2,t\}$ repeated twice instead has image $\PP^2$,
with the same normalization polarizations $(tH,2tH)$ and mixed degrees
$(4t^2,2t^2,t^2)$.
\end{corollary}
\begin{proof}
For $S$, geometric sums give support at the odd integers and multiples
of $2t$, hence $\Gamma=t\mathbb N$.  Corollary~\ref{v6:cond:surface}
gives the algebra, conductor, and Hilbert polynomial.  The boundary power
map has $\rho_*\mathcal O_{\PP^1}=\mathcal O_C\oplus
\mathcal O_C(-1)^{t-1}$, giving the defect by
\eqref{v6:cond:surface-quotient}.  For $S'$, geometric sums give
\[
 \lambda'_2=1-\zeta^{-2}\ne0,\qquad
 \lambda'_1=\frac{1+\zeta^2}{\zeta(1-\zeta)}\ne0.
\]
Corollary~\ref{v6:norm:homogeneous-threshold} proves projective smoothness.
The polarizations give the stated intersection numbers.
\end{proof}

\subsection{The first singular interval for a two-element mask}
\label{r21:cond:interval-route}

For $\{0,1\}\subset C_{2L}$ the first missing Fourier index is $L$.
Theorem~\ref{v7:cond:semigroup} applies at $m=L$; the remaining interval
$L<m<2L$ requires a separate finite-output calculation.

\begin{theorem}[Pinching in the first singular interval]
\label{v7:cond:first-interval}\label{v6:cond:high-dimensional}
Let $L\ge2$, $0\le b<L$, $m=L+b$, $d=2L$, and use the mask $\{0,1\}$
repeated $m$ times.  For $\zeta$ of order $2L$, let $T$ multiply the
coefficient of $X^{m-j}Z^j$ by $\zeta^j$.  The normalization map is
\begin{equation}\label{v6:cond:high-map}
 \nu:\PP^m\longrightarrow Y,\qquad[f]\longmapsto[fTf].
\end{equation}
Its conductor in the normalization is the reduced ideal of
\[
 C=\{[h(aX^L+cZ^L)]:[h]\in\PP^b,\ [a:c]\in\PP^1\}
       \simeq\PP^b\times\PP^1.
\]
Its image $D$ is also $\PP^b\times\PP^1$, and
$\rho=\nu|_C$ is $\mathrm{id}\times([a:c]\mapsto[a^2:-c^2])$.
Writing $i:D\hookrightarrow Y$,
\begin{equation}\label{v6:cond:high-quotient}
 \mathcal Q\simeq i_*\mathcal O_D(0,-1),\qquad
 \mathcal O_Y(1)|_D=\mathcal O_D(2,1),\qquad
 P_Y(n)=\binom{2n+m}{m}-n\binom{2n+b}{b}.
\end{equation}
The variety $Y$ is seminormal, and its local depth at every conductor
closed point is $b+2$.  It is everywhere locally Cohen--Macaulay exactly
when $L=2$.  For odd $L\ge3$, the mask $\{0,2\}$ gives a smooth image
$\PP^m$ with the same normalized relation cover and mixed degrees
\begin{equation}\label{v6:cond:high-degrees}
 \delta_j=(2L)^{m-j}2^j\qquad(0\le j\le m).
\end{equation}
\end{theorem}
\begin{proof}
\emph{The actual coefficient algebra.}
On the chart $a_0=1$, put $y=Z/X$, $g_j=[y^j](fTf)$, and
$A=k[g_1,\ldots,g_{2m}]$.  Write
\[
 f=h+E+zy^Lh,\quad h=1+\sum_{i=1}^b h_i y^i,\quad z=a_L,
 \quad E=\sum_{j\in\mathcal J}e_jy^j,
 \quad\mathcal J=\{b+1,\ldots,m\}\setminus\{L\}.
\]
These are polynomial coordinates: $h_i=a_i$, $e_j=a_j$ for $b<j<L$,
and $e_{L+i}=a_{L+i}-za_i$ for $1\le i\le b$.
There are $L-1$ variables $e_j$.  Since $\zeta^L=-1$,
\begin{equation}\label{v7:cond:interval-expansion}
 fTf=(h+E)T(h+E)+zy^L(hTE-ThE)-z^2y^{2L}hTh.
\end{equation}
Put $U=h+E=\sum u_jy^j$, so $u_0=1$ and $u_L=0$.
The terms involving $z$ begin in degree at least $m+1$.  Therefore,
for $1\le j\le m$,
\[
 g_j=(1+\zeta^j)u_j+
       \sum_{i=1}^{j-1}\zeta^{j-i}u_i u_{j-i}.
\]
For $j\ne L$ the coefficient $1+\zeta^j$ is nonzero; at $j=L$
there is no unknown $u_L$.  Induction recovers every $h_i,e_j$ in $A$.
For the second induction, set $h_0=1$ and
$r_j=g_{L+j}-[y^{L+j}](UTU)$.  The exact identity is
\[
 r_j=\sum_{\substack{0\le i\le b\\j-i\in\mathcal J}}
       (\zeta^{j-i}-\zeta^i)h_i(ze_{j-i})
       -z^2[y^{j-L}](hTh),
\]
where a coefficient of negative index is zero.  For $j\in\mathcal J$
the $i=0$ term is $(\zeta^j-1)ze_j$, with nonzero coefficient,
and every other $ze$-index is smaller.  Recover these terms first for
$j<L$.  At $j=L$ there is no $e_L$, and the remaining coefficient
of $z^2$ is $-1$, so $z^2\in A$.  The same identity then recovers
all $ze_j$ for $j>L$.  All indices used are at most $L+m\le2m$;
the empty first range when $b=L-1$ causes no exception.
Conversely, \eqref{v7:cond:interval-expansion} expresses every output
coefficient in these generators, proving
\begin{equation}\label{v6:cond:high-ring}
 B=k[\mathbf h,\mathbf e,z],\quad J=(\mathbf e)B,\qquad
 A=k[\mathbf h,\mathbf e,z\mathbf e,z^2]=k[\mathbf h,z^2]+J,
 \qquad(A:B)=J.
\end{equation}
More explicitly, put $\mathcal R=k[\mathbf h,\mathbf e,z^2]$.
The recovered generators give
\[
 B=\mathcal R\oplus z\mathcal R,\qquad
 A=\mathcal R\oplus z(\mathbf e)\mathcal R,\qquad
 B/A=z\bigl(\mathcal R/(\mathbf e)\mathcal R\bigr).
\]
Thus $J\subset A$ and annihilates $B/A$.  Modulo $J$, multiplication by
$z$ takes every nonzero even polynomial out of $A/J$, proving
$(A:B)=J$.  As an $A$-ideal, $J$ is generated by $e_j,ze_j$.
The square-and-cube criterion proves seminormality: an element of
$\operatorname{Frac}(A)$ whose square and cube are in $A$ first belongs
to the normal ring $B$; its reduction modulo $J$ is fixed by $z\mapsto-z$
because its square and cube are fixed, and hence is even.  The same
criterion applies to the complete models, including both branches.

\emph{Boundary embedding and point fibers.}
Because $b<L$, the coefficient lists of $hX^L$ and $hZ^L$ are disjoint.
Their products with $a,c$ give a Segre closed immersion
$\PP^b\times\PP^1\hookrightarrow\PP^m$, identifying $C$ with this
smooth product, including $a=0$ and $c=0$.  On it the output is
\[
 hTh(a^2X^{2L}-c^2Z^{2L}).
\]
The map $[h]\mapsto[hTh]$ is a closed immersion with image $\PP^b$:
for $b>0$ use Corollary~\ref{v6:norm:homogeneous-threshold}, since
$b<L$, and for $b=0$ it is a point.  The two shifted output lists are
again disjoint, since $2b<2L$.  This identifies the reduced image $D$
with the stated product, its actual polarization with $\mathcal O_D(2,1)$,
and $\rho$ with the displayed double cover.

For a nonzero $\mu_{2L}$-orbit the integer kernel of $C_S=I+P$ is
$\mathbb Z\varepsilon$, where $\varepsilon_c=(-1)^c$.
Two feasible configurations differ by $a\varepsilon$, $a\in\mathbb Z$.
A nonzero difference needs at least $L$ parameters, and $|a|\ge2$
needs at least $2L$.  Both changes $+\varepsilon$ and $-\varepsilon$
cannot be feasible at one configuration of size less than $2L$.
Thus every such orbit fiber has at most two points; it has two precisely
when one full parity class is present.  The replacement exchanges the
two degree-$L$ binomial factors and leaves the residual divisor fixed.
Every proper subset of the columns of $I+P$ is independent, since its
one-dimensional kernel has full support.  Hence all nonzero-orbit
branches here are immersive, even with repeated input roots.
At zero and infinity the first missing Fourier index is $L$.
Corollary~\ref{v6:spec:conductor-locus} identifies the nonnormal support
with $D$ and its inverse image with $C$; the ideal is determined next.

\emph{The conductor ideal from the whole fiber.}
Group the output roots by \emph{whole} nonzero $\mu_{2L}$-orbits,
with zero and infinity as separate groups $\Omega$.
The output multiplicity $2k_\Omega$ recovers each input size, and
$\sum k_\Omega=m$.  The full fiber is the product of the group fibers.
Hensel factorization separates groups in source and target while retaining
every cluster and branch within each group.  The exact moment tuples of
Theorem~\ref{v6:spec:completed-model} use only their group variables, so
the whole completed image is the completed tensor product of the group
image algebras.  In disjoint coefficient variables its kernel is the
sum of the extended group kernels, each computed over the entire group fiber.

Groups with $k_\Omega<L$ give isomorphisms by
Corollary~\ref{v6:norm:homogeneous-threshold}.  Since $m<2L$, at most one
group has size $k=L+b'\ge L$.  With no such group all factors are smooth;
otherwise,
writing $A_{L,b'}$ for the actual algebra
\eqref{v6:cond:high-ring} and $\mathfrak m_*$ for this group's
output point, the full model is
\[
 \widehat{\mathcal O}_{Y,y}\simeq
 \widehat{(A_{L,b'})_{\mathfrak m_*}}
 [[v_1,\ldots,v_{b-b'}]].
\]
At infinity use the reciprocal algebra.  Completion at the output point
retains every normalization preimage; the other groups supply
$m-k=b-b'$ regular parameters.  The source conductor is the reduced ideal
$(\mathbf e)$ in this model.  Its support is the already identified
smooth $C$, so it is the completed ideal of $C$.
The construction includes simultaneous zero and infinity inputs and
residual roots colliding inside the large orbit: all coefficient
variables are retained.
For a finite module $Q$, its annihilator is the kernel of the map
$R\to Q^{\oplus n}$ defined by a finite generating list.  Flat extension
preserves this kernel.  Applied to the normalization quotient, this
justifies completion and adjoining the regular parameters for the
actual conductor ideal.  Faithful flatness gives $\mathcal I_C$;
the same local models prove seminormality and smoothness off $D$.

\emph{Defect, polarization, and depth.}
With the reduced conductor sides $C,D$ now identified,
\eqref{v7:descent:classical} gives the exact sequence
\cite[Lemma 1.3]{Ferrand}
\[
 0\longrightarrow\mathcal O_Y\longrightarrow\nu_*\mathcal O_{\PP^m}
 \longrightarrow i_*(\rho_*\mathcal O_C/\mathcal O_D)
 \longrightarrow0.
\]
For the double cover of the second factor,
$\rho_*\mathcal O_C=\mathcal O_D\oplus\mathcal O_D(0,-1)$.
Its odd generators $z,z^{-1}$ have transition $z^2$.
The restriction $\mathcal O_Y(1)|_D=\mathcal O_D(2,1)$ gives
$\mathcal Q(n)=i_*\mathcal O_D(2n,n-1)$, of Euler characteristic
$n\binom{2n+b}{b}$.  Since $\nu^*\mathcal O_Y(1)=\mathcal O_{\PP^m}(2)$,
the exact sequence proves \eqref{v6:cond:high-quotient}.

At a conductor closed point, the finite normalization module has depth
$m$: its completion is a product of regular local rings of dimension
$m$.  The quotient is invertible on the smooth $(b+1)$-fold $D$, so
has depth $b+1$.  The depth lemma \cite[Tag 00LX]{Stacks} gives depth
$b+2$ when $L>2$, and full depth $m=b+2$ when $L=2$.
For odd $L$, $\zeta^2$ has odd order, so $1+\zeta^{2j}$ never vanishes.
The comparison mask has smooth projective image by
Corollary~\ref{v6:norm:homogeneous-threshold}.  Both masks have trivial
rotation stabilizer and the same normalized relation morphism, taking
$2L$th powers of the input parameters.  Their two pullbacks are
$2H$ and $2LH$, giving \eqref{v6:cond:high-degrees}.
\end{proof}

\begin{remark}[Completed local forms in the first singular interval]
\label{v6:cond:high-local}
For $b=0$, the theorem has conductor line and quotient
$\mathcal O_D(-1)$.  Away from the two branch values the completed ring is
\[
 k[[s,x_1,\ldots,x_{L-1},y_1,\ldots,y_{L-1}]]/(x_iy_j:1\le i,j<L).
\]
At a branch value it is
\[
 \frac{k[[t,x_1,\ldots,x_{L-1},y_1,\ldots,y_{L-1}]]}
 {(x_iy_j-x_jy_i,\ y_iy_j-tx_ix_j:1\le i,j<L)}.
\]
Here $t=z^2$, $y_i=x_iz$.  The first ring is the fiber product of two
smooth branches over their common line.  For the second, put
$R=k[[t,\mathbf x]]$.  Products of two $y$'s reduce to $R$, and the
relations $x_iy_j-x_jy_i$ are precisely the Koszul syzygies of
$(x_1,\ldots,x_{L-1})$.  Thus its image is exactly
$R\oplus z(\mathbf x)R\subset R\oplus zR$, with $z^2=t$;
so the displayed relations generate the kernel.  For $L=2$ the second is the pinch-point
equation $y^2=tx^2$.

For every $0\le b<L$, the same two rings with $b$ independent formal
variables adjoined describe all conductor closed points.  The first
case is $ac\ne0$ and the second is $ac=0$ in the binomial parameter
of $C$.  Indeed, the unique large group in the proof has size $L+b'$,
contributing $b'$ residual coefficient variables, and the other groups
contribute $m-(L+b')=b-b'$ regular variables.  The free $\mathbf h$
coordinates in \eqref{v6:cond:high-ring} remain independent when residual
roots collide with binomial roots, so these are all completed local
types throughout $L\le m<2L$.
\end{remark}

\section{One-jet descent for singleton--complement profiles}
\label{v7:sec:mixed-trace}

The singleton--complement image requires a one-jet condition beyond
pointwise compatibility.  We prove finite-output generation before
projective descent, separately for $\ell<d$, $\ell=d$ and $\ell=d+1$.
Throughout, $k$ is algebraically closed of characteristic zero,
$d\ge3$, and $1\le\ell\le d$ unless another range is stated.

The section has three intrinsic regimes.  For $1\le\ell<d$, the
actual image is described by a finite free one-jet descent algebra; at
$\ell=d$ the same first-order descent condition persists, but the natural
finite projection changes and $t$ itself no longer descends.  The boundary
normality threshold in Proposition~\ref{v8:mixed:boundary-threshold} shows
that $\ell=d$ is the last value for which the reduced boundary is normal.
We therefore use the difference-kernel calculation only as a lifting and
specialization device, not as a substitute for finite-output generation.
At $\ell=d+1$, Theorems~\ref{v8r9:mixed:first-excluded}
and~\ref{v8r11:mixed:projective-first} give a separate finite-output
certificate and projective conductor calculation across the first
nonnormal-boundary layer.  This is the endpoint of the finite-output and
projective-descent range proved here.

\subsection{The family and its finite output algebra}
\label{r22:mixed:finite-output}

Use the masks $\{0\}$ and $C_d\setminus\{0\}$ with multiplicities
$1,\ell$.  The product-normalization theorem
\cite[Theorem~5.4]{LiProfileGeometry} gives the normalization of the
reduced image $Y=Y_{d,\ell}$:
\[
 \nu:N=\PP^1\times\PP^\ell\longrightarrow Y,
 \qquad\nu^*\mathcal O_Y(1)=\mathcal O_N(1,d-1).
\]
The profile lies in the Krylov range for
$I=d\{0,\ldots,\ell+1\}$ and $r=1+\ell(d-1)$.
On the finite-root chart write $f(x)=\prod_{i=1}^\ell(x-v_i)$.  The output
polynomial and its Newton sums are
\begin{align}
 G(x)&=(x-u)\prod_{i=1}^\ell\prod_{a=1}^{d-1}(x-\zeta^a v_i),
 \label{v7:mixed:output}\\
 Q_j&=u^j+a_jP_j(f),\qquad
 a_j=\begin{cases}-1,&d\nmid j,\\d-1,&d\mid j.\end{cases}
 \label{v7:mixed:newton}
\end{align}
Its actual monic image algebra is $A=k[Q_1,\ldots,Q_r]$;
all higher $Q_j$ follow from the output Newton recurrence.

Divide once by $x-u$:
\begin{equation}\label{v7:mixed:division}
 f(x)=(x-u)h(x)+t,\qquad t=f(u),\qquad
 h(x)=x^{\ell-1}+h_1x^{\ell-2}+\cdots+h_{\ell-1}.
\end{equation}
This is a polynomial coordinate change on the normalization, giving
$B=k[\mathbf h,t,u]$.  For $\ell=1$ put $h=1$.
In the range $\ell\le d$, for $j<\ell$ we have $Q_j=-P_j(h)$, so $\mathbf h\in A$.
If $\ell<d$, the additional identity
$Q_\ell=-P_\ell(h)+\ell t$ recovers $t\in A$.
For this subcritical range put
\begin{equation}\label{v7:mixed:noether-ring}
 c=d-\ell,\qquad
 w=\frac{Q_d-(d-1)P_d(h)}d,\qquad R=k[\mathbf h,t,w].
\end{equation}
The polynomial $w$ is monic of degree $d$ in $u$, with all other
$u$-degrees smaller.  Thus $R$ is a polynomial ring and
$B$ is free over $R$, with basis $1,u,\ldots,u^{d-1}$.

Let $H_k(u,h)$ be the complete homogeneous symmetric function of degree
$k$ in $u$ and the roots of $h$.  Equivalently, it is the quotient in the
division of $x^{k+\ell-1}$ by $h(x)$, evaluated at $x=u$.  Therefore
\begin{equation}\label{v7:mixed:division-remainder}
 h(u)H_k(u,h)=u^{k+\ell-1}+r_k(u),\qquad\deg_u r_k<\ell-1.
\end{equation}
For $\ell=1$, $r_k=0$.  Set $H_k=0$ for $k<0$.
For $k\ge0$ with $d\nmid\ell+k$, put
\begin{equation}\label{v7:mixed:bk}
 b_k=\frac{Q_{\ell+k}+P_{\ell+k}(h)}{\ell+k}.
\end{equation}

\subsection{The actual algebra below the critical multiplicity}
\label{r22:mixed:subcritical}

Over $R_0=k[\mathbf h,s]$ consider
$E=R_0[u]/(u^d-s)$.  Define
\begin{equation}\label{v7:mixed:trace}
 \tau_h(p)=[u^{d-1}]\bigl(h(u)p(u)\bmod(u^d-s)\bigr).
\end{equation}
The ring $R_0$ describes restriction to $t=0$; it is not an assumed
subring of $A$.  Put
$\varepsilon_s(p)=[u^{d-1}](p\bmod(u^d-s))$, so
$\tau_h(p)=\varepsilon_s(hp)$.  In the power basis, the matrix of
$\varepsilon_s(pq)$ is anti-diagonal with entries one, hence is perfect
even at $s=0$.  This does not assert perfection of the weighted pairing
$\tau_h(pq)$.  Both functionals are $R_0$-linear; neither is being
identified with ordinary algebra trace.  For $k\ge0$, formula \eqref{v7:mixed:division-remainder} gives
\begin{equation}\label{v7:mixed:trace-H}
 \tau_h(H_k)=
 \begin{cases}s^{q-1},&\ell+k=dq,\\0,&d\nmid\ell+k.\end{cases}
\end{equation}
Among the triangular basis $H_0,\ldots,H_{d-1}$, only $H_c$ has
nonzero $\tau_h$-value, equal to one, where $c=d-\ell$.  Thus $\tau_h(1)=0$
for $\ell<d$ and $\tau_h(1)=1$ for $\ell=d$.

\begin{theorem}[The mixed one-jet descent algebra]
\label{v7:mixed:affine}
Assume $1\le\ell<d$.  With the notation above, the actual monic
image algebra is the free $R$-module
\begin{equation}\label{v7:mixed:free-basis}
 A=R\ \oplus\!
 \bigoplus_{\substack{1\le k<d\\k\ne c}}Rb_k
 \ \oplus R(t^2u^c).
\end{equation}
For $F=F_0(\mathbf h,u)+tF_1(\mathbf h,u)\pmod{t^2B}$, membership in $A$
is equivalent to
\begin{equation}\label{v7:mixed:jet}
 F_0=H(\mathbf h,u^d)\text{ for some }H\in R_0,
 \qquad \tau_h(F_1)=-(d-1)\partial_s H.
\end{equation}
Moreover,
\begin{equation}\label{v7:mixed:affine-conductor}
 (A:B)=t^2B,\qquad
 A^{\mathrm{sn}}=k[\mathbf h,u^d]+tB,\qquad
 A^{\mathrm{sn}}/A\simeq R_0.
\end{equation}
The last module is generated by $[tu^c]$.  The algebra $A$ is
Cohen--Macaulay everywhere on this chart.
\end{theorem}
\begin{proof}
\textit{The triangular module.}
Comparing the Laurent coefficients of
\[
 \log f=\log((x-u)h)+
 \sum_{a\ge1}\frac{(-1)^{a-1}}a\frac{t^a}{((x-u)h)^a}
\]
gives
\begin{equation}\label{v7:mixed:triangular-moments}
 P_j(f)=u^j+P_j(h)-jtH_{j-\ell}(u,h)+O(t^2),\qquad
 b_k=tH_k+t^2L_k,\quad\deg_u L_k\le k-\ell<k.
\end{equation}
These are polynomial identities; the Laurent expansion is a device for
extracting their coefficients.  In particular, $b_k/t$ is monic of
$u$-degree $k$.

Let $M$ denote the right side of \eqref{v7:mixed:free-basis}.
Triangular induction on $j$ gives $t^2u^j\in M$ for $0\le j<d$:
use $tb_j$ except at $j=c$, which is already a basis element.
Consequently $t^2B\subset M$.  The nonconstant displayed basis elements
all belong to $tB$, so their products lie in $t^2B$; hence $M$ is an
$R$-subalgebra.  In the $u$-basis its determinant is $\pm t^d$, proving
$R$-linear independence.

\textit{Actual output membership.}
All $b_k$ used in the displayed basis come from
the finite outputs, since
$r-(\ell+d-1)=(\ell-1)(d-2)\ge0$.
We prove $M\subset A$, including $c=1$.
If $c\ge2$, triangular induction first recovers $t^2u^j$ for $j<c$ in
$A$, and
\[
 b_1b_{c-1}=t^2u^c+t^2\cdot
          (\text{a polynomial of $u$-degree smaller than }c)
\]
recovers the missing generator.  This includes $c=2$.
If $c=1$, then $\ell=d-1\ge2$.  Write $\gamma=H_1(h)$ for the first
complete homogeneous function in the roots of $h$ alone.  The exact
identities from the same Laurent expansion are
\begin{align*}
 w&=u^d-(d-1)t(u+\gamma),&
 w_0&=w+(d-1)t\gamma=u^d-(d-1)tu,\\
 b_k&=tH_k\quad(2\le k\le d-2),&
 b_{d-1}&=tH_{d-1}-\tfrac12t^2,\\
 b_d&=tH_d-t^2(u+\gamma).&&
\end{align*}
The last moment is $Q_{2d-1}$, which occurs among
the finite output Newton generators since
$r-(2d-1)=(d-1)(d-3)\ge0$.
Division by the monic triangular basis yields
\[
 H_d-u^d=\alpha_0+\sum_{k=2}^{d-1}\alpha_kH_k,
       \qquad\alpha_k\in k[\mathbf h].
\]
There is no $H_1$ term: its coefficient is detected by
\eqref{v7:mixed:trace-H} at $s=0$, where $\tau_h(H_d)=0$ and
$u^d=0$.  Therefore
\[
 b_d-\alpha_0t-\sum_{k=2}^{d-1}\alpha_kb_k-tw_0
 =(d-2)t^2u+t^2(\alpha_{d-1}/2-\gamma).
\]
Since $d-2\ne0$, this recovers $t^2u$ and proves $M\subset A$ in all cases.

\textit{The one-jet description and reverse inclusion.}
Modulo $t$, the image of $M$ is $R_0$, with $s=u^d$, and
$w=s-(d-1)tH_c\pmod{t^2B}$.  The first-order parts of $t$ and the
$b_k$ with $k\ne c$ span $t\ker\tau_h$.  Expanding $H(\mathbf h,w)$
therefore gives exactly \eqref{v7:mixed:jet}, since $t^2B\subset M$.
For $d\nmid j$, the zero-order part of $Q_j$ is $-P_j(h)$ and its
first-order part is $jH_{j-\ell}$, with $\tau_h$-value zero.  For $j=dq$,
the zero-order part is $ds^q+(d-1)P_j(h)$ and the $\tau_h$-value of its
first-order part is $-j(d-1)s^{q-1}$.  In both cases \eqref{v7:mixed:jet} holds.
All $Q_j$ belong to $M$, proving $A=M$.

\textit{Conductor, seminormalization, and module action.}
The jet description contains $t^2B$.  Since
$k[\mathbf h,u^d]\subset k[\mathbf h,u]$ has zero conductor,
any conductor element lies in $tB$.  Its first coefficient $p$ must
satisfy $\tau_h(pq)=0$ for every $q\in E$.  The perfect unweighted
pairing gives $hp=0$ in the universal domain $E=k[\mathbf h,u]$;
hence $p=0$, proving $(A:B)=t^2B$.

Put $A'=k[\mathbf h,u^d]+tB$.  Both rings equal $B$ after inverting
$t^2$, and have reduced boundary quotient $R_0$, so the finite
extension $A\subset A'$ is subintegral.  The square-and-cube criterion,
reduced modulo $t$, proves $A'$ seminormal using
$k[\mathbf h,u]\cap k(\mathbf h,u^d)=k[\mathbf h,u^d]$.
The jet functional
$F\mapsto(d-1)\partial_sH+\tau_h(F_1)$ is an $A$-linear surjection
$A'\to R_0$: its relative Leibniz rule uses the action of $A$ on
$R_0$ by restriction to $t=0$.  Its kernel is $A$, and $tu^c$ maps
to $1$.  This proves the quotient assertion without assuming $R_0\subset A$.
These arguments use only the established jet description and
$\tau_h(H_c)=1$, so they also apply once that description is proved
at $\ell=d$.  The displayed free basis proves Cohen--Macaulayness.
\end{proof}

The two conditions in \eqref{v7:mixed:jet} describe cyclic descent and
its transverse first order, including residual collisions and $s=0$.

\subsection{The critical multiplicity and its parameter algebra}
\label{v8:mixed:critical-section}

At $\ell=d$ the value $\tau_h(1)=1$ prevents $t$ itself from descending.
It does not force a higher-order descent condition.  The appropriate
finite projection instead uses $t^2$ as a parameter.

\begin{theorem}[One-jet descent at the critical multiplicity]
\label{v8:mixed:critical}
Assume $\ell=d$.  In the coordinates \eqref{v7:mixed:division} put
\[
 w=\frac{Q_d-(d-1)P_d(h)}d=u^d-(d-1)t,\qquad
 \sigma=t^2,\qquad R^\dagger=k[\mathbf h,w,\sigma].
\]
Then $R^\dagger\subset A$ is a polynomial ring and the actual monic
image algebra is the following free module of rank $2d$:
\begin{equation}\label{v8:mixed:critical-basis}
 A=R^\dagger\ \oplus\!\bigoplus_{j=1}^{d-1}R^\dagger(tH_j)
   \ \oplus\!\bigoplus_{j=1}^{d-1}R^\dagger(\sigma u^j)
   \ \oplus R^\dagger t^3.
\end{equation}
Membership in $A$ is still exactly the one-jet condition
\eqref{v7:mixed:jet}, now with $\deg h=d-1$ and $\tau_h(1)=1$.
Moreover,
\[
 (A:B)=t^2B,\qquad A^{\mathrm{sn}}=k[\mathbf h,u^d]+tB,
 \qquad A^{\mathrm{sn}}/A\simeq R_0,
\]
where the last module is generated by $[t]$.  In particular $A$ is
Cohen--Macaulay, and $t\notin A$ although $t^2,t^3\in A$.
\end{theorem}
\begin{proof}
\textit{Finite output recovery.}
The moments of indices below $d$ recover $\mathbf h$.  For
$1\le j<d$, the Laurent expansion used in
\eqref{v7:mixed:triangular-moments} has no quadratic $t$-term at
index $d+j$, and gives the exact identity $b_j=tH_j\in A$.
Here and below higher Newton sums belong to $A$ by the output Newton
identities, even when an index exceeds $r$.

We first recover $\sigma$ from the output algebra.  Put
\[
 V_0=H_d-u^d,\qquad
 W_2=\frac{Q_{2d}-(d-1)P_{2d}(h)}d.
\]
Formula \eqref{v7:mixed:division-remainder} shows that
$\deg_u V_0\le d-1$, and \eqref{v7:mixed:trace-H} gives
$\tau_h(V_0)=0$.  Thus $V_0$ is a $k[\mathbf h]$-linear combination of
$H_1,\ldots,H_{d-1}$, with no $H_0$ term, and $tV_0\in A$.
The coefficient of $t^2$ in $P_{2d}(f)$ is $d$, so
\begin{equation}\label{v8:mixed:critical-square}
 W_2-w^2=-2(d-1)tV_0-(d-1)(d-2)t^2.
\end{equation}
Since $d\ge3$, this proves $\sigma=t^2\in A$.

For $1\le j<d$, set $V_j=H_{d+j}-u^dH_j$.
Multiplying by $h(u)$ and using
\eqref{v7:mixed:division-remainder} gives $\deg_u V_j\le d-1$;
its $\tau_h$-value is again zero, so $tV_j\in A$.
Define $H_j^{(2)}$ by
$\sum_{j\ge0}H_j^{(2)}T^j=(\sum_{j\ge0}H_jT^j)^2$.
Because $2d+j<3d$, the same logarithmic expansion is exact through
order two and yields
\begin{align}
 b_{d+j}&=tH_{d+j}-\tfrac12t^2H_j^{(2)},\notag\\
 b_{d+j}-wb_j-tV_j
   &=\sigma\bigl((d-1)H_j-\tfrac12H_j^{(2)}\bigr).
 \label{v8:mixed:critical-triangular}
\end{align}
The polynomial in parentheses has $u$-degree $j$ and leading
coefficient $(2d-j-3)/2\ne0$.  Starting with $\sigma$, triangular
induction recovers $\sigma u^j$ for all $1\le j<d$.
Next $b_1b_{d-1}=\sigma H_1H_{d-1}$ is monic of $u$-degree $d$
after the factor $\sigma$ is removed.  Subtracting the already
recovered terms gives $\sigma u^d\in A$.  The identity
\[
 \sigma u^d-w\sigma=(d-1)t^3
\]
then gives $t^3\in A$.

\textit{The parameter ring and free module.}
Start with independent symbols $W,\Sigma$ and put
$T=k[\mathbf h,W,\Sigma]$.  The algebra
\[
 T[t,u]/(t^2-\Sigma,\ u^d-W-(d-1)t)
\]
is free over $T$ with basis $u^j,tu^j$, $0\le j<d$, by successive
monic division.  Eliminating $W,\Sigma$ identifies it with
$k[\mathbf h,t,u]=B$.  Thus $T$ embeds in $B$ as $R^\dagger$, and
this is the asserted polynomial parameter ring and $B$-basis.
Let $M$ be the right side of
\eqref{v8:mixed:critical-basis}.  It contains $\sigma u^j$ and
$\sigma t$; the products $\sigma b_j=t^3H_j$, followed by triangular
induction, give $\sigma tu^j\in M$.  Hence $t^2B\subset M$.
All the displayed nonscalar generators lie in $tB$, so $M$ is an
$R^\dagger$-subalgebra.  In the basis of $B$ its determinant is
$\pm\sigma^d$, proving linear independence.  The preceding recoveries
prove $M\subset A$.

\textit{The jet condition and its consequences.}
Modulo $t^2B$ the scalar subring is obtained by substituting
$w=s-(d-1)t$ into $R_0$.  The other first-order terms span
$t\ker\tau_h$, since $H_1,\ldots,H_{d-1}$ form a basis of that kernel.
Consequently $M$ is precisely the subalgebra defined by
\eqref{v7:mixed:jet}.  For every $Q_j$ the zero-order and first-order
parts satisfy that condition, by \eqref{v7:mixed:trace-H}, exactly as
in the proof of Theorem~\ref{v7:mixed:affine}.  Thus $A\subset M$.

Now $A=M$, $t^2B\subset A$, and $\tau_h(H_0)=1$; the universal
ring $E=k[\mathbf h,u]$ is still a domain.  The final conductor and
seminormalization argument of Theorem~\ref{v7:mixed:affine} therefore
applies, with $[t]$ generating the quotient and the same restriction
action.  The free basis proves Cohen--Macaulayness, whereas
$\tau_h(1)=1$ in the jet test proves $t\notin A$.
\end{proof}

The ranks $d$ and $2d$ belong to the two parameter projections;
both presentations impose the same one-jet order.

\subsection{The projective conductor square and its defect layers}
\label{r22:mixed:projective-route}

The following structural lemma supplies the mixed local model used in
both projective descent theorems.  It is independent of the affine
algebra presentations, which are inserted only after they have been proved.

\begin{lemma}[Completed mixed image over the entire fiber]
\label{r35:mixed:completed-model}
Let $k$ be algebraically closed of characteristic zero and $d\ge3$.
Fix $1\le\ell\le d+1$.  For $0\le m\le\ell$, let $A_{d,m}$ be
the actual monic output algebra for one singleton and $m$ complements;
put $A_{d,0}=k[u]$.
At a closed output point $y$ of $Y_{d,\ell}$, group its roots by entire
nonzero $\mu_d$-orbits, and treat zero and infinity as separate groups.
The output determines a unique group containing the singleton, and
fixes the number $m_\gamma$ of complement parameters in every group.
Write $m=m_*$ for the count in the singleton group.  Then
\begin{equation}\label{r35:mixed:completed-product}
 \widehat{\mathcal O}_{Y_{d,\ell},y}
 \simeq \widehat{(A_{d,m})_{\mathfrak m_*}}
       \widehat\otimes_k k[[z_1,\ldots,z_{\ell-m}]],
\end{equation}
where $\mathfrak m_*$ is the group output point; use the reciprocal
algebra when that group is at infinity.  The isomorphism identifies
the finite normalization algebras, including every point above $y$.
In each normalization factor a local equation of the incidence divisor
($f(u)$ on the monic chart) is a unit times the singleton-group equation,
taken to be $1$ when $m=0$.

For the companion profile with one full mask and
$\ell-1$ complements, the same completed-product
decomposition and identification of the finite
normalization algebras hold.  Here $m$ is the number
of complements in the group containing the full mask;
use that group's actual output algebra and
$\ell-1-m$ regular parameters.  These statements include
repeated roots within a group and simultaneous roots
at zero and infinity.
\end{lemma}
\begin{proof}
For the singleton profile, the total multiplicity in a group is
$\epsilon+(d-1)m_\gamma$, with $\epsilon\in\{0,1\}$.
Reduction modulo $d-1$ recovers $\epsilon$, and then $m_\gamma$;
$d\ge3$ is used here.  Thus $\sum\epsilon=1$ and
$\sum m_\gamma=\ell$ impose no further choices once the output is fixed.
On a nonzero orbit, if the singleton occupies position $a$, the
possible complement multiplicities are precisely the nonnegative
solutions of
\[
 b_c=m-n_c+\mathbf1_{c=a},\qquad \sum_cn_c=m.
\]
All such solutions must be retained.  A group without the singleton has
$b_c=m_\gamma-n_c$, hence a unique allocation; its complement circulant
is $J-I$, whose eigenvalues are $d-1$ and $-1$.
At an endpoint the complement moment multipliers are likewise $d-1$
or $-1$.  Therefore exact power sums and Newton identities recover the
source coefficients in every complement-only group, giving a regular
completed image ring of dimension $m_\gamma$.

Hensel factorization separates the coprime output groups and the
corresponding input factors.  It leaves repeated roots inside their
groups.  On each feasible input configuration, the cluster coefficient
maps depend only on that group's variables.  As in the proof of
Theorem~\ref{v6:spec:completed-model}, finite normalization and flat
completion embed the completed image in
$\prod_{x\in\nu^{-1}(y)}\widehat{\mathcal O}_{N,x}$, and identify it
with the closed algebra generated by the output moment tuples.
Equivalently, its defining ideal is the intersection of the coefficient
kernels for \emph{all} feasible configurations.  Here those configurations
form the product of the group configurations, since their counts have
already been fixed.  Let $R_*$ be the completed ambient coefficient ring
of the singleton group, with kernels $I_\alpha$ for its feasible
configurations.  Quotienting by the complement-only group kernels leaves
$R_*[[\mathbf z]]$, and the kernel for a full configuration is
$I_\alpha R_*[[\mathbf z]]$.  Flatness of this power-series extension
and finiteness of the configuration set give
\[
 \bigcap_\alpha I_\alpha R_*[[\mathbf z]]
 =\left(\bigcap_\alpha I_\alpha\right)R_*[[\mathbf z]].
\]
Thus the full image ring is
$(R_*/\bigcap_\alpha I_\alpha)[[\mathbf z]]$, with $\ell-m$ regular
parameters.  This proves \eqref{r35:mixed:completed-product}
as an identity of embedded completed algebras.  Completing at one chosen
normalization point would replace the kernel intersection by one kernel
and does not prove this identity.  Evaluating all input factors outside
the singleton group at $u$ gives units, proving the incidence assertion.

For the full-mask profile the group total is
$d\epsilon+(d-1)m_\gamma$, which again fixes both counts.  On a nonzero
orbit its equation is $b_c=\epsilon+m_\gamma-n_c$; at each endpoint the
same count argument applies.  Hensel factorization and the preceding
coefficient-kernel argument give its stated completed product.
\end{proof}

\begin{theorem}[Projective mixed descent]\label{v7:mixed:projective}
For $1\le\ell\le d$, in $N=\PP^1\times\PP^\ell$ let
\[
 C=\{(u,f):f(u)=0\}\simeq\PP^1\times\PP^{\ell-1},
 \qquad (u,h)\longmapsto(u,(x-u)h).
\]
Its image $D\subset Y$ is also $\PP^1\times\PP^{\ell-1}$, and
$\rho=\nu|_C$ is the $d$th-power map on the first factor and the identity
on the second.  The variety $Y$ is smooth off $D$ and locally
Cohen--Macaulay everywhere.  Its normalization conductor is
\begin{equation}\label{v7:mixed:global-conductor}
 \mathfrak c\mathcal O_N=\mathcal I_C^2=
             \mathcal O_N(-2\ell,-2).
\end{equation}
The seminormalization pinches the reduced $C$ by $\rho$.
Writing $i:D\hookrightarrow Y$, its defect layers are
\begin{equation}\label{v7:mixed:global-layers}
 \mathcal Q_{\mathrm{glue}}\simeq i_*\mathcal O_D(-1,0)^{\oplus(d-1)},
 \qquad
 \mathcal Q_{\mathrm{inf}}\simeq i_*\mathcal O_D(-2,0).
\end{equation}
Their extension \eqref{v7:descent:extension} does not split as an
$\mathcal O_Y$-module extension.  Moreover,
\begin{align}
 \mathcal O_Y(1)|_D&=\mathcal O_D(1,d-1),\label{v7:mixed:boundary-polarization}\\
 P_Y(n)&=(n+1)\binom{(d-1)n+\ell}{\ell}
       -(dn-1)\binom{(d-1)n+\ell-1}{\ell-1}.
 \label{v7:mixed:projective-hilbert}
\end{align}
For $Y^{\mathrm{sn}}$ with the polarization pulled back from $Y$, replace
$dn-1$ in the last formula by $(d-1)n$.
\end{theorem}
\begin{proof}
\textit{The boundary and its polarization.}
The universal root divisor has class $(\ell,1)$ on $N$ and the displayed
product parameterization of $C$.  On it, the singleton together with
its coincident complement supplies a full mask, leaving $\ell-1$
complement parameters.  This latter profile has smooth image: for $\ell=1$ it is the binomial line;
for $\ell>1$, the projective normality criterion
\cite[Theorem~6.2]{LiProfileGeometry} applies because its weighted degree
sets are $\{d\}$ and $\{1,\ldots,\ell-1\}$, the required Fourier
coefficients are nonzero, and
$(\epsilon,k)\mapsto d\epsilon+(d-1)k$, $\epsilon\in\{0,1\}$,
is injective.  Hence this image is exactly the smooth $D$ stated above,
not merely a bijective image of that product.  This also gives
\eqref{v7:mixed:boundary-polarization} and $\rho$.

\textit{The completed image and conductor.}
In a nonzero singleton group, moving the singleton from $a$ to another
position is feasible exactly when $n_a>0$, hence exactly on $C$.
The residual divisor $h$ and $u^d$ remain fixed.  At the endpoints the
fiber is one geometric point, with the ramification of the power map.
Lemma~\ref{r35:mixed:completed-model} gives the completed local rings,
with all feasible preimages retained.  For its mixed factor use
Theorem~\ref{v7:mixed:affine} when $1\le m<d$,
Theorem~\ref{v8:mixed:critical} when $m=d$, and $k[u]$ when $m=0$.

For a finite module $Q$ over a Noetherian local ring $T$, express
$\operatorname{Ann}_T(Q)$ as the kernel of
$T\to Q^{\oplus b}$ using a finite generating list of $Q$.
Flatness then shows that this annihilator commutes with completion
and with adjoining formal parameters.  Apply this to the finite
normalization quotient in the displayed model.  Its conductor is the
square of the singleton-group equation, which differs from $f(u)$ by
a unit.  Faithfully flat completion proves
\eqref{v7:mixed:global-conductor}, smoothness off $D$, and local
Cohen--Macaulayness.

\textit{Seminormalization and the gluing defect.}
Define the finite intermediate algebra on $Y$
\[
 \mathcal S=\nu_*\mathcal O_N
       \mathop{\times}_{\rho_*\mathcal O_C}\mathcal O_D.
\]
The same local models identify it with the seminormalization algebra:
locally it is $k[\mathbf h,u^d]+tB$, with the smooth factors adjoined.
Thus $\mathcal S/\mathcal O_Y$ is an invertible $\mathcal O_D$-module,
by \eqref{v7:mixed:affine-conductor}, Theorem~\ref{v8:mixed:critical},
and faithfully flat completion.
The quotient $\nu_*\mathcal O_N/\mathcal S$ is
$\rho_*\mathcal O_C/\mathcal O_D$, giving the first formula in
\eqref{v7:mixed:global-layers}.

\textit{The infinitesimal line bundle.}
The conormal and relative dualizing bundles on $C$ are
\[
 L_C=\mathcal I_C/\mathcal I_C^2=\mathcal O_C(-\ell-1,-1),
 \qquad \omega_\rho=\mathcal O_C(2d-2,0).
\]
Finite flat duality \cite[Tags~0BUZ and 0BVE]{Stacks} identifies a morphism
$\rho_*L_C\to\mathcal O_D(-2,0)$ with a section of
\[
 L_C^\vee\otimes\omega_\rho\otimes\rho^*\mathcal O_D(-2,0)
       =\mathcal O_C(\ell-1,1).
\]
The evaluation section $h(u)$ gives $\tau_h$ on the monic chart,
up to a nonzero duality-trivialization constant.  The map is fiberwise
surjective: the coefficient pairing on each length-$d$ fiber of $\rho$
is perfect, and the nonzero form $h$ of degree $\ell-1<d$ cannot vanish
on that whole fiber, even when it is nonreduced.

The inclusion $\nu_*\mathcal I_C\subset\mathcal S$ induces a canonical map
\[
 i_*\rho_*L_C=\nu_*\mathcal I_C/\nu_*\mathcal I_C^2
       \longrightarrow\mathcal S/\mathcal O_Y,
\]
because $\nu_*\mathcal I_C^2\subset\mathcal O_Y$ is already proved.
Multiplication by $\mathcal I_D$ sends $\nu_*\mathcal I_C$ into
$\nu_*\mathcal I_C^2$, so the map is $\mathcal O_D$-linear.  Subtracting
an $\mathcal O_Y$-lift of a section's restriction to $D$ proves
surjectivity, without entering the map's definition.  Its monic kernel
is $\ker\tau_h$ by \eqref{v7:mixed:jet}.  Both maps have line-bundle
targets on the integral smooth $D$, hence saturated kernels in
$\rho_*L_C$.  Each kernel is the intersection with its generic kernel;
generic agreement therefore proves agreement everywhere.  Thus
$\mathcal S/\mathcal O_Y\simeq\mathcal O_D(-2,0)$ on every
projective chart.

\textit{Non-splitting.}
Both layers are killed by $\mathcal I_D$.
For $\ell<d$, on the finite chart $t[u^c]=[tu^c]\ne0$ in $B/A$.
For $\ell=d$, the recovered generator $b_1=tH_1=t(u-h_1)$
lies in $\mathcal I_D$.  In the quotient identification
$A^{\mathrm{sn}}/A\simeq R_0$, its product with $[u^{d-1}]$ maps to
\[
 \tau_h\bigl((u-h_1)u^{d-1}\bigr)=s-h_1h_{d-1}\ne0.
\]
This is a nonzero polynomial in the universal boundary ring; no
resultant is inverted.
In either case $\mathcal I_D\mathcal Q\ne0$, so the extension cannot
split as an $\mathcal O_Y$-module extension.

\textit{The actual Hilbert polynomial.}
Twisting the two layers by \eqref{v7:mixed:boundary-polarization} gives
\[
 \chi(\mathcal Q(n))=(dn-1)
                 \binom{(d-1)n+\ell-1}{\ell-1}.
\]
On the normalization the Euler characteristic is
$(n+1)\binom{(d-1)n+\ell}{\ell}$, which proves
\eqref{v7:mixed:projective-hilbert}.  Removing only the gluing layer gives
the seminormalization formula with its stated pulled-back polarization.
\end{proof}

\Needspace{9\baselineskip}
\subsection{Residual collisions and the local type}
\label{r22:mixed:artin-type}

For $1\le\ell\le d$, compute type in the actual Artin quotient of
$A$, not its image in a fiber of $B$.  The two parameter rings require
different multiplication laws.

\begin{theorem}[The gcd formula for the local type]
\label{v7:mixed:type}
At a conductor closed point $y\in D$, let $F_s$ denote its binary
binomial of degree $d$, and let $h$ denote its residual binary form of
degree $\ell-1$.  Then
\begin{equation}\label{v7:mixed:gcd-type}
 \operatorname{type}\mathcal O_{Y,y}
          =1+\deg\gcd(h,F_s).
\end{equation}
Thus the non-Gorenstein locus has underlying set
$\{\Res(h,F_s)=0\}\subset D$.  No scheme structure on this locus is
inferred from the set equality.
\end{theorem}
\begin{proof}
On the monic chart set $E_y=k[u]/(u^d-s(y))$, $\tau=\tau_h$, and
$\mathcal R=\ker(\cdot h:E_y\to E_y)$.
The perfect unweighted coefficient pairing identifies $\mathcal R$
with the radical of $(p,q)\mapsto\tau(pq)$; its dimension is
$r_y=\deg\gcd(h,u^d-s(y))$.

\textit{The subcritical Artin algebra.}
For $\ell<d$, reduce \eqref{v7:mixed:free-basis} by the regular
parameters $\mathbf h-\mathbf h(y),t,w-s(y)$ of $R$.
Put $V=\ker\tau\subset E_y$ and $W=V/k1$.
The Artin algebra so obtained has vector space $k\oplus W\oplus kz$,
where $z$ is the class of $t^2u^c$.  Its multiplication is
\begin{equation}\label{v7:mixed:artin-product}
 v\,v'=\tau_h(vv')z\quad(v,v'\in W),\qquad zW=z^2=0.
\end{equation}
To check this directly, $b_k=tH_k\pmod{t^2B}$, and all terms in
$t^3B$ belong to $tA$ and disappear in the reduction.  Reduction of a
product $t^2H_iH_j$ in the triangular $H$-basis kills every component
except $H_c$; its coefficient is $\tau_h(H_iH_j)$.
The relation for $w$ reduces to $u^d=s(y)$ in this computation, as its
$t$-terms contribute only to $t^3B$.  This proves
\eqref{v7:mixed:artin-product}, not just a tangent-space identity.

\textit{The subcritical socle.}
Here $\tau(1)=0$ and $1\notin\mathcal R$, since $\deg h<d$
and $h\ne0$.  After quotienting the radical, $V$ is the
orthogonal complement of the nonzero isotropic vector $1$; quotienting
by $k1$ leaves a nondegenerate form.  Hence the radical on $W$ has
dimension $r_y$, and the socle in \eqref{v7:mixed:artin-product} has
dimension $1+r_y$.  This is the Cohen--Macaulay type.

\textit{Critical multiplication.}
At $\ell=d$, $\tau(1)=1$, so $V/k1$ is unavailable.
Reduce \eqref{v8:mixed:critical-basis} by
$\mathbf h-\mathbf h(y),w-s(y),\sigma$ of $R^\dagger$.
Set $V=\ker\tau$ and $U=E_y/k1$; then $E_y=k1\oplus V$.
The resulting Artin algebra has vector space
\[
 k\oplus V\oplus U\oplus k\omega,
 \qquad \omega=[t^3],
\]
where $V$ is represented by the $tH_j$ and $U$ by $t^2u^j$,
$1\le j<d$.  To specify the full product, take the unique representatives
$p,p'$ of degree less than $d$ for two elements of $V$, and divide
\[
 pp'=a+(u^d-s(y))b,\qquad \deg a<d,\quad\deg b\le d-2.
\]
Writing $v(p)$ for the $V$-class and $u([a])$ for the $U$-class, one has
\[
 v(p)v(p')=u([a])+(d-1)\tau(b)\omega.
\]
Indeed, multiplying the division identity by $t^2$ and using
$u^d-s(y)=(d-1)t$ gives $t^2a+(d-1)t^3b$; the constant part of
$t^2a$ vanishes, while $t^3b$ reduces to $\tau(b)\omega$ in the
triangular $H$-basis.  In particular the $U$-part is $[pp']$.
Products of a $V$-class and a $U$-class are
\begin{equation}\label{v8:mixed:critical-cross-pairing}
 p\cdot[q]=\tau(pq)\omega.
\end{equation}
Also $U^2=0$ and $\omega$ annihilates the maximal ideal.
These assertions follow directly from $u^d=s(y)+(d-1)t$:
$t^4B\subset\sigma A$ disappears, and
$t^3H_j=\sigma b_j$ disappears for $1\le j<d$.
The displayed $\omega$-term in $V\cdot V$ is retained below;
it need not vanish, even when its $U$-part does.

\textit{The critical socle.}
If a socle element has $V$-part $p$,
\eqref{v8:mixed:critical-cross-pairing} forces $p\in\mathcal R$.
Its products with $V$ have zero $U$-part, so $pV\subset k1$.
Since $hp=0$ and $h\ne0$ in $E_y$, this implies $pV=0$.
If $p\ne0$, the equality $E_y=k1\oplus V$ then gives
$pa=\tau(a)p$ for every $a\in E_y$.  Multiplying twice shows that
$\tau$ is an algebra homomorphism, hence evaluation at a root
$\alpha$ of $u^d-s(y)$.  The perfect coefficient pairing implies
\[
 h=\frac{u^d-s(y)}{u-\alpha},\qquad p\in kh.
\]
The condition $hp=0$ then requires $h(\alpha)=d\alpha^{d-1}=0$.
Thus $\alpha=s(y)=0$, $h=u^{d-1}$ and $p$ is a multiple of
$u^{d-1}$.  In this exceptional case its product with $tu$ has
nonzero $\omega$-part $(d-1)\omega$.  No $U$-part of the socle
candidate can cancel it, since $\tau(uq)=0$ for all $q$ when
$h=u^{d-1}$ and $s(y)=0$.  This contradiction proves $p=0$.

The socle therefore consists of $k\omega$ and the $U$-classes
orthogonal to $V$.  Under the form $(a,b)\mapsto\tau(ab)$ one has
$V^\perp=k1\oplus\mathcal R$, since $\tau(1)=1$.
Modulo $k1$ this has dimension $r_y$.
The socle dimension is consequently $1+r_y$, as asserted.
Both arguments allow $s(y)=0$ and a nonreduced $E_y$.

\textit{The projective boundary.}
The reciprocal and completed product decompositions of
Lemma~\ref{r35:mixed:completed-model} add only regular parameters; factors of
$h$ outside the relevant cluster are coprime to $F_s$.
Thus they preserve this type and identify the gcd with the binary gcd
in \eqref{v7:mixed:gcd-type} on the whole projective boundary.
\end{proof}

\begin{example}[The source and target conductor thickenings differ]
\label{v7:mixed:threefold}
For $(d,\ell)=(3,2)$ write $f=(x-u)(x-a)+t$ and change the generators to
\[
 w=u^3-2tu,\quad \alpha=t(u^2+au),\quad\beta=t^2u.
\]
Then
\[
 A=k[a,t,w,\alpha,\beta]
   =R\oplus R\alpha\oplus R\beta,\quad R=k[a,t,w],\quad B=k[a,t,u].
\]
The conductor is $t^2B$, whereas the reduced boundary ideal is
$I_D=(t,\alpha,\beta)$.  In the Artin reduction $a=t=w=0$, the classes
$\alpha,\beta$ are independent and all their products vanish.
Thus $\beta\notin I_D^2$, although $\beta\in(A:B)$, and
\[
 I_D^2\subsetneq\mathfrak c,\qquad
 \mathfrak c\mathcal O_N=\mathcal I_C^2.
\]
Along $D$, the type is two when $s=a^3$ and one otherwise.  The Hilbert
polynomial is $2n^3-n^2+3n+2$.
The notation ``the double conductor'' must therefore specify whether it
refers to the normalization or to the target.
\end{example}

\subsection{Hilbert series and the multiplicative-string deformation}
\label{r17:mixed:deformation}

With the root-weight grading, $\deg h_j=j$, $\deg t=\ell$, $\deg w=d$.
For $\ell<d$ the free basis gives
\begin{equation}\label{v7:mixed:weighted-hilbert}
 \Hilb_A(z)=
 \frac{1+\sum_{\substack{1\le k<d\\k\ne d-\ell}}z^{\ell+k}+z^{\ell+d}}
 {(1-z^d)\prod_{j=1}^\ell(1-z^j)}
 =\frac{1-z+z^{\ell+1}}{(1-z)\prod_{j=1}^\ell(1-z^j)}.
\end{equation}
At $\ell=d$, Theorem~\ref{v8:mixed:critical} instead gives
\begin{equation}\label{v8:mixed:critical-hilbert}
 \Hilb_A(z)=
 \frac{1+\sum_{j=1}^{d-1}(z^{d+j}+z^{2d+j})+z^{3d}}
 {(1-z^{2d})\prod_{j=1}^{d}(1-z^j)}
 =\frac{1-z+z^{d+1}}{(1-z)\prod_{j=1}^{d}(1-z^j)}.
\end{equation}
The common root-weight series is independent of $d$ for fixed $\ell\le d$;
it is not the projective Hilbert polynomial
\eqref{v7:mixed:projective-hilbert}.  It agrees with the Weil-generic
series for $(r,s)=(\ell,1)$ in
\cite[Theorem~3.2(ii)]{EtingofRains}, but that numerical
agreement does not identify a specialized algebra.

For $\ell=1$, Theorem~\ref{v7:mixed:type} shows that $Y$ is everywhere
Gorenstein, but \eqref{v7:mixed:global-layers} shows that it is nowhere
seminormal along $D$.  In this case the Hilbert polynomial is
$(d-1)n^2+2$, and the target conductor does equal $I_D^2$.
In coordinates $v$ and $\epsilon=v-u$, put
$q=uv^{d-1}$ and
$J=(\epsilon v^j:0\le j\le d-2)=I_D$.
Then $A=k[\epsilon,\epsilon v,\ldots,\epsilon v^{d-2},q]$,
and $B$ is free over $k[\epsilon,q]\subset A$ with basis
$1,v,\ldots,v^{d-1}$, since
$v^d-\epsilon v^{d-1}-q=0$.
As $d-1\le2(d-2)$, every $\epsilon^2v^j$ for $0\le j<d$
is a product of two generators of $J$.
Thus $\epsilon^2B\subseteq J^2\subseteq\epsilon^2B$.
The reciprocal chart gives the same equality; these charts
cover $C$, and off $D$ both ideals are the unit ideal.
Hence the target conductor is $I_D^2$ globally.

\begin{proposition}[A pointwise compatible relation map need not descend]
\label{v7:mixed:relation-obstruction}
For $\ell=1$, the normalized relation morphism
$(u,v)\mapsto\{u^d,v^d\}$ is constant on every geometric normalization
fiber and descends to $Y^{\mathrm{sn}}$, but not to $Y$.
\end{proposition}
\begin{proof}
Off $C$, normalization is an isomorphism.  On $C$ we have $u=v$, and the
fibers are determined by $u^d=v^d$, so the relation map factors through
$D$.  The seminormal pinching square gives descent to $Y^{\mathrm{sn}}$.
In the coordinates \eqref{v7:mixed:division}, $v=u-t$ and $h=1$.
If the map descended to $Y$, the coordinate $u^d+(u-t)^d$ on its monic
relation chart would belong to $A$.  Its zero-order value is $2s$ and
its first-order coefficient has $\tau_h$-value $-d$.
Equation \eqref{v7:mixed:jet} instead requires $-2(d-1)$.
These differ for $d\ge3$, proving the obstruction.
\end{proof}

Unlike Example~\ref{v6:spec:no-relation-descent}, the obstruction is
first-order despite compatible point fibers.
For $\ell=1$, Theorem~\ref{v7:linear:sumsets} identifies the transverse
singularities on the nonzero binomial boundary as the classical elliptic
$d$-fold points \cite[Definition~2.1]{Smyth}
(see also \cite[Section~1]{Stevens}); the endpoints still require
Theorem~\ref{v7:mixed:projective}.  For $(d,\ell)=(3,1)$, \eqref{v7:mixed:projective-hilbert}
gives $P_Y(n)=2n^2+2$, so the surface has degree four.

\textit{The multiplicative-string family.}
For the difference calculations below, set
\[
 a_j(\theta)=\sum_{a=1}^{d-1}\theta^{-aj},\qquad
 q_\theta=\theta^{-(d-1)},\qquad
 Q_j(\theta)=u^j+a_j(\theta)P_j(f).
\]
These are the power sums of one singleton and $\ell$ strings of length
$d-1$.  At $\theta=\zeta$ the Laurent polynomial gives the coefficients
in \eqref{v7:mixed:newton}; no quotient $0/0$ is used.
The root-string construction has direct precedents
\cite{EtingofRains,SergeevVeselovMR,KasataniMultiwheel};
Appendix~\ref{r17:mixed:dictionary} retains their exact parameter
comparisons.  The next steps separate boundary normality, lifting,
and finite-output generation.

\Needspace{10\baselineskip}
\subsection{The maximal range with a smooth reduced boundary}
\label{r23:mixed:boundary-range}

\begin{proposition}[The boundary normality threshold]
\label{v8:mixed:boundary-threshold}
For this proposition allow any $\ell\ge1$, with $d\ge3$ fixed.
Let $C=\{f(u)=0\}$ in the singleton--complement normalization,
and give $D_{d,\ell}=\nu(C)$ its reduced image structure.  Then
\[
 D_{d,\ell}\text{ is normal}
 \quad\Longleftrightarrow\quad \ell\le d
 \quad\Longleftrightarrow\quad
 D_{d,\ell}\simeq\PP^1\times\PP^{\ell-1}.
\]
Thus the range $1\le\ell\le d$ in the projective descent theorem
is maximal for smoothness of this reduced boundary image.
\end{proposition}
\begin{proof}
For $\ell=1$ the image is the binomial line.  Assume $\ell>1$.
On $C$ the output has one full mask and $\ell-1$ complement masks.
Its normalization is $\PP^1\times\PP^{\ell-1}$.
The two weighted degree sets in the monic normality criterion
\cite[Theorem~6.1]{LiProfileGeometry} are $\{d\}$ and
$\{1,\ldots,\ell-1\}$.  They are disjoint exactly when $\ell\le d$,
and all the required Fourier coefficients are nonzero.
The endpoint weighted-sum map is injective: equality of
$d\epsilon+(d-1)k$, with $\epsilon\in\{0,1\}$, implies
$d-1\mid\epsilon-\epsilon'$, hence $\epsilon=\epsilon'$ and $k=k'$.
The projective normality criterion \cite[Theorem~6.2]{LiProfileGeometry} proves sufficiency.  If $\ell>d$, the overlap at degree $d$ already violates
the necessary affine normality condition.
\end{proof}

At $\ell=d+1$, $h=F_s$ is zero in $k[u]/(u^d-s)$, so $\tau_h$
vanishes on that fiber.  This is failure of the surjectivity in
Theorem~\ref{v7:mixed:projective}, not the parameter change at $\ell=d$;
Theorems~\ref{v8r9:mixed:first-excluded} and~\ref{v8r11:mixed:projective-first}
treat the resulting ambient algebra and projective defects.

\subsection{Confluent difference descent and saturated specialization}

For $\ell\le d$, the already proved output algebras determine the
saturated specialization.  The directly specialized equation has a larger
kernel.  Let $\mathscr O=k[\theta,\theta^{-1}]_{(\theta-\zeta)}$, with uniformizer
$\pi=\theta/\zeta-1$, and retain $q_\theta,Q_j(\theta)$ from
Section~\ref{r17:mixed:deformation}.  In
\[
 \mathscr B=\mathscr O[u,v_1,\ldots,v_\ell]^{\mathfrak S_\ell}
             =\mathscr O[\mathbf h,t,u]
\]
put $\mathscr A=\mathscr O[Q_1(\theta),\ldots,Q_r(\theta)]$,
$r=1+(d-1)\ell$.  For a polynomial symmetric in the $v_i$, define
\[
 (\mathscr D_\theta P)(z;v_2,\ldots,v_\ell)
 =P(q_\theta z,\theta z,v_2,\ldots,v_\ell)
     -P(z,z,v_2,\ldots,v_\ell).
\]
Its target is $\mathscr O[z,v_2,\ldots,v_\ell]^{\mathfrak S_{\ell-1}}
=\mathscr O[\mathbf h,z]$.  Symmetry makes this one identity equivalent
to imposing the corresponding identity for every $v_i$.

\begin{proposition}[Saturated confluent specialization]
\label{v8r8:mixed:confluence}
For $d\ge3$ and $1\le\ell\le d$, let
$\mathscr K=\ker\mathscr D_\theta$.  Then
\[
 \mathscr A=\mathscr K,\qquad
 \mathscr A/\pi\mathscr A\xrightarrow{\ \sim\ }A\subset B.
\]
Each root-weight component is finite free over $\mathscr O$; in
particular $\mathscr A$ is a finitely generated flat $\mathscr O$-algebra.
However, the kernel after direct specialization is
\[
 \ker\mathscr D_\zeta=A^{\rm sn}=k[\mathbf h,u^d]+tB,
 \qquad (\ker\mathscr D_\zeta)/A\simeq R_0(-d)
\]
as graded $A$-modules, with $A$ acting on $R_0$ by restriction to $t=0$.
Among elements of
$\ker\mathscr D_\zeta$, those admitting a lift satisfying
$\mathscr D_\theta P=0$ modulo $\pi^2$ are already exactly $A$.
\end{proposition}
\begin{proof}
\textit{Direct specialization.}
The identity
\[
 q_\theta^j-1+a_j(\theta)(\theta^j-1)=0
\]
shows $\mathscr A\subset\mathscr K$; higher output Newton sums are
polynomials in the finite list used to define $\mathscr A$.
Restriction to $v_1=u$ is restriction to $t=0$, whose quotient is
$k[\mathbf h,u]$.  At $\theta=\zeta$ the two evaluated pairs are
$(\zeta z,\zeta z)$ and $(z,z)$, with the same residual polynomial $h$.
Thus $\ker\mathscr D_\zeta=k[\mathbf h,u^d]+tB$.

\textit{First-order coordinates.}\label{r23:mixed:first-variation}
Suppose $P+\pi P^{(1)}$ is a lift modulo $\pi^2$, with
$P=F_0+tF_1\pmod{t^2B}$.  Then $F_0=H(\mathbf h,u^d)$.
The following division identities hold for every $\ell\ge1$, without
removing collision loci.  With
$h(x)=\prod_{i=2}^{\ell}(x-v_i)$, the division coordinates of the shifted
pair $(q_\theta z,\theta z)$ are
\[
 \widetilde h(x)=h(x)+(q_\theta-\theta)z
       \frac{h(x)-h(q_\theta z)}{x-q_\theta z},\qquad
 \widetilde t=(q_\theta-\theta)z h(q_\theta z).
\]
The divided difference is a polynomial.  Write $h_0=1$ and
$\chi_i(u)=\sum_{a=0}^{i-1}h_a u^{i-1-a}$ for $1\le i<\ell$.
Since
$q_\theta=\zeta(1-(d-1)\pi)+O(\pi^2)$, the coefficient
$\mathscr D_1P$ of $\pi$ in $\mathscr D_\theta P$ is
\[
 \begin{split}
 \mathscr D_1P={}&-(d-1)\zeta z\,\partial_uF_0(\mathbf h,\zeta z)\\
 &-d\zeta z\sum_{i=1}^{\ell-1}\chi_i(\zeta z)
                         \partial_{h_i}F_0(\mathbf h,\zeta z)\\
 &-d\zeta z h(\zeta z)F_1(\mathbf h,\zeta z).
 \end{split}
\]
The derivatives use $(\mathbf h,u)$.  Now use $\ell\le d$: in
$\mathscr D_\zeta P^{(1)}+\mathscr D_1P=0$, sum over the substitutions
$z\mapsto\zeta^a z$, $0\le a<d$.
The cyclic sum of $\mathscr D_\zeta P^{(1)}$ telescopes.
The $h_i$-variation terms sum to zero because every power in
$z\chi_i(z)$ lies between $1$ and $i\le\ell-1<d$.
The remaining two terms give the polynomial identity
\[
 \sum_{a=0}^{d-1}(\mathscr D_1P)(\zeta^az)
 =-d^2z^d\bigl((d-1)\partial_sH+
                       \tau_h(F_1)\bigr)\big|_{s=z^d}.
\]
\textit{Using the actual algebra.}
Cancel $z^d$ in the universal domain and use
$R_0\hookrightarrow k[\mathbf h,z]$ to obtain \eqref{v7:mixed:jet}.
This does not localize away from $s=0$.
Theorems~\ref{v7:mixed:affine} and \ref{v8:mixed:critical} give $P\in A$.
Conversely, expressing an element of $A$ in the finite $Q_j$ and replacing
them by $Q_j(\theta)$ gives an exact lift in $\mathscr A$.

\textit{Saturation.}\label{r23:mixed:saturation}
In each root weight $n$, source and target of $\mathscr D_\theta$
are finite free $\mathscr O$-modules.  Its kernel $\mathscr K_n$ is
finite free and saturated in $\mathscr B_n$: if $\pi P$ belongs to
the kernel, the torsion-free target forces $P$ to belong to it.
Consequently reduction embeds $\mathscr K_n/\pi\mathscr K_n$ into $B_n$.
The preceding obstruction calculation and the generator lifts show that
its image is $A_n$.  Thus
$\mathscr K_n=\mathscr A_n+\pi\mathscr K_n$.
The module $\mathscr K_n/\mathscr A_n$ is finite over the local ring
$\mathscr O$, so Nakayama's lemma makes it zero.  This proves all
assertions about $\mathscr A$ and its embedded special fiber.
Finally $[tu^{d-\ell}]$ generates $A^{\rm sn}/A$ by the jet test,
including $[t]$ when $\ell=d$; its root weight is $d$.
This proves the last displayed graded formula.
\end{proof}

This affine root-graded specialization does not give a projective
compactification.  At $\ell=d+1$ the residual term with $i=d$ can survive
the cyclic sum; the next calculation retains all such terms without
assuming generation.

\subsection{The lifting kernel and a finite generation certificate}
\label{r23:mixed:kernel-certificate}

Now allow every $\ell\ge1$, with $d\ge3$ fixed.  Retain $B$, $R_0$
and the finite output algebra $A=k[Q_1,\ldots,Q_r]$; do not assume
$R_0\subset A$.  With $h_0=1$, define
\[
 \gamma_i(s)=\sum_{q=1}^{\lfloor i/d\rfloor}h_{i-qd}s^{q-1},\qquad
 \mathcal V=(d-1)\partial_s+
              \sum_{i=d}^{\ell-1}\gamma_i(s)\partial_{h_i}.
\]
An empty sum is zero.  On $S_0=R_0+tB\subset B$ put
\[
 \Phi\bigl(H(\mathbf h,u^d)+tF_1+O(t^2)\bigr)
       =\mathcal V H+\tau_h(F_1),\qquad M=\ker\Phi.
\]
Thus $\Phi$ is a derivation relative to restriction $S_0\to R_0$, not
an $R_0$-linear map on $S_0$.  At this stage $S_0$ is the directly
specialized kernel, not an asserted seminormalization of $A$.

\begin{proposition}[Cyclic obstruction and a finite generation certificate]
\label{v8r10:mixed:certificate}
For every $d\ge3$ and $\ell\ge1$, the actual output satisfies $A\subset M$,
and
\[
 \Hilb_M(T)=\mathcal H_\ell(T):=
       \frac{1-T+T^{\ell+1}}{(1-T)\prod_{j=1}^{\ell}(1-T^j)}.
\]
In the difference family above, reduction identifies
$\mathscr K/\pi\mathscr K$ with $M\subset B$, where
$\mathscr K=\ker\mathscr D_\theta$.  An element of
$\ker\mathscr D_\zeta=S_0$ admits a lift modulo $\pi^2$ exactly when
it belongs to $M$; every such element has an exact lift in $\mathscr K$.
These statements about the difference equation do not alone assert $A=M$.

Here is a finite sufficient certificate for that last equality.
Exhibit homogeneous $b_1,\ldots,b_{\ell+1}\in A$,
algebraically independent, and homogeneous
$a_1,\ldots,a_N\in A$ linearly independent over
$R=k[b_1,\ldots,b_{\ell+1}]$, such that
\[
 \frac{\sum_{i=1}^N T^{\deg a_i}}
      {\prod_{j=1}^{\ell+1}(1-T^{\deg b_j})}=\mathcal H_\ell(T).
\]
Membership of these finitely many elements must be certified in the
finite output generators, not inferred from $\Phi=0$.  Then
\[
 A=M=\bigoplus_i Ra_i,\qquad (A:B)=t^2B,
 \qquad \mathscr A=\mathscr K,\qquad
 \mathscr A/\pi\mathscr A\simeq A\subset B.
\]
In particular $A$ is Cohen--Macaulay and this is an embedded flat
specialization.  No smoothness of its reduced boundary is assumed.
\end{proposition}
\begin{proof}
\textit{The kernel and its Hilbert series.}
The operator $\mathcal V$ lowers root weight by $d$ and
$\mathcal V(s)=d-1$.  It is locally nilpotent because it strictly lowers
the positive root-weight grading.  It is surjective: for
$x=s/(d-1)$ and $H\in R_0$ the finite sum
\[
 \sum_{j\ge0}\frac{(-1)^jx^{j+1}}{(j+1)!}\mathcal V^jH
\]
has $\mathcal V$-derivative $H$.  Thus $\Phi$ is surjective.  Its degree is $-d$,
so $S_0=R_0\oplus tB$ as graded vector spaces gives
\[
 \Hilb_M=(1-T^d)\Hilb_{R_0}+T^\ell\Hilb_B=\mathcal H_\ell.
\]
\textit{The first-order obstruction.}
Use the universal coordinate calculation in
Section~\ref{r23:mixed:first-variation}, now retaining every $h_i$ term.
The cyclic sum of $u\chi_i(u)$ is
$d\sum_{q\ge1}h_{i-qd}s^q=d s\gamma_i(s)$.  Consequently, for
$F=H(\mathbf h,u^d)+tF_1+O(t^2)$,
\[
 \sum_{a=0}^{d-1}(\mathscr D_1F)(\zeta^a y)
       =-d^2y^d\bigl(\mathcal VH+\tau_h(F_1)\bigr)_{s=y^d}.
\]
The difference operator at $\pi=0$ is the cyclic difference on
restriction to $t=0$.  Its image is precisely the subspace of
$k[\mathbf h,y]$ with zero cyclic sum, since the characteristic is zero.
This proves the stated criterion for a lift modulo $\pi^2$.
Cancellation of $y^d$ takes place in the universal polynomial domain,
not on a punctured open set.  Each $Q_j$ has the exact lift
$Q_j(\theta)$; hence $A\subset M$.

\textit{Exact lifts before generation.}
To prove exact lifting without assuming generation, use the monomial
symmetric basis $m_\mu$ in the residual $\ell-1$ root variables.
In the target $\mathscr C=\mathscr O[y,v_2,\ldots,v_\ell]^{\mathfrak S_{\ell-1}}$
consider the $\mathscr O$-lattice with basis
\[
 \epsilon_p y^p m_\mu,\qquad p\ge1,\quad \operatorname{length}(\mu)\le\ell-1,
 \qquad
 \epsilon_p=\begin{cases}\pi,&d\mid p,\\1,&d\nmid p.\end{cases}
\]
Call it $\mathscr C^\sharp$; it is a module, not an asserted subalgebra.
Every coefficient of $\mathscr D_\theta$ at $y^p$ is divisible by
$\epsilon_p$, and its constant coefficient is zero.  This follows by
expanding a source basis element $u^a m_\lambda(v_1,\ldots,v_\ell)$:
the coefficient from a selected part $b$ is $q_\theta^a\theta^b-1$,
whose value at $\zeta$ is $\zeta^{a+b}-1$.
For $p\ge1$ and $\mu$ as above, we use the same partition notation
on both sides of the difference map: in the source $m_\mu$ is evaluated
in the $\ell$ variables $v_1,\ldots,v_\ell$, while in the target the
symbols $m_\nu$ are evaluated in the $\ell-1$ residual variables
$v_2,\ldots,v_\ell$.  Then
\[
 \mathscr D_\theta(u^p m_\mu)
  =(q_\theta^p-1)y^p m_\mu+
    \sum_{b\in\operatorname{parts}_+(\mu)}
       (q_\theta^p\theta^b-1)y^{p+b}m_{\mu\setminus(b)}.
\]
The sum is over distinct positive parts.  The other partitions have
strictly smaller length.  Relative to the displayed target basis the
diagonal coefficient $(q_\theta^p-1)/\epsilon_p$ is a unit: for
$d\mid p$ its residue is $-(d-1)p\ne0$.  Induction on partition length
therefore proves that $\mathscr D_\theta:\mathscr B\to\mathscr C^\sharp$
is surjective and split in each root weight.  On reduction, its equations
are the noninvariant coefficients of $\mathscr D_0$ and the invariant
coefficients of $\mathscr D_1$.  Their common kernel is exactly $M$ by
the cyclic calculation.  The split sequence proves
$\mathscr K/\pi\mathscr K\simeq M$ and exact lifting of every element
of $M$.

\textit{Generation by finite outputs.}
The certified inclusions $\bigoplus_iRa_i\subset A\subset M$ and the
rational-function identity force equality in every root weight.  Thus
$A=M$ is finite free over $R$, hence Cohen--Macaulay.

\textit{Conductor and embedded specialization.}
The conductor argument of Theorem~\ref{v7:mixed:affine} now applies.
First $t^2B\subset A$.  Reducing a conductor element modulo $t$ gives
$H\in R_0$ with $H\,k[\mathbf h,u]\subset R_0$; since the conductor of
$R_0=k[\mathbf h,u^d]$ in $k[\mathbf h,u]$ is zero, $H=0$.  Thus it lies
in $tB$.  If $tp$ is its first jet, multiplication by arbitrary
$q\in k[\mathbf h,u]$ and the jet condition give $\tau_h(pq)=0$.
Via $R_0[u]/(u^d-s)\simeq k[\mathbf h,u]$, the perfect unweighted pairing
then gives $hp=0$, hence $p=0$ and $(A:B)=t^2B$.
For specialization, the split kernel and exact generator lifts give
$\mathscr K_n=\mathscr A_n+\pi\mathscr K_n$.
The finite-module Nakayama and saturation argument in
Section~\ref{r23:mixed:saturation} then proves
$\mathscr A=\mathscr K$ and the embedded special fiber.
\end{proof}

The dimension argument is standard; compare
\cite[Lemmas~2.2--2.3 and the proof of Theorem~2.4]{EtingofRains}
and \cite[proof of Theorem~5.8]{SergeevVeselovMR}.
The specific inputs here are the scaled target $\mathscr C^\sharp$, its
split coefficient map, and a certificate of membership and independence
in the prescribed finite outputs.  The kernel calculation alone does
not supply that certificate.

\subsection{Uniform descent across the first nonnormal boundary}

For $\ell=d+1$ and every $d\ge3$, we now prove the finite certificate:
output membership, parameter independence and a free basis.
Theorem~\ref{v8r11:mixed:projective-first} supplies the separate
projective calculation.

\begin{theorem}[One-jet descent at $\ell=d+1$]
\label{v8r9:mixed:first-excluded}
Let $d\ge3$, $\ell=d+1$, and put
\[
 h=x^d+a_1x^{d-1}+\cdots+a_{d-1}x+c,\quad f=(x-u)h+t,
 \quad B=k[\mathbf a,c,t,u],\quad w=u^d-(d-1)c,\quad z=u^d+c.
\]
Define, for $1\le i<d$,
\[
 \begin{aligned}
 U_{d+i}&=a_i z-dtu^{i-1},& U_{2d}&=z^2-2dtu^{d-1},\\
 U_{2d+i}&=t(zu^{i-1}-a_i u^{d-1}),& U_{3d}&=z^3-3dztu^{d-1}.
 \end{aligned}
\]
For the actual degree-$d^2$ output algebra $A=k[Q_1,\ldots,Q_{d^2}]$,
the ring $R=k[\mathbf a,w,U_{d+1},U_{2d}]$ is polynomial and
\[
 A=R\ \oplus\!\bigoplus_{i=2}^{d-1}RU_{d+i}
       \ \oplus\!\bigoplus_{i=1}^{d-1}RU_{2d+i}
       \ \oplus RU_{3d}\ \oplus R(U_{d+2}U_{2d-1}).
\]
This is free of rank $2d$.  In particular $A$ is Cohen--Macaulay.
For $F=F_0+tF_1\pmod{t^2B}$, actual membership is equivalent to
\[
 F_0=H(\mathbf a,c,u^d),\qquad
 \tau_h(F_1)+\partial_cH+(d-1)\partial_sH=0,
 \qquad R_0=k[\mathbf a,c,s]=k[\mathbf a,w,z].
\]
Here $s=u^d$, and
$\tau_h(\sum_{j=0}^{d-1}p_ju^j)=\sum_{i=1}^{d-1}a_i p_{i-1}+zp_{d-1}$.
The reduced boundary algebra is
\[
 A_\partial=\operatorname{im}(A\to B/tB)
   =k[\mathbf a,w,a_1z,\ldots,a_{d-1}z,z^2,z^3].
\]
Its normalization is $R_0$, with conductor
$(a_1,\ldots,a_{d-1},z^2)R_0$ and quotient $R_0/A_\partial\simeq k[w]z$.
It is normal off $\mathbf a=z=0$ and has depth two at every closed point
of that line.  Nevertheless
\[
 (A:B)=t^2B,\qquad A^{\rm sn}=S_0=R_0+tB,\qquad S_0/A\simeq R_0(-d).
\]
The last formula is of graded $A$-modules, with action through
$A\to A_\partial\subset R_0$; it is not invertible over $A_\partial$
along the line.  The local type of $A$ at the total-zero point is $d+1$.
Moreover $\mathscr A=\ker\mathscr D_\theta$ and
$\mathscr A/\pi\mathscr A\simeq A\subset B$ for this whole layer.
\end{theorem}
\begin{proof}
Abbreviate $X_i=U_{d+i}$, $Y_0=U_{2d}$, $Z_i=U_{2d+i}$ and $W_0=U_{3d}$.

\textit{Finite-output membership.}
For a formal variable $v$ define
\[
 \begin{gathered}
 H_*(v)=1+\sum_{i=1}^{d-1}a_iv^i,\qquad
 C(v)=H_*(v)^{-1}=\sum_{j\ge0}C_jv^j,\qquad
 \Lambda_j=[v^j]\log H_*(v),\\
 L(v)=H_*(v)+cv^d+\frac{tv^{d+1}}{1-uv},\qquad
 \kappa_j=[v^j]\log L(v).
 \end{gathered}
\]
These are coefficient identities of formal series over polynomial rings.
The reciprocal of $f$ is $(1-uv)L(v)$, and the reciprocal output is
$(1-u^dv^d)\prod_{a=1}^{d-1}L(\zeta^a v)$.  Thus
\[
 Q_j=j\kappa_j\ (d\nmid j),\qquad
 Q_{kd}=d(u^d)^k-(d-1)kd\kappa_{kd}.
\]
The first $d-1$ sums recover $\mathbf a$ triangularly, and
$Q_d=dw-d(d-1)\Lambda_d$ recovers $w$.
Put
\[
 K(v)=\frac{w+\sum_{i=1}^{d-1}X_iv^i}{H_*(v)},\qquad
 A_*(v)=\frac{z-K(v)}d,\qquad
 B_*(v)=\frac{tu^{d-1}}{(1-uv)H_*(v)}.
\]
The exact identity $L/H_*=1+A_*v^d+B_*v^{2d}$ gives
\[
 \log L=\log H_*+A_*v^d+(B_*-A_*^2/2)v^{2d}
                    +(-A_*B_*+A_*^3/3)v^{3d}\pmod{v^{3d+1}}.
\]
For $1\le i<d$, $\kappa_{d+i}=\Lambda_{d+i}-K_i/d$.
Since $K_i=wC_i+\sum_{j=1}^iX_jC_{i-j}$, these equations recover
$X_1,\ldots,X_{d-1}$ successively.  At the next multiple of $d$,
\[
 Q_{2d}=w^2+(d-1)Y_0+2(d-1)K_d-2d(d-1)\Lambda_{2d},
\]
so $Y_0\in A$.

To verify the remaining recoveries, put $s=u^d$, $T=tu^{d-1}$ and
$S_i=\sum_{j=1}^iZ_jC_{i-j}$.  Coefficient extraction gives
\[
 (B_*)_i=T\sum_{j=0}^iu^jC_{i-j},\qquad
 K_i=wC_i+\sum_{j=1}^iX_jC_{i-j},\qquad
 S_i=tz\sum_{j=0}^{i-1}u^jC_{i-1-j}+TC_i.
\]
Substitute $ds=w+(d-1)z$, $2dT=z^2-Y_0$ and
$\sum_{j=1}^ia_jC_{i-j}=-C_i$.  For $1\le i<d$ these identities yield
\[
 (B_*)_i+\frac{zK_i}{d^2}
 =\frac{d-2}{d}S_i-\frac{Y_0C_i+w\sum_{j=1}^iX_jC_{i-j}}{d^2}.
\]
Since $[v^i]A_*^2=((K^2)_i-2zK_i)/d^2$, the coefficient of
$v^{2d+i}$ in the logarithm is
\[
 \begin{split}
 \kappa_{2d+i}={}&\Lambda_{2d+i}-\frac{K_{d+i}}d
   +\frac{d-2}{d}\sum_{j=1}^i Z_jC_{i-j}\\
 &-\frac{Y_0C_i+w\sum_{j=1}^iX_jC_{i-j}+(K^2)_i/2}{d^2}.
 \end{split}
\]
The nonzero coefficient $(d-2)/d$ recovers all $Z_i$ successively.
To handle weight $3d$, let
\[
 T_d=-\frac{C_dY_0}{d^2}+\sum_{i=1}^{d-1}C_{d-i}
             \left(\frac{d-2}{d}Z_i-\frac{wX_i}{d^2}\right),\quad
 L_0=\Lambda_{3d}-\frac{K_{2d}}d-\frac{(K^2)_d}{2d^2}+T_d.
\]
The same convolution, now at index $d$, gives
$(B_*)_d+zK_d/d^2=sT+T_d$.  Since $A_*(0)=c$ and $B_*(0)=T$,
the coefficient at $v^{3d}$ is therefore
\[
 \kappa_{3d}=L_0+(s-c)T+c^3/3.
\]
Substituting this into $Q_{3d}=ds^3-3d(d-1)\kappa_{3d}$ and using
$w=s-(d-1)c$, $z=s+c$, $Y_0=z^2-2dT$ gives
\[
 Q_{3d}=\frac{w^3+3(d-1)wY_0+(d-1)(d-2)W_0}{d}
                    -3d(d-1)L_0.
\]
All terms except $W_0$ have already been recovered, and its coefficient
is nonzero.  These recoveries use only $Q_j$ with $j\le3d\le d^2$.
Thus every proposed parameter and basis element belongs to the actual
finite output algebra; in particular, the displayed product basis element
does as well.  No membership is inferred merely from the jet condition.

\textit{Polynomial parameters.}
Use independent symbols $\mathbf a,w,\alpha,\beta$ for the proposed
parameters.  Eliminating
\[
 c=\frac{u^d-w}{d-1},\qquad
 z=\frac{du^d-w}{d-1},\qquad t=\frac{a_1z-\alpha}{d}
\]
from $\beta=z^2-2dtu^{d-1}$ gives the monic equation
\[
 \begin{split}
 0={}&u^{2d}-\frac{2(d-1)}d a_1u^{2d-1}-\frac{2w}d u^d\\
 &+\frac{2(d-1)}{d^2}\bigl(a_1w+(d-1)\alpha\bigr)u^{d-1}
       +\frac{w^2-(d-1)^2\beta}{d^2}.
 \end{split}
\]
The map from this monic quotient to $B$ sends
$\alpha$ to $X_1$ and $\beta$ to $Y_0$.
Its inverse sends $c$ and $t$ to the displayed expressions; the monic
relation is exactly $\beta=z^2-2dtu^{d-1}$ after substitution.
Thus the quotient is $B$, so $R$ embeds as a polynomial ring and $B$
is free over it on $1,u,\ldots,u^{2d-1}$.

\textit{Module independence and the Hilbert identity.}
For independence of the proposed $A$-basis, specialize its coefficient
matrix at $\mathbf a=w=0$, retaining $\alpha,\beta$.  Its columns are
\[
 1,\quad \alpha u^{i-1}\ (2\le i<d),\quad
 -\frac{\alpha}{d-1}u^{d+i-1}\ (1\le i<d),\quad
 \frac d{d-1}(\alpha u^{2d-1}+\beta u^d),\quad
 \alpha^2u^{d-1}.
\]
In the displayed order their determinant is
$-d\alpha^{2d}/(d-1)^d\ne0$.  The universal determinant is consequently
nonzero, proving independence over $R$, not independence in every
special fiber of $B$.
The basis weights are zero and all integers from $d+2$ through $3d+1$
except $2d$.  Their numerator satisfies
\[
 1+\sum_{j=d+2}^{3d+1}T^j-T^{2d}
       =\frac{(1-T+T^{d+2})(1-T^{2d})}{1-T}.
\]
Together with the parameter weights $1,\ldots,d+1,2d$, this verifies
the finite certificate of Proposition~\ref{v8r10:mixed:certificate}.
It proves $A=\ker\Phi$, the free basis, Cohen--Macaulayness, the conductor,
and saturated specialization.  Here
$\mathcal V=\partial_c+(d-1)\partial_s=d\partial_z$.

\textit{The actual reduced boundary.}\label{r23:mixed:boundary-algebra}
The coefficient-functional image is $J_\tau=(\mathbf a,z)R_0$, so a zero-order value $H$
occurs exactly when $\partial_zH\in J_\tau$.  Expansion in $z$ omits
precisely $k[w]z$, giving $A_\partial$ and $R_0/A_\partial=k[w]z$.
Its conductor is $J_\partial=(\mathbf a,z^2)R_0$, not $J_\tau$.
Adjoining $z$, whose square and cube are in $A_\partial$, gives its
subintegral normalization $R_0$.  At every closed point of the line,
the normalization module has depth $d+1$ and its quotient has depth one;
hence $A_\partial$ has depth two.

\textit{Seminormalization and its module action.}
The square-and-cube argument of Theorem~\ref{v7:mixed:affine} proves
$S_0$ seminormal, using
$k[\mathbf a,c,u]\cap k(\mathbf a,c,u^d)=R_0$ modulo $t$.
For subintegrality, note that $t^2B\subset A$ gives
\[
 A[1/t^2]=S_0[1/t^2]=B[1/t^2],\qquad
 (A/t^2A)_{\rm red}=A_\partial,\quad
 (S_0/t^2S_0)_{\rm red}=R_0.
\]
Indeed, if $b=tg\in A\cap tB$, then
$b^4=t^2(t^2g^4)\in t^2A$; the same argument applies
to $S_0$.  Since $A=S_0$ off $V(t^2)$ and their reduced map on $V(t^2)$ is the
subintegral extension $A_\partial\subset R_0$, the finite extension
$A\subset S_0$ is subintegral (the prime and residue-field conditions are
checked on these two pieces).  The surjective relative derivation $\Phi$ gives
$S_0/A\simeq R_0(-d)$ as $A$-modules, acting through
$A\to A_\partial\subset R_0$.
At a closed point of the line,
$R_0/\mathfrak m_{A_\partial}R_0\simeq k[z]/(z^2)$, so this module
requires two generators and is not invertible.

\textit{The Artin quotient at the total-zero point.}
First form $A_0=A/(\mathbf a,w)A$, which is free over
$k[\alpha,\beta]$ on the displayed basis.  Its map into
$B_0=B/(\mathbf a,w)B$ is injective by the specialized nonzero
determinant above.  Thus polynomial multiplication in $B_0$ may be
rewritten uniquely in the $A_0$-basis \emph{before} setting
$\alpha=\beta=0$ in $A_0$.
For $2\le i,j<d$, the columns $X_i=\alpha u^{i-1}$ give the exact
identities in $A_0$
\[
 X_iX_j=\alpha X_{i+j-1}\quad(i+j\le d),\qquad
 X_iX_j=X_2X_{d-1}\quad(i+j=d+1).
\]
The largest weight of an Artin basis class is $3d+1$.
All products involving $Z_i$, $W_0$ or $X_2X_{d-1}$, and the remaining
$X_iX_j$ with $i+j>d+1$, have larger weight and vanish.
Consequently the only nonzero products of positive-weight basis classes are
\[
 [U_{d+i}][U_{d+j}]=[U_{d+2}U_{2d-1}]
       \quad(2\le i,j<d,\ i+j=d+1).
\]
Their pairing on the $d-2$ classes $[X_i]$ is nonsingular anti-diagonal.
The socle consists of the $d-1$ classes $[Z_i]$, $[W_0]$, and
$[X_2X_{d-1}]$, so its dimension is $d+1$.
The Artin quotient is $A_0/(\alpha,\beta)A_0$; the embedding into $B_0$
was used only before this final quotient.
\end{proof}

\paragraph{Minimal output generators.}
Under the hypotheses of Theorem~\ref{v8r9:mixed:first-excluded}, let
\(\mathfrak m=A_+=\bigoplus_{n>0}A_n\) be the homogeneous maximal ideal of
the total-zero point.  Then
\[
 \Hilb_{\mathfrak m/\mathfrak m^2}(T)=T+T^2+\cdots+T^{3d}.
\]
Consequently the embedding dimension at that point is $3d$,
$Q_1,\ldots,Q_{3d}$ form a minimum-cardinality homogeneous $k$-algebra
generating set, and every $k$-algebra generating set of $A$ has at least
$3d$ elements.  In particular, the least integer $M$ for which
$A=k[Q_1,\ldots,Q_M]$ is $M=3d$.  The finite free extension $R\subset A$ gives
$\dim A_{\mathfrak m}=\dim R_{R_+}=d+2$, so the embedding codimension at
the total-zero point is $2d-2$.

Indeed, the recoveries in the proof express every polynomial parameter and
every element of the displayed free $R$-basis in
$Q_1,\ldots,Q_{3d}$, so these outputs generate $A$.  Write
\[
 R_+=(\mathbf a,w,X_1,Y_0)R,\qquad I=R_+A,\qquad C=A/I,
 \qquad P=X_2X_{d-1}.
\]
The Artin multiplication table in the proof gives
\[
 C_+^2=k\cdot[P].
\]
Hence the classes of
\[
 X_2,\ldots,X_{d-1},\quad Z_1,\ldots,Z_{d-1},\quad W_0
\]
form a basis of $C_+/C_+^2$, of dimension
$(d-2)+(d-1)+1=2d-2$.

It remains to see that the $d+2$ polynomial parameters of $R$ remain
independent in $\mathfrak m/\mathfrak m^2$.  Let
$\rho:A\to R$ be the $R$-linear projection onto the scalar summand in the
free decomposition of Theorem~\ref{v8r9:mixed:first-excluded}.  Every
non-scalar basis element has weight at least $d+2$.  Therefore, in total
weight $n\le 2d$, a product of two positive-weight elements can contribute
to the scalar summand only through the product of their scalar parts; thus
\[
 \rho\bigl((\mathfrak m^2)_n\bigr)\subset (R_+^2)_n,
 \qquad n\le 2d.
\]
The parameter weights are $1,\ldots,d+1,2d$, and the corresponding
variables are not in $R_+^2$.  Their classes are consequently independent
in $\mathfrak m/\mathfrak m^2$.

The quotient map $A\to C$ now gives an exact sequence
\[
 0\longrightarrow (I+\mathfrak m^2)/\mathfrak m^2
 \longrightarrow \mathfrak m/\mathfrak m^2
 \longrightarrow C_+/C_+^2\longrightarrow0.
\]
Modulo $\mathfrak m^2$, the left term is spanned by the parameter classes,
so it has dimension $d+2$.  Their weights, together with the weights
$d+2,\ldots,2d-1$, $2d+1,\ldots,3d-1$, and $3d$ of the displayed basis of
$C_+/C_+^2$, are exactly $1,\ldots,3d$.  This proves the Hilbert series.
Since $Q_1,\ldots,Q_{3d}$ generate $A$, their cotangent classes form a
basis.  Finally, after subtracting constant values at the total-zero point,
the elements of any $k$-algebra generating set must span
$\mathfrak m/\mathfrak m^2$, which proves the asserted lower bound and the
sharp consecutive cutoff.

For $d=3$, the degree-nine model has last basis element $U_5^2$ and
total-zero type four.  No type formula at other points, or finite-output
generation for $\ell\ge d+2$, follows from this certificate.

The abstract boundary model is a classical conductor fiber product:
with $J_\partial=(\mathbf a,z^2)R_0$,
\[
 A_\partial=R_0\mathbin{\times}_{k[w,z]/(z^2)}k[w].
\]
This is the preimage of $k[w]$ in $R_0/J_\partial$
\cite[Lemma~1.3]{Ferrand}.  The calculation above identifies this
classical fiber product with the restriction of the actual finite-output
algebra, rather than just a candidate boundary.

\subsection{Projective descent over the first nonnormal boundary}

Globalizing Theorem~\ref{v8r9:mixed:first-excluded} requires three
boundary objects: $C$ in the normalization of $Y$, the boundary
normalization $\overline D$, and its reduced image $D$.

\begin{theorem}[The projective first nonnormal boundary]
\label{v8r11:mixed:projective-first}
Let $d\ge3$, $\ell=d+1$, and let $\nu:N=\PP^1\times\PP^{d+1}\to Y$
be the singleton--complement normalization.  Put
$C=\{f(u)=0\}\simeq\PP^1\times\PP^d$ and $D=\nu(C)_{\rm red}$.
There is a factorization
\[
 C\xrightarrow{\rho}\overline D=\PP^1\times\PP^d
       \xrightarrow{\beta}D\xrightarrow{i}Y,
 \qquad \rho(u,h)=(u^d,h),\qquad j=i\beta.
\]
Here $\beta$ is the finite subintegral normalization.  If $F_s$ denotes
the binary binomial represented by $s\in\PP^1$, set
\[
 Z=\{(s,h):[h]=[F_s]\}\subset\overline D,\qquad
 \Sigma=\beta(Z),\qquad \iota:\Sigma\hookrightarrow D.
\]
Then $Z\simeq\Sigma\simeq\PP^1$, the map to $Y$ is $[F_s]\mapsto[F_s^d]$,
and $\beta$ is an isomorphism off $\Sigma$.  The boundary has depth two
at each closed point of $\Sigma$, and
\[
 0\longrightarrow\mathcal O_D\longrightarrow\beta_*\mathcal O_{\overline D}
   \longrightarrow\iota_*\mathcal O_\Sigma(-2)\longrightarrow0.
\]
In contrast, $Y$ is smooth off $D$ and locally Cohen--Macaulay everywhere.
Its normalization conductor is
\[
 \mathfrak c\mathcal O_N=\mathcal I_C^2
             =\mathcal O_N(-2d-2,-2).
\]
Its seminormalization algebra and defect layers are
\[
 \mathcal S=\nu_*\mathcal O_N
       \mathop{\times}_{j_*\rho_*\mathcal O_C}j_*\mathcal O_{\overline D},
 \qquad
 \mathcal Q_{\rm glue}=j_*\mathcal O_{\overline D}(-1,0)^{\oplus(d-1)},
 \qquad
 \mathcal Q_{\rm inf}=j_*\mathcal L,
 \quad \mathcal L=\mathcal O_{\overline D}(-2,0).
\]
The extension \eqref{v7:descent:extension} does not split.  The rank-one
$\mathcal O_D$-module $\beta_*\mathcal L$, corresponding to
$\mathcal Q_{\rm inf}=i_*\beta_*\mathcal L$, is not invertible along
$\Sigma$ and requires two generators there.  More
precisely, with $L_C=\mathcal I_C/\mathcal I_C^2$, the canonical map
\[
 j_*\rho_*L_C\longrightarrow\mathcal Q_{\rm inf}
\]
has image $j_*(\mathcal I_Z\mathcal L)$ and cokernel
$(i\iota)_*\mathcal O_\Sigma(-2)$.

The actual polarization satisfies
$j^*\mathcal O_Y(1)=\mathcal O_{\overline D}(1,d-1)$ and
$\mathcal O_Y(1)|_\Sigma=\mathcal O_\Sigma(d)$.  With $m_n=(d-1)n$,
\[
 \begin{aligned}
 P_Y(n)&=(n+1)\binom{m_n+d+1}{d+1}-(dn-1)\binom{m_n+d}{d},\\
 P_{Y^{\rm sn}}(n)&=(n+1)\binom{m_n+d+1}{d+1}-(d-1)n\binom{m_n+d}{d},\\
 P_D(n)&=(n+1)\binom{m_n+d}{d}-(dn-1).
 \end{aligned}
\]
For $Y^{\rm sn}$ the line bundle is the pullback from $Y$, and for $D$
it is the restriction from $Y$.
\end{theorem}
\begin{proof}
\emph{The reduced boundary and its complete fibers.}
On $C$ the output consists of one full mask and $d$ complements.  Its
normalization and hyperplane pullback are $\overline D$ and
$\mathcal O_{\overline D}(1,d-1)$ by \cite[Theorem~5.4]{LiProfileGeometry}.
For each entire nonzero $\mu_d$-orbit, and separately at zero and
infinity, the output total $d\epsilon+(d-1)k$ uniquely recovers
$\epsilon\in\{0,1\}$ and $k$.  On a nonzero orbit the complement
multiplicities are recovered from $b_c=\epsilon+k-n_c$.
The reconstruction holds on every geometric fiber, so the finite surjection
$\beta$ is radicial.  In characteristic zero its residue extensions are
trivial; hence $\beta$ is subintegral (and it is birational by the
normalization input).  Its monic ring is the
actual $A_\partial$ of
Theorem~\ref{v8r9:mixed:first-excluded}; its only nonnormal locus is
$\mathbf a=z=0$, equivalently $h=F_s$.  The reciprocal chart gives the
same conclusion at infinity.

A point omitted by both charts has output roots at both zero and infinity.
Group the output by entire nonzero orbits and the two endpoints.  The group
containing the full mask then has $k<d$ complements; otherwise there would
be no root in any other group.  Its truncated full--complement profile is
smooth by \cite[Theorems~6.1--6.2]{LiProfileGeometry}; every other group is a smooth homogeneous
complement profile.  The companion-profile assertion of
Lemma~\ref{r35:mixed:completed-model} therefore makes $\beta$ an
isomorphism at the omitted points.  The local boundary calculation shows
that $D$ has depth two at every closed point of $\Sigma$.  On $Z$ the
output is $F_s^d$, whose coefficients give the degree-$d$ Veronese map of
the binomial line.  Thus $\beta|_Z$ is a closed immersion and identifies
$Z$ with $\Sigma$.

\emph{The ambient image and its conductor.}
Lemma~\ref{r35:mixed:completed-model} applies with $\ell=d+1$.
Insert Theorem~\ref{v7:mixed:affine}, \ref{v8:mixed:critical}, or
\ref{v8r9:mixed:first-excluded} for the mixed factor according as
$1\le m<d$, $m=d$, or $m=d+1$; $m=0$ is smooth.
There are $d+1-m$ additional regular parameters.  The annihilator and
faithfully flat descent argument in Theorem~\ref{v7:mixed:projective}
then gives the conductor, smoothness off $D$, and local
Cohen--Macaulayness.  The same models identify the displayed coherent
fiber-product algebra $\mathcal S$ with the seminormalization: on each
mixed factor it is $R_0+tB$, with the regular factors adjoined.

\emph{The normalization action on the infinitesimal module.}
\label{r23:mixed:normalization-action}
For $a\in A$ and $F\in S_0$, the relative Leibniz rule gives
$\Phi(aF)=a|_{t=0}\Phi(F)$, since $\Phi(a)=0$.
Thus $\mathcal I_D$ kills $\mathcal Q_{\rm inf}=\mathcal S/\mathcal O_Y$.
On the monic and reciprocal charts this is the torsion-free rank-one
normalization module $R_0$ over $A_\partial$, by
Theorem~\ref{v8r9:mixed:first-excluded}.  Off $\Sigma$ the same completed
models show that it is a line bundle on the smooth $D$.
Consequently its endomorphism algebra is $\beta_*\mathcal O_{\overline D}$:
on the two charts this is
$\operatorname{End}_{A_\partial}(R_0)=R_0$, since an endomorphism is
multiplication by a fraction and its value on $1$ belongs to $R_0$;
off $\Sigma$ the assertion is immediate.  These identifications glue in
the common function field.  Thus $\mathcal Q_{\rm inf}=j_*\mathcal M$
for a line bundle $\mathcal M$ on $\overline D$.  This constructs the
normalization action; it does not assume that $\mathcal S/\mathcal O_Y$
is an $\mathcal S$-module.

\emph{The reciprocal transition.}
We determine the line bundle, rather than just its affine rank.  Write
$\Phi_0=\partial_c+(d-1)\partial_s+\tau_h$ for the local quotient map,
with the derivative acting on the zero-order part.  On the overlap put
$u'=u^{-1}$ and $f'(x)=x^{d+1}f(x^{-1})/f(0)$, and divide
$f'=(x-u')h'+t'$.  At $t=0$ the transition and its first derivative are
\[
 s'=s^{-1},\quad a'_i=a_{d-i}/c,\quad c'=c^{-1},\qquad
 t'=-\frac{t}{u^{d+2}c}\pmod{t^2},\qquad
 \left.\partial_t h'_i\right|_{t=0}
       =\frac{a_{d-i}-cu^{-i}}{uc^2}\quad(1\le i\le d),
\]
where $a_0=1$ and $h'_d=c'$.  The already constructed normalization
action lets us compute the transition on one quotient generator:
$[z]/d=[s+c]/d$ generates $\mathcal M$ on the monic chart because
$\Phi_0(z)=d$, and $[z']/d=[s'+c']/d$ does so on the reciprocal chart.
Writing $z=s+c$, the first derivatives above give
\[
 \Phi_0(s')=-\frac{d-1}{s^2},\qquad
 \Phi_0(c')=-\frac1{c^2}
       +\frac{z}{sc^2}-\frac{z}{s^2c}=-\frac1{s^2}.
\]
Consequently $[z']=-s^{-2}[z]$ under that action.  Equivalently,
\[
 \Phi_0(P')=-s^{-2}\Phi_\infty(P')
\]
for every local function $P'$ of the primed seminormal algebra, with
the right side in unprimed boundary coordinates.  Omitting the
$t$-variation of $c'$ would incorrectly leave $-c^{-2}$.
The transition is that of $\mathcal O(-2,0)$, up to a constant unit sign.
The two opens require respectively a finite binomial parameter and
nonzero leading coefficient of $h$, or a nonzero binomial parameter
and nonzero constant coefficient of $h$.  They contain $Z$; their
complement is a union of codimension-two loci.  Both $\mathcal M$ and $\mathcal O(-2,0)$ are already
line bundles on the smooth $\overline D$, so their isomorphism extends
uniquely across that subset \cite[Tag~0EBJ]{Stacks}.  This proves
$\mathcal M=\mathcal L$, without assigning a line bundle to the nonnormal
$D$.

\emph{Why the conormal map is not surjective.}
The inclusion $\nu_*\mathcal I_C\subset\mathcal S$ gives the canonical
map $j_*\rho_*L_C\to\mathcal Q_{\rm inf}$, since
$\nu_*\mathcal I_C^2\subset\mathcal O_Y$.  On the monic chart it is
$tp\mapsto\tau_h(p)$.  Finite duality for the finite flat map $\rho$
\cite[Tags~0BUZ and 0BVE]{Stacks} identifies the corresponding global map
\[
 T:\rho_*L_C\longrightarrow\mathcal L
\]
with the evaluation section $h(u)$: indeed
\[
 L_C=\mathcal O_C(-d-2,-1),\quad
 \omega_\rho=\mathcal O_C(2d-2,0),\quad
 L_C^\vee\otimes\omega_\rho\otimes\rho^*\mathcal L
       =\mathcal O_C(d,1).
\]
The induced map is $\mathcal O_{\overline D}$-linear: scalar commutation
holds generically, and its target is torsion-free, so it holds everywhere.
The two maps agree on the dense monic chart and hence everywhere in the
torsion-free target.  Its image near $Z$ is
$(a_1,\ldots,a_{d-1},z)\mathcal L=\mathcal I_Z\mathcal L$, by the
explicit row for $\tau_h$.  Off $Z$ it is surjective: a nonzero degree-$d$
form vanishes on the entire length-$d$ binomial fiber exactly when it is
that binomial, even at a ramified endpoint.  This proves the claimed image.

\emph{The boundary's own defect.}
The cokernel identification is formal after the restriction maps and
actual conormal image have been determined.  With $J=tB\subset S_0$
and $J^2\subset A$,

\[
 \operatorname{coker}(J/J^2\longrightarrow S_0/A)
   =S_0/(A+J)\simeq R_0/A_\partial.
\]
Restriction $\mathcal S\to j_*\mathcal O_{\overline D}$ is surjective,
with kernel $\nu_*\mathcal I_C$, and the image of $\mathcal O_Y$ under
restriction is $i_*\mathcal O_D$.  Therefore the cokernel of the conormal
map is also $i_*(\beta_*\mathcal O_{\overline D}/\mathcal O_D)$.
It is $j_*(\mathcal L|_Z)= (i\iota)_*\mathcal O_\Sigma(-2)$,
proving the boundary sequence and its twist.
The monic and reciprocal charts cover $Z$, so the two-generator and
noninvertibility assertions follow from the $A_\partial$-module
calculation in Theorem~\ref{v8r9:mixed:first-excluded}.
Also $\nu_*\mathcal O_N/\mathcal S=
 j_*(\rho_*\mathcal O_C/\mathcal O_{\overline D})$;
the degree-$d$ power-map quotient is
$\mathcal O_{\overline D}(-1,0)^{\oplus(d-1)}$, giving $\mathcal Q_{\rm glue}$.

\emph{The actual conductor square and the nonsplit extension.}
The conductor square has source $2C$, not $C$.  For
$\mathcal K=\ker T$, its target fits into
\[
 0\longrightarrow j_*\mathcal K\longrightarrow
       \mathcal O_Y/\mathfrak c\longrightarrow i_*\mathcal O_D
       \longrightarrow0.
\]
On the monic chart its subalgebra in $B/t^2B$ consists of the pairs
$H+tp$ with $(\partial_c+(d-1)\partial_s)H+\tau_h(p)=0$.
The reciprocal and full-fiber models specify it on the other charts.
Together with
$\mathcal O_Y=\nu_*\mathcal O_N\times_{\nu_*\mathcal O_{2C}}
(\mathcal O_Y/\mathfrak c)$ this records the target subalgebra, not only
the source ideal; the fiber-product mechanism is classical
\cite[Lemma~1.3]{Ferrand}.  No equality $\mathfrak c=\mathcal I_D^2$
or splitting of this sequence is asserted.

Both layers are killed by $\mathcal I_D$, but their extension is not.
The finite-output generator
$U_{2d+1}=t(z-a_1u^{d-1})$ belongs to $\mathcal I_D$.
Its product with $[u^{d-1}]$ lies in the infinitesimal layer and maps under
$\Phi$ to
\[
 \tau_h\bigl((z-a_1u^{d-1})u^{d-1}\bigr)
       =z^2-sa_1a_{d-1}\ne0
\]
in the universal ring $R_0$.  Thus $\mathcal I_D\mathcal Q\ne0$,
excluding a splitting of \eqref{v7:descent:extension} without inverting
a resultant.  This witness proves global nonsplitting, not nonvanishing
of the displayed polynomial at every boundary point.

\emph{Actual polarization and Euler characteristics.}
The boundary output is $F_s$ times the complement output of $h$, so its
specified line bundle pulls back to $\mathcal O_{\overline D}(1,d-1)$.
On $Z$ this has degree $d$, consistently with $F_s^d$.
The two defect layers after twisting by $\mathcal O_Y(n)$ have Euler
characteristics
\[
 (d-1)n\binom{m_n+d}{d},\qquad
 (n-1)\binom{m_n+d}{d}.
\]
Subtracting their sum from
$\chi(N,\mathcal O_N(n,(d-1)n))$ proves $P_Y$; subtracting only the
first proves the stated polarized formula for $Y^{\rm sn}$.
The boundary sequence instead subtracts $\chi(\Sigma,\mathcal O_\Sigma(dn-2))
=dn-1$ from $\chi(\overline D,\mathcal O(n,(d-1)n))$, proving $P_D$.
The boundary quotient is subtracted separately for $P_D$, not again
for $P_Y$: it has already entered the identification of $\mathcal Q_{\rm inf}$.
All three formulas use the actual line bundles, not the root-weight series.
\end{proof}

The same formula for $P_Y$ continues through $\ell=d+1$; its
infinitesimal layer is then pushed forward from the boundary normalization.

The exact root-string and multi-wheel parameter dictionaries, together
with the integral free decomposition implied by the finite-output
certificate, are recorded in Appendix~\ref{r17:mixed:comparisons}.

\Needspace{14\baselineskip}
\section{Separations: what the descent data remember}\label{v7:sec:separations}

Over an algebraically closed characteristic-zero field, four pairs test
which data determine the image (Table~\ref{v7:table:separations}).
Polarizations are compared by their pullbacks to the normalization.

\subsection{Mixed degrees do not determine the normalization}
\label{r24:separation:degrees}

\begin{example}[Two quartic surfaces]\label{v6:cond:quartic}
Take $I=\{0,4,8\}$ and $r=4$.  Let $A=1100$ and $O=1010$ be the
necklaces with stabilizer orders one and two.  The profiles $AA$ and
$AO$ have the same expanded weights $(2,2)$, covering degree $16$,
and mixed degrees $(16,8,4)$, but their image normalizations are
$\PP^2$ and $\PP^1\times\PP^1$, respectively.

For $AA$, fix $i^2=-1$.  The roots
$\alpha,i\alpha,\beta,i\beta$ give the normalization map
\[
 [P:Q:R]\longmapsto
 [P^2:-(1+i)PQ:iQ^2:-i(1+i)QR:-R^2].
\]
After diagonal coordinate rescaling this is
$[P^2:PQ:Q^2:QR:R^2]$, whose image is the complete intersection
$z_0z_2=z_1^2$, $z_2z_4=z_3^2$ in $\PP^4$.
Indeed, on $z_2\ne0$ this complete intersection is an
affine plane, while its locus $z_2=0$ is supported on a
line.  It therefore has one generically reduced
component and no embedded components, so is integral
and equals the image.  On the monic chart the algebra is
\[
 k[p,pq,q^2]\simeq k[p,b,c]/(b^2-p^2c),
\]
with normalization $B=k[p,q]$.  Its conductor is $pB$, and the
normalization line $Q=0$ maps two-to-one onto the nonnormal line.

For $AO$, write the factor as
\[
 (X-\alpha Z)(X-i\alpha Z)(X^2-zZ^2).
\]
The first two monic coefficients are $-(1+i)\alpha$ and
$i\alpha^2-z$; hence they recover $\alpha,z$ polynomially and this
chart is smooth.  The projective normalization is
$\PP^1\times\PP^1$ with pullbacks
\[
 \eta=2H_1+H_2,\qquad \xi=4H_1+2H_2.
\]
For $AA$ the pullbacks are $2H,4H$, giving the same three mixed degrees;
the normalization Picard ranks differ.  The projective $AO$ image is
nevertheless nonnormal: assigning its two weight-two types to zero and
infinity in either order gives the same factor $X^2Z^2$.
\end{example}

\subsection{Polarized normalizations do not determine the image}
\label{r24:separation:pullbacks}

\begin{example}[Same normalization pullbacks, different boundary]
\label{v6:cond:degree-nine}
Take $I=\{0,6,12\}$ and $r=6$.  Compare two copies of
$S_A=\{0,1,2\}$ with two copies of $S_B=\{0,1,3\}$.  Both are
aperiodic necklaces of weight three.  Thus both profiles have
\[
 \widetilde Y=\PP^2,\quad(\eta,\xi)=(3H,6H),\quad
 D=36,\quad\delta=(36,18,9),
\]
and the same normalized relation cover.  In $B=k[p,q]$,
Corollary~\ref{v6:cond:halfmask} with $t=3$ gives
\[
 A_A=k[p,pq,pq^2,q^3]=k[q^3]+pB,\qquad(A_A:B)=pB.
\]
Its conductor line maps with degree three onto a line.  If $\zeta$ is
a primitive sixth root, then for $S_B$ one has
$\lambda_{B,1}=\zeta$ and $\lambda_{B,2}=1+\zeta$, both nonzero.
Corollary~\ref{v6:norm:homogeneous-threshold} therefore gives $Y_B\simeq\PP^2$
for the entire projective image.  Consequently
\[
 P_{Y_A}(n)=\binom{3n+2}{2}-2n,
 \qquad P_{Y_B}(n)=\binom{3n+2}{2}.
\]
The general hyperplane sections have arithmetic genera three and one.
Thus the common cover and pullbacks do not determine the image or its
Hilbert polynomial.
\end{example}

\Needspace{10\baselineskip}
\subsection{Point fibers do not determine conductor thickness}
\label{r24:separation:thickness}

\begin{theorem}[A pair with different conductor thickness]
\label{v6:cond:thick}
For $I=\{0,12,24\}$ and $r=8$, consider the profiles consisting of two
copies of
\[
 S_0=\{0,1,3,4\},\qquad S_1=\{0,2,3,5\}.
\]
Write $Y_0,Y_1$ for their reduced projective images.  They have the
same normalized relation cover, image normalization $\PP^2$,
pullbacks $(\eta,\xi)=(4H,12H)$, and mixed degrees $(144,48,16)$.
On the monic normalization set $B=k[p,q]=k[p,t]$, where $t=q-p^2$.
The image algebras and conductors are
\begin{align}
 A_0&=k[q^2]+pB,&(A_0:B)&=pB,\label{v6:cond:thin-ring}\\
 A_1&=k[p,t^2,p^3t,pt^3],&(A_1:B)&=(p^3,pt^2)B.
 \label{v6:cond:thick-ring}
\end{align}
The seminormalization of $Y_1$ is $Y_0$, compatibly with the fixed
normalization $\PP^2$.  Their Hilbert polynomials are
\begin{equation}\label{v6:cond:thick-hilbert}
 P_{Y_0}(n)=\binom{4n+2}{2}-2n,
 \qquad P_{Y_1}(n)=\binom{4n+2}{2}-2n-4.
\end{equation}
The first surface is seminormal and locally Cohen--Macaulay.  The
second fails both properties at exactly two points.
\end{theorem}

\begin{proof}
\emph{Common normalization data.}
The mask polynomials are $(1+T)(1+T^3)$ and $(1+T^2)(1+T^3)$.
Both masks are aperiodic: a stabilizer order divides four, whereas
translations by three and six preserve neither mask.
The product-normalization theorem \cite[Theorem~5.4]{LiProfileGeometry} gives the normalization and polarizations.
In the common root coordinates the normalized relation map is
$\{u,v\}\mapsto\{u^{12},v^{12}\}$, and the stated mixed degrees follow.
For $S_0$ all odd Fourier coefficients are nonzero and the supported
even indices generate $\Gamma=2\mathbb N$.  This proves
\eqref{v6:cond:thin-ring} by Corollary~\ref{v6:cond:surface}.

\emph{The actual finite-output algebra.}
For $S_1$, the supported indices between one and eight are
$1,4,5,7,8$.  Since $\lambda_3(S_1)=0$, we use
Proposition~\ref{v6:spec:monic-algebra} directly: the image algebra is
generated by these $P_j$.
Writing $p=u+v$, $q=uv$ and $t=q-p^2$, the recurrence
$P_j=pP_{j-1}-qP_{j-2}$, with $P_0=2$ and $P_1=p$, gives
\begin{align*}
 P_4&=-p^4+2t^2,&
 P_5&=p^5+5p^3t+5pt^2,\\
 P_7&=p^7-7p^3t^2-7pt^3,&
 P_8&=-p^8-8p^6t-16p^4t^2-8p^2t^3+2t^4.
\end{align*}
After recovering $p$, these identities recover $t^2$, then $p^3t$,
then $pt^3$, in that order.  The expression for $P_8$ belongs to the
same ring, proving both inclusions in the image-algebra equality.

\emph{Conductor, quotient action, and seminormalization.}
Put $s=t^2$ and
$R=k[p,s]$, $I=(p^3,ps)R$.  The even--odd decomposition in $t$ gives
\[
 B=R\oplus tR,\qquad A_1=R\oplus tI,\qquad
 A_0=R\oplus ptR.
\]
The last expression equals $k[q^2]+pB$, since
$t^2-q^2=-2p^2q+p^4\in pB$.  Equivalently, $p^at^b\in A_1$ for all
$a\ge0$ if $b$ is even, $a\ge3$ if $b=1$, and $a\ge1$ if $b\ge3$ is odd.
For $r+tj\in A_1$, multiplication by the $R$-basis element $t$ of
$B$ stays in $A_1$ exactly when $r\in I$.  Consequently
\[
 (A_1:B)=I\oplus tI=IB=(p^3,pt^2)B.
\]
Thus the conductor is the indicated ideal of $B$.

The finite infinitesimal quotient is the cyclic module
\[
 M=A_0/A_1=tpR/tI\simeq R/(p^2,s)\,[pt].
\]
Thus it has length two, with basis $[pt],[p^2t]$ and nonzero action
$p[pt]=[p^2t]$.  The odd generators $p^3t,pt^3$ act by zero on $M$.
Moreover $A_0=A_1[pt]$ and $(pt)^2,(pt)^3\in A_1$; adjoining such an
element is subintegral (on every residue field its square and cube determine
its unique value).  Since $A_0$ is seminormal by
Corollary~\ref{v6:cond:surface}, it is the seminormalization of $A_1$.

\emph{The projective comparison.}
Use normalization coordinates
$[P:Q:R_2]$ for $PX^2-QXZ+R_2Z^2$; the monic coordinates are
$p=Q/P$, $q=R_2/P$.  On the reciprocal chart,
\[
 p'=p/q,\qquad q'=q^{-1},\qquad t'=t/q^2.
\]
The reciprocal Fourier pattern gives $A_1'\subset A_0'$ with another
length-two quotient.  On the overlap, the last output coefficient is a
nonzero multiple of $q^4$ and annihilates $M$, so the algebras agree
after inverting it.  The displayed reciprocal change also identifies
the conductor ideals: $(p')$ and $((p')^3,p'(t')^2)$ become $(p)$ and
$(p^3,pt^2)$ because $q$ is a unit.  The sole normalization point
outside these charts is $[0:1:0]$.
Its output is $[X^4Z^4]$, with the unique input allocation of one zero
and one infinite root.  Hensel separation and the nonzero coefficients
$\lambda_1,\lambda_{-1}$ recover both local parameters, by
Theorem~\ref{v6:spec:completed-model}.  The normalization is an
isomorphism near this point on either image.
The reduced common graph has finite projections; the projection to
$Y_0$ is an isomorphism on both charts and at the remaining completion,
hence globally by faithful flatness.  The other projection gives a
finite subintegral morphism $\varphi:Y_0\to Y_1$ lifting the identity
of $\PP^2$, with quotient length four.
Both images have the same point fibers, including the generically
two-point fibers on the line $Q=0$; only the two endpoints carry $M$.

\emph{Actual polarization and local depth.}
Use the line bundle $\mathcal L=\varphi^*\mathcal O_{Y_1}(1)$ on $Y_0$.
Its normalization pullback is $\mathcal O_{\PP^2}(4)$.
The reduced boundary $D_0\subset Y_0$ is $\PP^1$, covered with degree
two by the normalization line $Q=0$, and therefore
$\deg(\mathcal L|_{D_0})=2$.
Twisting the conductor sequence for $Y_0$ by $\mathcal L^n$ gives
\[
 \chi(\PP^2,\mathcal O(4n))-\chi(Y_0,\mathcal L^n)
 =(4n+1)-(2n+1)=2n.
\]
The same computation applies separately to $\mathcal O_{Y_0}(1)$.
The length-four quotient of $\varphi$ is unchanged by twisting,
which proves \eqref{v6:cond:thick-hilbert}.
At the finite endpoint, with maximal ideal
$\mathfrak m=(p,s,p^3t,pt^3)$ of $A_1$, the ring $A_0$
is free over $R$ with basis $1,pt$, whereas $M$ is nonzero
of finite length.  Since $(p,s)A_1$ is $\mathfrak m$-primary,
the depth lemma over $R_{(p,s)}$ gives
$\operatorname{depth}(A_1)_{\mathfrak m}=1$.
The reciprocal endpoint is identical.  Elsewhere $\varphi$ is an
isomorphism, proving the stated local Cohen--Macaulay and seminormality
claims.

\emph{The nonsplit defect extension.}
As $R$-modules, the terms of \eqref{v7:descent:extension} are
\[
 0\longrightarrow tpR/tI\longrightarrow tR/tI
   \longrightarrow tR/tpR\longrightarrow0.
\]
The odd part $tI$ of $A_1$ kills all three terms.
Multiplication by $p^2$ kills both outer terms but sends $[t]$ in the
middle term to $[p^2t]\ne0$.  The extension is therefore nonsplit as
an $A_1$-module extension, and hence also nonsplit globally on $Y_1$.
\end{proof}

\subsection{A fixed source conductor does not determine the image}
\label{r24:separation:fixed-conductor}

The final pair also fixes the source conductor and varies the target
boundary subring.

\begin{theorem}[A separation with a fixed source conductor]
\label{v7:cond:same-ideal}
In $C_{24}$ take
\[
 T=\{0,1,2,3,4,6,7,8,9,10\},\qquad
 S=\{0,1,3,5,6,7,9,11,15,21\},
\]
and repeat either mask twice.  These profiles have $I=\{0,24,48\}$,
$r=20$, normalization $\PP^2$, pullback polarizations $(10H,24H)$,
and mixed degrees $(576,240,100)$.
They have identical normalization point fibers, the same
seminormalization, and the same conductor ideal in the fixed
normalization.  Nevertheless their images are different, even as
abstract varieties.  On $B=k[p,q]$ the rings are
\[
 A_T=k[q^2]+pB,\qquad A_S=k[q^4,q^6]+pB,
 \qquad(A_T:B)=(A_S:B)=pB.
\]
Their projective Hilbert polynomials are
\[
 P_{Y_T}(n)=50n^2+10n+1,\qquad
 P_{Y_S}(n)=50n^2+10n-1.
\]
\end{theorem}
\begin{proof}
\emph{Spectra and actual image rings.}
The mask polynomials factor as
\[
 M_T=(1+z^6)(1+z+z^2+z^3+z^4),\qquad
 M_S=(1+z^6)(1+z+z^3+z^5+z^{15}).
\]
Both masks have weight ten and trivial stabilizer: its order divides
$\gcd(10,24)=2$, and translation by twelve preserves neither mask.
The factor $1+z^6$ is $\Phi_4\Phi_{12}$.
The other factor of $M_T$ is $\Phi_5$, so gives no further divisor of
$z^{24}-1$.  For the other factor $P$ of $M_S$, direct remainders for
$q=2,3,4,6,8,12,24$ are respectively
\[
 -3,\quad2,\quad1,\quad0,\quad1,\quad
 3z^3+1,\quad z^5+z+1.
\]
Only the $q=6$ remainder is zero.  Hence every odd Fourier coefficient
is nonzero, and the supported even indices divided by two generate
$\Gamma_T=2\mathbb N$ and $\Gamma_S=\langle4,6\rangle$.
Corollary~\ref{v6:cond:surface} gives the rings, conductors, and Hilbert
polynomials.

\emph{Global descent and intrinsic inequivalence.}
Here $A_T=A_S[q^2]$ and $(q^2)^2,(q^2)^3\in A_S$, so this finite
extension is subintegral; Theorem~\ref{v7:cond:semigroup} identifies
$A_T=A_S^{\mathrm{sn}}$.  Its quotient is the one-dimensional space $k\cdot[q^2]$,
killed by $pB,q^4,q^6$.
The two reciprocal charts therefore each contribute length one.
The actual last output coefficient is a nonzero multiple of $q^{10}$.
On the overlap,
\[
 A_T[q^{-10}]=A_S[q^{-10}]
   =k[q^2,q^{-2}]+p\,k[p,q,q^{-1}].
\]
Indeed $q^2=q^{12}q^{-10}$ and $q^{-2}=q^8q^{-10}$ already belong
to the localized boundary subring.  The changes $p'=p/q$, $q'=q^{-1}$
therefore identify both actual rings and conductor ideals.
The remaining point of $\PP^2$ has one input at each endpoint and is
the unique preimage of a smooth point by
Theorem~\ref{v7:cond:semigroup}.  The common reduced graph is finite
over both images; its projection to $Y_T$ is an isomorphism on these
charts and at the remaining completion, hence globally.  It gives a
finite subintegral morphism $Y_T\to Y_S$, equating all point fibers.
The source conductor is the same line ideal.  The first image is
seminormal and locally Cohen--Macaulay; the second fails both
properties at its two endpoints.
These intrinsic properties rule out an abstract isomorphism.
\end{proof}

\begin{table}[H]
\centering\small
\caption{Which data fail to determine the image?  Each row is a separate pair;
all pullbacks are on its normalization.  In the last two rows, fibers and
source conductors are compared under the displayed normalization identification.}
\label{v7:table:separations}
\begin{tabular}{p{.15\linewidth}p{.39\linewidth}p{.38\linewidth}}
\toprule
Pair & Data held fixed & Distinguishing datum \\
\midrule
Example~\ref{v6:cond:quartic} & Weights, covering degree, and mixed degrees.
 & Normalization: $\PP^2$ versus $\PP^1\times\PP^1$.\\[3pt]
Example~\ref{v6:cond:degree-nine} & Normalized relation cover and both hyperplane pullbacks.
 & Nonnormal image versus the smooth $\PP^2$.\\[3pt]
Theorem~\ref{v6:cond:thick} & Normalized cover, both pullbacks, all point fibers, and the seminormalization.
 & Source ideals $(p)B$ versus $(p^3,pt^2)B$.\\[3pt]
Theorem~\ref{v7:cond:same-ideal} & Normalization, both pullbacks, point fibers, seminormalization, and conductor $(p)B$.
 & Boundary subrings $k[q^2]$ versus $k[q^4,q^6]$.\\
\bottomrule
\end{tabular}
\end{table}

The mixed examples distinguish a further two issues: pointwise relation
compatibility need not satisfy first-order descent
(Proposition~\ref{v7:mixed:relation-obstruction}), and the source conductor
$\mathcal I_C^2$ need not be the target ideal $\mathcal I_D^2$
(Example~\ref{v7:mixed:threefold}).  Together with the nonsplit extension
in Theorem~\ref{v6:cond:thick}, these show why the boundary subalgebra and
its module action must be retained.

\section{Scope, range boundaries, and outlook}\label{sec:outlook-b}
The results distinguish the normalization-side markings, the
point-fiber/seminormal level, the actual descended image algebra, and the
conductor/defect level.  Fourier zero spectra or Fourier lines compare actual
profile images under specified normalization or relation markings.  For
homogeneous profiles, bounded integer relations determine point fibers and
therefore seminormalization.  The explicit calculations in the homogeneous
and singleton--complement families go further and determine the actual
subalgebras, conductors, defect modules, and their actions.

The homogeneous semigroup calculation applies under the spectral hypothesis
of Theorem \ref{v7:cond:semigroup}; the two-element mask is computed
through the first singular interval \(L\le m<2L\) in Theorem
\ref{v7:cond:first-interval}.  Singleton--complement descent is proved
for \(d\ge3\) and \(1\le\ell\le d+1\).  The all-boundary gcd type formula is
proved for \(\ell\le d\); at \(\ell=d+1\) the finite-output generation,
projective conductor, and the type at the total-zero point are determined.
The lifting-kernel calculation applies for arbitrary \(\ell\), but finite
output generation beyond the stated range is not asserted.

Natural next targets are the homogeneous replacement problem for \(m\ge2L\),
singleton--complement generation and projective descent for \(\ell\ge d+2\),
and profiles with several singleton factors.  More conceptually, the
separating examples ask for a minimal functorial package of normalization-side
data that determines the actual image.  Point fibers, seminormalization, and
the source conductor are still too coarse, whereas the whole-fiber completed
image algebra is sufficient in the families treated here.
\appendix
\section{Squarefree mixed boundaries and Fourier line arrangements}
\label{v7:app:linear}

Over an algebraically closed field of characteristic zero, squarefree
nonzero outputs give exact linear configurations and transverse line
conductors.  Section~\ref{v7:sec:mixed-trace} handles repeated roots
and endpoints.

\subsection{The entire squarefree fiber}
\label{r27:linear:fiber}

\begin{proposition}[Exact linear configurations on a squarefree stratum]
\label{v7:linear:exact}
Let a mixed profile have nonempty aperiodic pairwise rotation-distinct
types, multiplicities $n_\nu$, and output degree $R$.
Suppose $G\in Y$ has $R$ distinct roots in $\Gm$.
Let $\mathcal F_G$ be all feasible allocations into the prescribed mask
blocks, modulo permutations within each type and satisfying every
global multiplicity $n_\nu$; these are the normalization preimages.
In relative output-root coordinates $z_1,\ldots,z_R$, each allocation
$\alpha$ defines a linear subspace $L_\alpha$ by equating the coordinates
belonging to the same block.  Then
\[
 \widehat{\mathcal O}_{Y,G}
   \simeq k[[z_1,\ldots,z_R]]/\bigcap_{\alpha\in\mathcal F_G}I(L_\alpha).
\]
This is an equality of completed image rings, not only a tangent-cone
calculation.
\end{proposition}
\begin{proof}
Squarefreeness makes ordered output roots \emph{\'etale} coordinates
and prevents coincident same-type inputs.  Under $u\mapsto u(1+t)$,
each block root is exactly $\zeta^au(1+t)$, so its relative coordinate
is $t$ and the branch is the closed immersion of $L_\alpha$.
The full-fiber completion of Theorems~\ref{v6:spec:completed-model}
and \ref{v6:spec:mixed-classification} intersects these kernels over
all feasible allocations, retaining every global type count.
\end{proof}

Proposition~\ref{v6:spec:immersive-conductor} applies to these immersive
branches.  Globally, use absolute root coordinates
$x_{\nu,i,a}=\zeta^a u_{\nu,i}$ to define $L$, and let
$U=\bigcup_{\sigma\in S_R}\sigma L$ be reduced.
Exactness of invariants in characteristic zero makes $k[U]^{S_R}$
the monic image algebra, generated by the restricted elementary
symmetric functions.  This invariant-arrangement presentation
\cite{BrooknerCorwinEtingofSam} retains the equations; at repeated roots
the \emph{\'etale} local model above no longer applies.

\subsection{Complementary masks and transverse conductors}
\label{r27:linear:transverse}

\begin{theorem}[Fourier sumsets and transverse conductors]
\label{v7:linear:sumsets}
Let $S$ be a nonempty proper aperiodic subset of $C_d$, $d\ge3$, and let
$T=C_d\setminus S$ not be rotation-equivalent to $S$.
Use $S,T$ once each.  At $G=X^d-v^dZ^d$ with $v\ne0$,
split off the common smooth factor along the binomial boundary and put
\[
 E=\{j\in C_d\setminus\{0\}:\lambda_{S,j}\ne0\}.
\]
The transverse completed ring is the completion at the origin of the
reduced union of the $d$ lines
\[
 z_a\longmapsto(\zeta^{aj}z_a)_{j\in E},\qquad a\in C_d.
\]
Write $R_E$ for its graded ring and $\widetilde R_E=\prod_{a\in C_d}k[z_a]$
for its normalization.  Let $hE$ be the sumset of exactly $h$ elements
of $E$, with $0E=\{0\}$, and let $c_E=\min\{h:hE=C_d\}$.
Then
\begin{align*}
 \dim_k(R_E)_h&=|hE|, &2\le c_E&\le d-|E|+1,\\
 \mathfrak c\widetilde R_E&=
       \prod_{a\in C_d}z_a^{c_E}k[z_a],&
 \delta(R_E)&=\sum_{h\ge0}(d-|hE|).
\end{align*}
\end{theorem}
\begin{proof}
\emph{The full fiber and the smooth factor.}
A preimage selects $S+a$ and $T+b$.  Complementarity gives $S+a=S+b$;
aperiodicity forces $a=b$, so there are exactly $d$ preimages.
Their planes in Proposition~\ref{v7:linear:exact} are spanned by
$\mathbf 1$ and $1_{S+a}$.  Splitting the common diagonal from all planes
gives an actual power-series factor, not merely a tangent direction.
Fourier transformation and nonzero diagonal rescaling of the remaining
coordinates give the stated lines.  For two different shifts, proportional
centered indicators force a period of $S$ or an equal-weight complementary
rotation.  The hypotheses exclude both, so the lines are distinct.

\emph{Hilbert function and saturation.}
Degree-$h$ monomials restrict to the independent Fourier characters
with frequencies in $hE$, giving $\dim_k(R_E)_h=|hE|$.
Choose $j_0\in E$ and set $E_0=E-j_0$.  Then $hE=hj_0+hE_0$;
only the translated sets $hE_0$ are necessarily nested.
The set $E_0$ generates $C_d$, or its nontrivial annihilator would
identify two line directions.  Since $0\in E_0$, growth cannot stop
before $C_d$: a stationary sumset is invariant under all its generating
translations.  Starting at $|E|$ proves $c_E\le d-|E|+1$;
$0\notin E$ gives $c_E\ge2$.

\emph{Conductor and quotient length.}
A nonzero homogeneous value supported on one branch has every Fourier
coefficient nonzero, so it belongs to the image exactly when $hE=C_d$.
This equality persists in every higher degree.  Multiplication by the
normalization idempotents tests each branch separately, proving the
conductor formula.  The degree-$h$ normalization quotient has dimension
$d-|hE|$, including $h=0$; summing gives $\delta(R_E)$.
\end{proof}

For $E=C_d\setminus\{0\}$, in particular for a singleton, the Hilbert
function is
$1,d-1,d,d,\ldots$, the conductor order is two, and $\delta=d$.
These classical elliptic $d$-fold points \cite[Definition~2.1]{Smyth}
(see also \cite[Section~1]{Stevens}) give the
squarefree transverse model for $\ell=1$ in Section~\ref{v7:sec:mixed-trace}.
For $d=6$, $S=\{0,1\}$ instead has $E=\{1,2,4,5\}$ and $2E=C_6$:
the same conductor order but $\delta=7$.
Neither computation determines the entire projective conductor of an
arbitrary complementary pair, including repeated roots or endpoints.

\section{Restriction algebras and integral specialization}

\label{r17:mixed:comparisons}

This appendix records the parameter dictionaries and integral consequences
of Section~\ref{v7:sec:mixed-trace}.  The finite-output generation and
projective descent proofs are independent of these comparisons.
Write $\mathcal O,\mathcal A,\mathcal B,\mathcal D_\theta,\mathcal K$
for the deformation objects denoted by
$\mathscr O,\mathscr A,\mathscr B,\mathscr D_\theta,\mathscr K$ above,
and put $K=k(\theta)$.

\subsection{Parameter dictionaries and the scope of the precedents}
\label{r17:mixed:dictionary}

Retain $q_\theta,Q_j(\theta)$ from
Section~\ref{r17:mixed:deformation}.  For arbitrary $\ell\ge1$,
$a_j(\theta)=(q_\theta^j-1)/(1-\theta^j)$ wherever the quotient is
defined.  In the convention of Etingof--Rains the parameters are
$(r,s,q,t)=(\ell,1,q_\theta,\theta)$, with scalar rescaling one.  The
result recorded there as Proposition~3.5(i), due to Sergeev--Veselov,
identifies the deformed power-sum algebra itself with the displayed
difference kernel whenever $q,t$ are not roots of unity and
$q^m\ne t^n$ for all $m,n\ge1$.  In our family the latter equality would
mean $\theta^{(d-1)m+n}=1$; hence the generic nonresonant kernel
description is a direct precedent (in particular at a transcendental
complex value of $\theta$).  It does not determine the saturated
$\mathcal O$-lattice at $\theta=\zeta$, its embedded special fiber, or the
projective conductor and defect modules computed below.
In the convention of Sergeev--Veselov
\cite[Section~5, equations~(11)--(12), arXiv version]{SergeevVeselovMR}, use
\[
 (n_{\rm SV},m_{\rm SV})=(1,\ell),\qquad
 x=u,\quad y_i=\theta^{-1}v_i,\qquad
 (q_{\rm SV},t_{\rm SV})=(\theta^{-(d-1)},\theta^{-1}).
\]
Their deformed power sum is exactly $Q_j(\theta)$.  The entire curve is
special, since $q_{\rm SV}=t_{\rm SV}^{d-1}$; their remark after
Corollary~5.7 gives the alphabet
\[
 (u;\ \theta^{-1}v_i,\theta^{-2}v_i,\ldots,
             \theta^{-(d-1)}v_i)_{1\le i\le\ell}.
\]
For $d=3,\ell=2$ its generic image is the five-variable double-shifted-diagonal
restriction of \cite[Section~2, equation~(1) and Theorem~2.2]{KasataniMiwaSergeevVeselov}.

There is an exact multi-wheel description over $K=k(\theta)$.
In \cite[Definition~6.1, arXiv version]{KasataniMultiwheel}, take
\[
 (n_{\rm K},k_{\rm K},r_{\rm K},m_{\rm K})
   =(1+(d-1)\ell,d-2,2,\ell),\qquad
 (t_{\rm K},q_{\rm K})=(\theta,\theta^{-(d-1)}).
\]
All wheel shifts vanish, and a string starting at $\theta^{-(d-1)}v_i$
is the displayed string in reverse order.  For the Laurent vanishing
ideal $I^{(d-2,2)}_\ell$, restriction of symmetric polynomials satisfies
\[
 \ker\operatorname{res}_\theta
   =I^{(d-2,2)}_\ell\cap K[X_1,\ldots,X_{n_{\rm K}}]^{\mathfrak S_{n_{\rm K}}},
 \qquad \operatorname{im}\operatorname{res}_\theta
   =K[Q_1(\theta),\ldots,Q_{n_{\rm K}}(\theta)].
\]
Symmetry removes the order of the strings, and their nonzero parameters
are dense, so contraction from Laurent to polynomial functions gives
the same vanishing kernel.  Newton's identities identify the image in
the fixed normalization.  The single-wheel precursor is
\cite[Theorem~2.4]{FJMMShifted}.
The quotient-basis statement in \cite[Theorem~6.5]{KasataniMultiwheel}
assumes $n_{\rm K}=k_{\rm K}+1$, not $n_{\rm K}=d^2$.
For composition series, \cite[Section~2.1.4]{EnomotoDAHA} imposes nonroot
parameters, while \cite[Section~5.2, Theorem~5.10]{EtingofStoica} removes
its half-integer exception.  Thus $d=3$, corresponding to $1/(d-1)=1/2$,
is included at that level.  These assertions do not identify the image
lattice over $\mathcal O=k[\theta^{\pm1}]_{(\theta-\zeta)}$.

At $\theta=\zeta$, the Etingof--Rains parameters are $(\zeta,\zeta)$ and
the Sergeev--Veselov parameters are $(\zeta,\zeta^{-1})$.
All specialized coefficients $a_j$ are nonzero, so that convention does
not exclude this family.  The Weil-generic theorem and
\cite[Remark~2.6]{EtingofRains}, which assumes roots distinct up to
inversion, do not give the specialization.
For $1\le\ell\le d$, direct specialization of the difference condition in
\cite[Proposition~3.5(i)]{EtingofRains} gives $A^{\rm sn}$ rather than
$A$.\footnote{Here we specialize the displayed difference equation in
\cite[Proposition~3.5(i), arXiv and accepted-manuscript versions]{EtingofRains}.
Its generation assertion is not invoked at root-of-unity parameters;
generation of the prescribed image is proved by the finite-output
certificates in this paper.}
Proposition~\ref{v8r8:mixed:confluence}, using
Theorems~\ref{v7:mixed:affine} and~\ref{v8:mixed:critical}, instead
identifies the saturated specialization with $A$; its extra compatibility
is \eqref{v7:mixed:jet}.

Root-of-unity enlargement of a wheel vanishing space occurs already in
\cite[Section~3.1]{FJMMTUnity}.  Its $(r_{\rm F},k_{\rm F})=(2,d-1)$
case uses a full $d$-cycle, not disjoint $(d-1)$-strings.
The one-resonance basis theorem of
\cite[Section~2, equation~(2.10)]{FJMMShifted}, the non-special generation
theorem \cite[Theorem~5.8]{SergeevVeselovMR}, and the nonroot operator
restrictions of \cite[Sections~3.2--3.4]{HLNR} do not supply the additional
specialization $\theta^d=1$.  Preservation of a difference algebra must
be separated from its generation.

The constant-coefficient limit $\theta=1$ gives $(x-u)f(x)^{d-1}$,
whose Cohen--Macaulayness is covered by
\cite[Theorem~2.1 and Proposition~2.2]{BrooknerCorwinEtingofSam}.
Its normalization is bijective on geometric points: the multiplicity
$(d-1)k+\epsilon$, $\epsilon\in\{0,1\}$, recovers both input counts at
each output root.  The cyclic image has $d$ preimages at a general point
of $D$, so the two affine images are not even abstractly isomorphic.
Nor is the cyclic algebra the rational quasi-invariant algebra of
\cite[Section~4]{SergeevVeselovCMS}: $Q_1=u-\sum_i v_i$ fails
$(\partial_{v_i}-(d-1)\partial_u)Q_1=0$ on $u=v_i$.
These constructions, the generic difference-kernel description, and
their general specialization mechanisms are precedents.  The additional
claims established here concern the saturated root-of-unity image lattice,
its embedded special fiber, and the projective conductor and defect data.

\paragraph{Ambient regularization and the prescribed image.}
Here $(q,t)=(q_\theta,\theta)$ denotes the standard Macdonald parameters,
not the division coordinate $t=f(u)$.  Root-of-unity Macdonald theory
is a relevant precedent, not excluded by our parameter choice.  Cherednik's Section~9.1 and Main
Theorem~9.1(i)--(ii) allow roots of unity and construct flat limits and
triangular bases of little generalized eigenspaces
\cite{CherednikNonsemisimple}.
Those ambient spaces must be distinguished from the image of a specified
restriction along the resonant curve $qt^{d-1}=1$.
Similarly, the modified bases of
\cite[Definition~3.3 and Theorem~3.14]{KasataniModified} are constructed
over the fraction field of a single specialization divisor, whereas
$\theta=\zeta$ imposes an additional resonance.
The displayed Uglov path $(q,t)=(\zeta p,\zeta p^\gamma)$ uses $\gamma>0$
\cite[Section~4.4]{UglovRootLimit}; our path formally gives
$\gamma=-1/(d-1)$.
Invariant-ideal and operator-restriction results also precede this work:
\cite[Theorem~3.3 and Remark~3.4]{FeiginSilantyev} identifies its type-$A$
ideals with the $r=2$ multi-wheel ideals, under the nonroot setting of
its Section~2.
A regular ambient basis or an invariant restriction kernel does not
itself identify the saturated image lattice.  The embedded specialization
used here is proved through the actual finite outputs; the following computation
records its integral free decomposition.  No exclusion of all possible
modified-basis descriptions is claimed.

\subsection{The integral image supplied by the finite-output certificate}
\label{r17:mixed:integral-lift}

Fix $\ell=d+1$ and $N=d^2$.  Write $X_i=U_{d+i}$,
$Y_0=U_{2d}$, $Z_i=U_{2d+i}$, and $W_0=U_{3d}$ as in the proof of
Theorem~\ref{v8r9:mixed:first-excluded}, retaining $t=f(u)$.
That certificate has the following all-weight consequence.  Let $\mathbf b=(\mathbf a,w,X_1,Y_0)$ and
\[
 \mathcal E=\{1\}\cup\{X_i:2\le i<d\}
  \cup\{Z_i:1\le i<d\}\cup\{W_0,X_2X_{d-1}\}.
\]
For each entry $b$ of $\mathbf b$ and each $e\in\mathcal E$, fix its
rational polynomial expression in $Q_1,\ldots,Q_{3d}$ supplied by the
successive recoveries in that theorem.  Define $\widehat b$ and
$\widehat e$ by substituting $Q_j(\theta)$ into those expressions.
These are lifts through the outputs, not the unchanged formulas in the
normalization coordinates.  With
$\widehat{\mathcal R}=\mathcal O[\widehat{\mathbf b}]$, one has
\[
 \mathcal A=\bigoplus_{e\in\mathcal E}
       \widehat{\mathcal R}\,\widehat e
   =\ker\mathcal D_\theta
   =\mathcal B\cap(\mathcal A\otimes_{\mathcal O}K).
\]
Since every displayed parameter and basis lift is expressed through
$Q_1(\theta),\ldots,Q_{3d}(\theta)$, one also has
\[
 \mathcal A=\mathcal O[Q_1(\theta),\ldots,Q_{3d}(\theta)].
\]
These $3d$ lifted outputs also have minimum cardinality as an
$\mathcal O$-algebra generating set of $\mathcal A$.  Indeed, any generating
set with fewer than $3d$ elements would reduce modulo $\pi$ to a
$k$-algebra generating set of $A$ with fewer than $3d$ elements, contradicting
the minimal-output calculation following
Theorem~\ref{v8r9:mixed:first-excluded}.  In particular, the consecutive
initial cutoff is sharp over the specified DVR $\mathcal O$.  This does not
assert the same minimum number of generators, or the same minimum cutoff,
for the generic fiber over $K$.
The parameter algebra is polynomial and the displayed module is free of
rank $2d$.  To verify this without assuming a generic basis theorem,
fix a root weight $n$.  The monomials
$\widehat{\mathbf b}^{\alpha}\widehat e$ of that weight reduce to the
basis of $M_n$ established by that certificate.  They are therefore independent over
$\mathcal O$: normalize a putative relation by the smallest valuation
of its coefficients and reduce modulo $\pi$.
Their span $L_n$ lies in $\mathcal A_n\subset\mathcal K_n$, and
Proposition~\ref{v8r10:mixed:certificate} identifies
$\mathcal K_n/\pi\mathcal K_n$ with $M_n$.  Hence
$\mathcal K_n=L_n+\pi\mathcal K_n$, so Nakayama gives
$L_n=\mathcal A_n=\mathcal K_n$.
The same reduction argument proves parameter independence.
Finally the kernel of a map to a torsion-free module is saturated,
which gives the intersection equality.  This is a consequence of the finite certificate, not an independent
generation proof from multi-wheel composition factors; its range remains
$\ell=d+1$.

\section*{Funding and declarations}

\noindent\textbf{Funding.}
This work was supported by the National Natural Science Foundation of China
(grant numbers 12571003 and 12501006) and the Basic and Applied Basic Research
Foundation of Guangdong Province (grant number 2024A1515010589).

\smallskip
\noindent\textbf{Declaration of competing interest.}
The author declares no known competing financial interests or personal
relationships that could have appeared to influence the work reported in this
paper.

\end{document}